\documentclass[amsart,fleqn]{amsart}

\usepackage{latexsym}
\usepackage{amsmath}
\usepackage{amssymb}
\usepackage{epsfig}
\usepackage{amsfonts}

\begin{document}
\newtheorem{definition}{Definition}[section]
\newtheorem{definition and examples}[definition]{Definition and examples}
\newtheorem{theorem}[definition]{Theorem}
\newtheorem{lemma}[definition]{Lemma}
\newtheorem{proposition}[definition]{Proposition}
\newtheorem{examples}[definition]{Examples}
\newtheorem{corollary}[definition]{Corollary}
\newtheorem{comments}[definition]{Comments}
\def\square{\Box}
\newtheorem{remark}[definition]{Remark}
\newtheorem{remarks}[definition]{Remarks}
\newtheorem{exercise}[definition]{Exercise}
\newtheorem{example}[definition]{Example}
\newtheorem{observation}[definition]{Observation}
\newtheorem{observations}[definition]{Observations}
\newtheorem{algorithm}[definition]{Algorithm}
\newtheorem{criterion}[definition]{Criterion}
\newtheorem{algcrit}[definition]{Algorithm and criterion}

\newenvironment{prf}[1]{\trivlist
\item[\hskip \labelsep{\it
#1.\hspace*{.3em}}]}{~\hspace{\fill}~$\square$\endtrivlist}

\title[Moduli for differential equations and Painlev\'e equations]{Moduli  spaces for linear differential equations and the Painlev\'e equations}

\dedicatory{Dedicated to Professor Bernard Malgrange on the occasion of his 80th birthday}
\author{Marius van der Put and Masa-Hiko Saito}
\thanks{Partly supported by Grant-in Aid
for Scientific Research (S-19104002) the Ministry of
Education, Science and Culture, Japan }
\address{Institute of Mathematics and Computing Science, University of Groningen, P.O. Box 407, 9700 AK Groningen, The Netherlands}
\address{Department of Mathematics, Graduate School of Science,  
Kobe University, Kobe, Rokko, 657-8501, Japan}
\email{mvdput@math.rug.nl}
\email{mhsaito@math.kobe-u.ac.jp}
\maketitle

\section*{Introduction}
The theme of this paper is a systematic construction of the ten isomonodromic
 families of  connections of rank two on $\mathbb{P}^1$ inducing 
Painlev\'e equations. They are obtained by considering the complex  analytic 
Riemann--Hilbert morphism $RH:\mathcal{M}
\rightarrow \mathcal{R}$ from a moduli space $\mathcal{M}$  of connections to 
a categorical moduli space of analytic data (i.e., ordinary monodromy, Stokes 
matrices and links) $\mathcal{R}$, here called the 
{\it monodromy space}. The fibres of $RH$ are the isomonodromic families. One 
requires that an isomonodromic family  has dimension 1, since it is then
 (locally) parametrized by one variable $t$ and some combination $q(t)$ of the 
entries of the connection is a potential solution of some second order 
Painlev\'e equation. This condition leads to the ten families. Our method
extends the work of Jimbo, Miwa and Ueno \cite{JMU,JM}, since we allow
all possible irregular singularities including ramification and resonance.\\

There is a natural  morphism 
$\mathcal{R}\rightarrow  \mathcal{P}$, where $\mathcal{P}$ is a parameter 
space build from traces of matrices. For each of the ten families,  
the morphism $\mathcal{R}\rightarrow \mathcal{P}$ turns out to be a family of 
affine cubic surfaces with three lines at infinity.  We will give explicit equations
 of $\mathcal{R}$ for these ten families in \S 2 and \S 3. The equation for Painlev\'e VI is classical \cite{FK65, Iw2}, and the equations for the other nine families seem to be new.

Since many aspects of  the well known family with four regular singularities 
leading to Painlev\'e VI, has been studied in great detail (\cite{Boalch, IIS1, IIS2, IISA, Iw2}), our 
emphasis will be on families with irregular singularities. Of the nine 
families with irregular singularities, six are again classical 
\cite{JM,FN}. The three remaining ones were also recently discovered in 
\cite{OO,OKSO}. The corresponding Painlev\'e equations appear already in
\cite{Sakai} from the viewpoint of the Okamoto--Painlev\'e pairs.   \\

 The moduli spaces of connections $\mathcal{M}$ are strongly related to the 
 Okamoto--Painlev\'e pairs $(S,Y)$ of non fibre type \cite{Sakai,STT02}. 
 The latter determine uniquely each type of Painlev\'e equation \cite{STT02}.
 We will give a brief description of this relation.\\
  
  The surface $S$ is the blow up of  nine points (allowing for infinitely near
 points) in $\mathbb{P}^2$ 
(or equivalently eight points in the Hirzebruch surface  $\Sigma _2$) 
which lie on an  effective anti-canonical divisor  of $\mathbb{P}^2$ or $\Sigma _2$. Let $Y$ be the 
unique effective anti-canonical divisor of $S$. The Okamoto--Painlev\'e 
condition on $Y$ implies that $Y$ has the same configuration as a degenerate 
elliptic curve in the  classification  by Kodaira--N\'eron \cite{O1, Sakai, STT02}.

The configuration of the irreducible components of $Y$ for the 
Okamoto--Painlev\'e pairs are given by the eight extended
   Dynkin diagrams
 \[ \tilde{D}_4,\tilde{D}_5,\tilde{D}_6,\tilde{D}_7 ,\tilde{D}_8,\ \tilde{E}_6,
 \ \tilde{E}_7,\ \tilde{E}_8. \]
 Each Dynkin diagram gives rise to a (uni)versal global family provided with a unique vector 
 field which induces a Painlev\'e equation \cite{STT02}.
 
 One conjectures that a relative compactification of each of the ten families 
of connections $\pi: \mathcal{M}\rightarrow T\times \Lambda$ with parameter
space $T\times \Lambda$, is isomorphic to one of the above global (uni)versal
families. As a consequence of this conjecture, the fibres of $\pi$ are the 
complement $S \setminus Y$  for a certain  Okamoto--Painlev\'e pair $(S, Y)$  
of the given type. The conjecture has been proven for Okamoto--Painlev\'e pairs
 of type $\tilde{D}_4$, which corresponds to Painlev\'e VI. 
(For the construction of the moduli spaces of linear connections with only 
regular singularities and the Riemann-Hilbert correspondence for these, see 
\cite{IIS1, IIS2, Ina06}).

 There is an explicit analytic morphism  $\Lambda \rightarrow \mathcal{P}$,
 given by exponentials, which is compatible with the Riemann--Hilbert morphism 
 $RH:\mathcal{M}\rightarrow \mathcal{R}$.  
 The monodromy space $\mathcal{R}\rightarrow \mathcal{P}$ can have, as fibre,
  a singular (affine) cubic surface $\mathcal{R}_p$.  As is conjectured and 
proved for the PVI case, the Riemann--Hilbert morphism yields 
  an analytic resolution $(S\setminus Y)\rightarrow \mathcal{R}_p$.
   The singular points of type
 $A_1,A_2, A_3, D_4$ on the cubic surface yield 1, 2, 3 and 4 exceptional 
curves on $S\setminus Y$ which are called Riccati curves.
 The latter are related to Riccati solutions of the corresponding Painlev\'e 
equation. Since the Riccati curves on the Okamoto--Painlev\'e pairs are known 
(\cite{STe02}), one can now link each of the ten monodromy spaces 
$\mathcal{R}$ to an  Okamoto--Painlev\'e pair and an extended Dynkin diagram 
(see Table \ref{tab:dynkin}).   We remark, as done in \cite{OO}, 
 that for the  case $\tilde{D}_6$ there are two types of isomonodromic families 
corresponding  to ${\rm PV}_{\rm deg}$ and PIII($D_6$). The same 
holds for $\tilde{E}_7$. 

In Section 4, a Zariski open set of the moduli space $\mathcal{M}$ of 
connections is described for each of the ten families. The 
corresponding isomonodromic equation produces an explicit Painlev\'e equation,
confirming the statements of Table 1.

     The contents of this paper is the following. The first section deals with
 the formal and analytic data attached to a differential module $M$ over 
$\mathbb{C}(z)$.  The   connections on  $\mathbb{P}^1$ inducing given formal 
and analytic data  are  studied.
   A weak and a strong form of the Riemann-Hilbert problem is treated. This
 result is also obtained by \cite{BMM} in a slightly different setting. 
   
     In Section 2, `good' families of connections on $\mathbb{P}^1$ are
 described and studied. The monodromy space $\mathcal{R}$ is defined as a 
categorical  quotient of the analytic data. 
  
 Then the  ten families where the fibres 
 of $RH:\mathcal{M}\rightarrow \mathcal{R}$ have     dimension 1 are computed.
 The third Section contains the computation of the
     ten monodromy spaces $\mathcal{R}\rightarrow \mathcal{P}$ and the
 singularities  of the fibres.
    
    A theory of apparent singularities $q$ is developed in Section 4. This is 
essential for the computation of the second order equation $q''=R(q,q',t)$
 (where $R$ is a rational 
    function of $q',q,t$) of the Painlev\'e type and of a corresponding 
symplectic structure with canonical coordinates $p,q$ and a Hamiltonian 
equation.  We obtain explicit Hamiltonian systems and explicit Painlev\'e equations for the
nine families (see Subsections \ref{ss:p5-d5}--\ref{ss:p1-e8}) which are natural from the view point  of the Okamoto--Painlev\'e pairs.  The explicit forms of  equations depend on the choice of a cyclic vector, the choice of the parameter $t$ and choices for the constants in the monodromy space. Though we will not tune up these data such that our explicit forms coincide with the classical Painlev\'e equations as in \cite{Gambier, P, JM}, one can transform one to the other by some birational transformation of coordinates. 
 Most of these computations in Section 4 were made using {\it Mathematica}.

\section{Singularities of a differential module}

\subsection{Summary}
Let $M$ be a differential module over $K=\mathbb{C}(z)$. The formal data (generalized local exponents, formal monodromy), and the analytic data (monodromy, Stokes matrices, links) of $M$ are described.  The weak form of the Riemann-Hilbert problem for arbitrary singularities has the positive answer:\\ 
{\bf Theorem 1.7} {\it For given formal and analytic data, there exists a unique 
{\rm (}up to isomorphism{\rm )} differential module $M$ inducing these data.}\\ 

A strong form of the Riemann--Hilbert problem is:\\
{\bf Theorem 1.11 }{\it Suppose that $M$ is irreducible and has at least one
 {\rm(}regular or irregular{\rm )}
singular point which is unramified. Then there is a connection $(\mathcal{V},\nabla )$ on $\mathbb{P}^1$ representing $M$,  such that  $\mathcal{V}$ is free (i.e., a direct sum of copies of $\mathcal{O}_{\mathbb{P}^1}$) and the poles of the connection $\nabla$ have the minimal order derived from the Katz invariant.}\\
 
 Results concerning invariant lattices are developed for the proof of Theorem 1.11.  
 That the strong Riemann--Hilbert problem has a negative answer if all the singularities of $M$ are ramified, is shown by two families of examples related to Painlev\'e equations. 

Bolibruch's work \cite{AB} on the strong form of the  Riemann--Hilbert problem is extended in the 
paper \cite{BMM} to the case of irregular singularities. Our Theorems 1.7 and 
1.11 clarify and supplement \cite{BMM}, by introducing links.


 For the convenience  of the  reader, the useful  compact way to describe the formal and the analytic  singularities of differential modules (see \cite{PS} for more details) is explained in the next subsections.  Explicit examples are given which will be used in the calculations for the monodromy spaces and the Painlev\'e equations.

\subsection{The formal classification}
This is the classification of differential modules $M=(M,\delta )$ over the differential field of the formal Laurent series $\mathbb{C}((t))$ (here $t$ is the local parameter) , due to M.~Hukuhara\cite{Hu} and  
H.~Turrittin \cite{Tu}.  For notational convenience we will use the derivation
$t\frac{d}{dt}$ on $\mathbb{C}((t))$. The $\mathbb{C}$-linear map $\delta :M\rightarrow M$
has, by definition, the property $\delta (fm)=t\frac{df}{dt}\cdot m+f\cdot \delta (m)$ for
 $f\in \mathbb{C}((t)),\ m\in M$.\\
 The module $M$ is called {\it regular singular} 
 (this includes regular) if there is an invariant lattice $\Lambda \subset M$, i.e., $\Lambda\subset M$ is a free $\mathbb{C}[[t]]$-submodule containing a basis of $M$ such that 
 $\delta (\Lambda )\subset \Lambda$. 
  A regular singular $M$ has a basis $e_1,\dots ,e_d$ such that  the vector space $W:=\oplus _{i=1}^d\mathbb{C}e_i$ is invariant under $\delta$ and such that the distinct eigenvalues $\lambda _1,\dots ,\lambda _s$
(with $1\leq s\leq d$) of $\delta$ acting on $W$ satisfy $\lambda _i-\lambda _j\not \in \mathbb{Z}$ for $i\neq j$. Using this basis the operator $\delta$ on $M$ obtains the form
$t\frac{d}{dt}+A$, where $A$ is the matrix of $\delta$ operating on $W$. The $\lambda _i$ are called the {\it local exponents}. These are only unique up to integers. The {\it (formal) monodromy} matrix is (up to conjugation) $e^{2\pi iA}$.

 Clearly $\Lambda :=\oplus _{i=1}^d\mathbb{C}[[t]]e_i$ is an invariant lattice. 
The {\it non resonant case} is defined by $s=d$, i.e., the matrix $A$ is
diagonalizable and its eigenvalues $\lambda _1,\dots ,\lambda _d$ satisfy 
$\lambda _i-\lambda _j\not \in \mathbb{Z}$ for $i\neq j$. In the non resonant case the collection of all invariant lattices is $\{\oplus _{i=1}^d\mathbb{C}[[t]]t^{n_i}e_i\ |\ n_1,\dots ,n_d\in \mathbb{Z}\}$ and the formal monodromy has $d$ distinct eigenvalues.

The {\it solutions} of a regular singular module $M$, say, represented in matrix form 
$t\frac{d}{dt}+A$, are vectors $v$ with $(t\frac{d}{dt}+A)v=0$. One needs the following 
differential ring extension $Univ_{rs}:=\mathbb{C}((t))[\{t^a\}_{a\in \mathbb{C}} ,\ell ]$ of $\mathbb{C}((t))$ to obtain the vector space $V$ of all solutions. The {\it symbols} $t^a, \ell$
are defined by the rules $t^a\cdot t^b=t^{a+b}$, $t^1$ coincides with $t\in \mathbb{C}((t))$.
Further $t\frac{d}{dt}t^a=at^a,\ t\frac{d}{dt}\ell =1$. The intuitive meaning of these symbols 
is rather clear: $t^a$ stands for $e^{a\log t}$ and $\ell$ for $\log t$. Because these 
functions are multivalued, they are replaced by symbols.

Then $V$ consists of the vectors $v$ with coordinates in $Univ_{rs}$ satisfying 
$(t\frac{d}{dt}+A)v=0$. In other words,
$V=\{ v\in Univ_{rs}\otimes _{\mathbb{C}((t))}M\ |\ \delta (v)=0\}$. The ring $Univ_{rs}$ 
has a $\mathbb{C}((t))$-linear differential automorphism $\gamma$, defined by
$\gamma t^a=e^{2\pi ia}t^a,\ \gamma \ell =\ell +2\pi i$. Now $\gamma$ induces on
automorphism $\gamma \otimes id$ on $Univ_{rs}\otimes M$, commuting with $\delta$.
Then $V$ is invariant under $\gamma$ and the restriction of $\gamma$ to $V$, again written
as $\gamma$ or $\gamma _V$, is the formal monodromy. From the pair $(V,\gamma _V)$
one recovers the differential module $(M,\delta _M)$ as the $\mathbb{C}((t))$-vector space of the $\gamma$-invariant elements of $Univ_{rs}\otimes _{\mathbb{C}}V$. On the last space the operator $\delta$ is defined by $\delta (u\otimes v)=\delta (u)\otimes v$ for 
$u\in Univ_{rs},\ v\in V$. The restriction of this $\delta$ to $M$ is the $\delta _M$.
The above describes an equivalence between the category of the regular singular 
differential modules and the category of the pairs $(V,\gamma )$ consisting of a finite dimensional vector space $V$ and an $\gamma \in {\rm GL}(V)$. This equivalence respects all
constructions of linear algebra, in particular tensor products.\\

This maybe somewhat abstract way to deal with regular singular differential modules
extends to the case of {\it irregular singular} differential modules. It greatly simplifies the various classical classification results. 

A typical example of an irregular singular module
is the one-dimensional module
 $M=\mathbb{C}((z))e$ with $\delta e=(a+q)e $ with $q\in t^{-1}\mathbb{C}[t^{-1}],\ q\neq 0,\ a\in \mathbb{C}$. We call $a+q$ the {\it (generalized) local exponent} and $q$ the {\it eigenvalue}. One observes that $q$ is unique and $a$ is unique up to a shift over an integer.
  
 A more complicated example is the following. For any integer 
 $n\geq 1$ we consider the field extension $\mathbb{C}((t^{1/n}))$ of degree $n$ and an
 element $a+q\in \mathbb{C}+(t^{-1/n}\mathbb{C}[t^{-1/n}])$. Then we define the differential
 module $\mathbb{C}((t^{1/n}))e$ of rank one over $\mathbb{C}((t^{1/n}))$ by
 $\delta (e)=(a+q)e$. Now $M$  is equal to this object, seen as a differential module over
 the field $\mathbb{C}((t))$. This module has dimension $n$. From these examples and
 the regular singular differential modules one can build, by constructions of linear algebra,
 all differential modules. In order to have solutions for all differential modules over
 $\mathbb{C}((t))$ we have to introduce new symbols $e(q)$ for 
 $q\in \mathcal{Q}:=\bigcup _{n\geq 1}t^{-1/n}\mathbb{C}[t^{-1/n}]$. The rules are
 $t\frac{d}{dt}e(q)=q\cdot e(q)$ and $e(q_1)e(q_2)=e(q_1+q_2)$.  One obtains the
 differential ring extension $Univ:=\oplus _{q\in \mathcal{Q}}Univ_{rs}\cdot e(q)$, equipped
 with the differential automorphism $\gamma$, extending the $\gamma$ on $Univ_{rs}$
 by $\gamma e(q)=e(\gamma q)$. The meaning of $\gamma (q)$ is already defined 
 since $\gamma (t^a)=e^{2\pi ia}t^a$ for any $a\in \mathbb{C}$. The intuitive meaning
 of $e(q)$ is rather evident, namely $e^{\int q\frac{dt}{t}}$. Since the latter is a 
 multivalued function we avoid its use and use the symbol $e(q)$ instead. 

 The solution space $V$ of a differential module $M$, say, represented by the matrix
 equation $t\frac{d}{dt}+A$ where $A$ is a $d\times d$-matrix with coefficients in 
 $\mathbb{C}((t))$, is defined as $V= \{v\in (Univ )^d \  |\ (t\frac{d}{dt}+A)(v)=0  \}$.
 In other words $V=\{ v\in Univ \otimes _{\mathbb{C}((t))}M\ | \ \delta (v)=0\}$.
   As before, there is an action of $\gamma$ on $V$. Moreover $V$ has a direct sum
   decomposition $V=\oplus _{q\in \mathcal{Q}}V_q$ where 
   $V_q:=\{v\in Univ_{rs}\cdot e(q)\otimes _{\mathbb{C}((t))}M\ |\ \delta (v)=0\}$. As the dimension of
   $V$ is finite (equal to $\dim _{\mathbb{C}((t))}M$), almost all $V_q$ are 0. Clearly
   $\gamma (V_q)=V_{\gamma (q)}$. Thus we have attached to $M$ a tuple
   $(V,\{V_q\},\gamma )$ consisting of a finite dimensional vector space $V$ over
   $\mathbb{C}$ and subspaces $V_q$ with $V=\oplus _{q\in \mathcal{Q}}V_q$ and
   an element $\gamma \in {\rm GL}(V)$ such that $\gamma (V_q)=V_{\gamma (q)}$ for
   all $q$. From this tuple one can recover $(M,\delta _M)$ as the $\mathbb{C}((t))$-vector
    space of the $\gamma$-invariant elements of 
   $\oplus _{q\in \mathcal{Q}}Univ_{rs}\cdot e(-q)\otimes _{\mathbb{C}}V_q$. By definition 
   $\delta$ acts as zero on $V$ and thus induces $\delta _M$. In fact, 
   $M\mapsto (V,\{V_q\},\gamma )$ defines an equivalence of categories commuting with all
   operations of linear algebra, and in particular with  tensor product.  Our formal
   classification is that of the tuples $(V,\{V_q\},\gamma )$.
   
   The elements $q$ with $V_q\neq 0$ are called the {\it eigenvalues} and $\gamma$,
   acting on $V$, is called the {\it formal monodromy}. The {\it Katz invariant} $r(M)$ of $M$
    is the maximum of the degrees in $t^{-1}$ of the eigenvalues $q$.\\
   
   \begin{examples}{\rm
   We {\it illustrate} the above by classifying all differential modules $M$ of 
   dimension 2 such that $\Lambda ^2M$ is isomorphic to the trivial module 
   ${\bf 1}:=\mathbb{C}((t))e$ with $\delta e=0$. The possibilities for the tuple 
   $(V,\{V_q\},\gamma )$ are
   \begin{itemize}
   \item[(i)] $V=V_0$ and $\gamma \in {\rm SL}(V)$. This is the {\it regular singular} case.
   By taking a logarithm $2\pi i A$ of $\gamma$ one obtains the matrix equation
   $t\frac{d}{dt}+A$.
   \item[(ii)] 
   $V=V_q\oplus V_{-q}$ with $q=a_1t^{-r}+\cdots +a_rt^{-1}$ and $a_1\neq 0$.This is
    the {\it  unramified irregular} case with eigenvalues $\pm q$ and Katz invariant $r$. 
    Give the spaces $V_q,V_{-q}$ a basis $e_1$ and $e_2$. Then the matrix of
     $\gamma$ has the form 
   ${\alpha \ \ \ \ 0 \choose \ 0  \ \ \ \alpha ^{-1}}$. A corresponding matrix differential
   equation can be written as $t\frac{d}{dt}+\left(\begin{array}{cc} q+a&0\\
   0& -q-a\end{array}\right)$ with $e^{2\pi i a}=\alpha $.
   \item[(iii)] For the {\it ramified irregular} case $V=V_q\oplus V_{-q}$ with Katz invariant $r$, 
   one must  have $r=\frac{1}{2}+m,\ m\in \mathbb{Z},\ m\geq 0$ and
    $q=  t^{1/2}(a_1t^{-r-1/2}+\cdots +a_*t^{-1}),\ a_1\neq 0$. This follows from
    $\gamma (q)=-q$. Consider a basis $b_1$ and  $b_2$ for $V_q$ and $V_{-q}$ such
     that $\gamma (b_1)=b_2$. Then $\gamma (b_2)=-b_1$ since $\gamma \in {\rm SL}(V)$.
   For the computation of the corresponding differential module, it is easier to compute
   first the invariants under $\gamma ^2$. This yields a differential module 
   $N=\oplus _{i=1}^2 \mathbb{C}((t^{1/2}))e_i$ over $\mathbb{C}((t^{1/2}))$ with
   $\delta (e_1)=qe_1,\ \delta (e_2)=-qe_2$. The element $\gamma$ acts on $N$  by
   $\gamma e_1=e_2, \ \gamma e_2=e_1$ and $\gamma t^{1/2}=-t^{1/2}$. The module
   $M$ of the invariants under $\gamma$ has the basis 
   $f_1=e_1+e_2,\ f_2=t^{1/2}(e_1-e_2)$.  Write $q=t^{1/2}h$.
   Then $\delta$ on the basis $f_1,f_2$ yields the matrix differential equation
   $t\frac{d}{dt}+ \left(\begin{array}{cc} 0 & th\\ h & \frac{1}{2}\end{array}\right)$.
   
   \end{itemize}
   }\end{examples}
 \begin{definition and examples} Invariant lattices. {\rm \\    
 Let the differential module $M=(M,\delta )$ have Katz invariant $r$ and let $r^+$ denote the smallest integer $\geq r$. A free submodule $\Lambda \subset M$ over 
 $\mathbb{C}[[t]]$,  containing a basis of $M$ is usually called a {\it lattice}. We say that
 $\Lambda$ is an  {\it invariant lattice} if, moreover
 $t^{r^+}\delta \Lambda \subset \Lambda$. There exists an invariant lattice, in fact
 the {\it standard lattice} defined  in \cite{PS}, p. 307 and p. 311,  is an
 invariant lattice. For later use we will compute all invariant lattices for the items of  Examples 1.1.
 \begin{enumerate}
 \item[(i)] If the regular singular module $M$ is {\it non resonant}, then $M$ has a basis
 $e_1,e_2$ with $\delta e_1=\frac{\theta}{2}e_1,\ \delta e_2=-\frac{\theta}{2}e_2$ and
 $\theta \not \in \mathbb{Z}$. The notation $\frac{\theta}{2}$ is chosen for historical reasons. The  invariant lattices are $\mathbb{C}[[t]]t^{n_1}e_1+\mathbb{C}[[t]]t^{n_2}e_2$ for any $n_1,n_2\in \mathbb{Z}$. \\
 A typical {\it resonant} case is $M=\mathbb{C}((t))e_1+\mathbb{C}((t))e_2$ with
 $\delta e_1=e_2,\ \delta e_2=0$. The invariant lattices are 
 $\mathbb{C}[[t]]t^{n_1}e_1+\mathbb{C}[[t]]t^{n_2}e_2$ with $n_1,n_2\in \mathbb{Z}$
 and $n_1\geq n_2$.
  
\item[(ii)] $M$ has a basis $e_1,e_2$ with $\delta e_1=(q+a)e_1,\ \delta e_2=-(q+a)e_2$.
In this case $r^+=r$ and the invariant lattices are  $\mathbb{C}[[t]]t^{n_1}e_1+\mathbb{C}[[t]]t^{n_2}e_2$ for any $n_1,n_2\in \mathbb{Z}$.   

\item[(iii)] $M$ has basis $f_1,f_2$ with $\delta f_1=hf_2,\ \delta f_2=th f_1+\frac{1}{2}f_2$.
Now $r^+=r+\frac{1}{2}$. The  invariant lattices are only the lattices
$t^n\cdot (\mathbb{C}[[t]]f_1+\mathbb{C}[[t]]f_2)$ and   $t^n\cdot (\mathbb{C}[[t]]tf_1+
\mathbb{C}[[t]]f_2)$ where $n\in \mathbb{Z}$.
\end{enumerate}
We omit the easy proofs for (i) and (ii). {\it The proof of case} (iii): \\
Consider the operator $\Delta =h^{-1}\delta$ on $M$. Thus $\Delta f_1=f_2,\
\Delta f_2=tf_1+\frac{1}{2h}f_2$ and $\Delta (fm)= (h^{-1}t\frac{df}{dt})\cdot m+f\cdot \Delta (m)$.  

A lattice $\Lambda$ is invariant if and only if $\Delta \Lambda \subset \Lambda$.  
If $\Lambda$ is an invariant lattice then also $t^n\cdot \Lambda$ for any $n\in \mathbb{Z}$.
The lattices generated by $f_1,f_2$ and by $tf_1,f_2$ are clearly invariant. Let 
$\Lambda$ be any invariant lattice. After multiplication by some power of $t$ we may suppose that $\Lambda \subset  (\mathbb{C}[[t]]f_1+\mathbb{C}[[t]]f_2)$ and not contained
in $t\cdot (\mathbb{C}[[t]]f_1+\mathbb{C}[[t]]f_2)$. If $\Lambda =\mathbb{C}[[t]]f_1+\mathbb{C}[[t]]f_2$, then we are finished. If not we consider the invariant lattice
 $\Lambda +t\cdot (\mathbb{C}[[t]]f_1+\mathbb{C}[[t]]f_2)$. Since $\Delta$ induces on
 $(\mathbb{C}[[t]]f_1+\mathbb{C}[[t]]f_2)/t\cdot (\mathbb{C}[[t]]f_1+\mathbb{C}[[t]]f_2)$ 
a nilpotent map with only one proper invariant subspace, namely generated by the image
of $f_2$, we have that $\Lambda +t\cdot (\mathbb{C}[[t]]f_1+\mathbb{C}[[t]]f_2)= 
(\mathbb{C}[[t]]tf_1+\mathbb{C}[[t]]f_2)$. It follows that $\Lambda$ contains an element
of the form $af_1+f_2$ for some $a\in t\mathbb{C}[[t]]$. Now
\[\Delta(af_1+f_2)-(a+\frac{1}{2h})(af_1+f_2)=
(t+h^{-1}t\frac{da}{dt}-(a+\frac{1}{2h})a)f_1\in \Lambda  . \]
Thus $tf_1\in \Lambda$ and also $f_2\in \Lambda$. Then
$\Lambda =\mathbb{C}[[t]]tf_1+\mathbb{C}[[t]]f_2$.\\

\noindent 
{\it Comment}. Two lattices $\Lambda _1, \Lambda _2$ in $\mathbb{C}((t))^2$ 
are  called {\it equivalent} if there exists an integer $n$ with $\Lambda _1=t^n\cdot \Lambda _2$. Two classes of lattices $[\Lambda _1],\  [\Lambda _2]$ form an edge if 
the representatives $\Lambda _1,\Lambda _2$ can be chosen such that there are proper inclusions $t\cdot \Lambda _1\subset \Lambda _2\subset \Lambda _1$. 
One obtains a tree with vertices the classes of lattices and edges as above. If one replaces $\mathbb{C}$ by a finite field then this object is the well known  Bruhat-Tits tree.  

The classes of the invariant lattices form a subset of this tree.  This subset is a line for the first case of (i) and a half line for the second case of (i). In case (ii), it is again a line 
and in case (iii) this subset  consists of two vertices which form an edge. }\hfill $\square$ \end{definition and examples}

\subsection{The analytic classification}
This is the classification of the differential modules $M$ over the field of the convergent
Laurent series $\mathbb{C}(\{t\})$. Again we use the derivation  $t\frac{d}{dt}$.
One associates to $M$ the formal differential module
$\widehat{M}=\mathbb{C}((t))\otimes M$. If $\widehat{M}$ is {\it regular singular}, then
one calls $M$ also regular singular. In that case there exists also a basis $e_1,\dots
,e_d$ of $M$  such that $W:=\oplus _{i=1}^d \mathbb{C}e_i$ is invariant under 
$\delta$ and the distinct eigenvalues of the matrix $A$ of $\delta$ on $W$ do not differ
by an integer. Then $M$ is isomorphic to the module corresponding to the matrix
differential operator $t\frac{d}{dt}+A$. The usual {\it topological monodromy} around 
$t=0$ coincides with the formal monodromy and that ends the classification.\\

 If $\widehat{M}$ is {\it irregular singular}, then it induces a tuple 
 $(V,\{V_q\},\gamma )$.  The {\it singular directions} $d\in \mathbb{R}$ of $M$ depend
 only on $\widehat{M}$ and are defined as follows. Let $q_1,\dots ,q_s$ denote the eigenvalues of $\widehat{M}$.  A direction $d\in \mathbb{R}$ is singular for  $q_i-q_j$ (with $i\neq j$) if the function  $exp (\int  (q_i-q_j)\frac{dt}{t})$ has `maximal descent' for $r\rightarrow 0$ on the half line $t=re^{id}$. More explicitly, if 
 $q_i-q_j=\alpha _kt^{-k}+\cdots +\alpha _1t^{-1},\ \alpha _k\neq 0$, then $d$ is a singular direction if and only if $\alpha _kre^{-idk}$ is a positive real number.
  The collection of the singular directions is the union over $i\neq j$ of the singular directions of $q_i-q_j$. 
 
If  a direction $d$ is non singular, then there is a functorial map $mults _d$, the 
{\it multisummation in the direction $d$}, which maps the (symbolic) solution space 
$V$ to the space of the actual  solutions $V(S)$ in a certain sector $S$ at $t=0$ around  $d$. For every $v\in V$ the element $mults_d(v)$ has $v$ as its asymptotic expansion.

For each singular direction $d$, there is an analytic object, namely the Stokes map 
$St_d\in {\rm GL}(V)$. The Stokes map $St_d$ is defined by comparing the multisummation at directions $d^-<d<d^+$ close to $d$. More precisely 
$mults_{d^+}\circ St_d=mults_{d^-}$. The map $St_d$ is unipotent and has the form
$Id+\sum _{i,j}L_{i,j}$ where the sum is taking over all pairs $i,j$ such that $d$ is singular for $q_i-q_j$ and where $L_{ij}$ is a linear map from $V_{q_i}$ to $V_{q_j}$.
The isomorphy class of $M$ induces a tuple 
$(V,\{V_q\},\gamma ,\{St_d\})$, where the $St_d$ are described above and where moreover the relation $\gamma ^{-1}St_d\gamma =St_{d+2\pi }$ holds.\\ 
The {\bf main result} of the  asymptotic analysis of irregular singularities is: \\

\noindent {\it The category of the differential modules over $\mathbb{C}(\{t\})$  is equivalent to the category of the tuples $(V,\{V_q\},\gamma ,\{St _d\})$, satisfying the above properties. This equivalence respects all constructions of linear algebra,
in particular the tensor product.  }\\

\noindent An {\bf important property} that we will use is:\\

\noindent 
{\it Let $0\leq d_1<\dots <d_s< 2\pi $ denote the singular directions in $[0,2\pi )$. Then the topological monodromy around the singular point is conjugated to 
$\gamma \circ St_{d_s}\circ \cdots \circ St_{d_1}$.}\\

\noindent We note that this conjugation depends on the way the solution space at a point close to the singular point $t=0$ is identified with the (formal) solution space $V$.
Now we illustrate the above by continuing Examples 1.1. 

\begin{examples} {\rm Let $M$ be an irregular differential module of dimension 2 over 
$\mathbb{C}(\{t\})$ such that $\Lambda ^2M=\{\bf 1\}$. \\  
(ii) If $\widehat{M}$ is unramified with Katz invariant $r$, then 
$V=V_q\oplus V_{-q}$, $q\in t^{-1}\mathbb{C}[t^{-1}]$ has degree $r$ in $t^{-1}$.
We recall that $\gamma$ has the matrix  ${\alpha \ \ \ \ 0 \choose \ 0  \ \ \ \alpha ^{-1}}$
on any basis $e_1,e_2$ of $V$ such that $V_q=\mathbb{C}e_1,\ V_{-q}=\mathbb{C}e_2$. For $q-(-q)$ there are $r$ singular directions (in $[0,2\pi )$)  and the same holds for
$(-q)-q$. The two pairs of singular directions intertwine. For the first ones the Stokes
matrices (w.r.t. the basis $e_1,e_2$) have the form ${1\ *\choose 0\ 1 }$, and for
the second ones the form is ${1\ 0 \choose *\ 1 }$.   Thus the Stokes matrices are given
by $2r$ constants $c_i$ and the topological monodromy around $t=0$ is up to conjugation (and we may choose the order) equal to
\[ {\alpha \ \ \ \ 0 \choose \ 0  \ \ \ \alpha ^{-1}}\cdot {1\ 0 \choose c_1\ 1  }\cdot
{1\ c_2\choose 0 \ 1 }\cdots {1\ \ 0 \choose c_{2r-1} \ 1 }\cdot 
{1\ c_{2r} \choose 0 \ \  1 } .\]
The basis $e_1,e_2$ is not unique, whereas the 1-dimensional spaces $V_q$ and
$V_{-q}$ are.  If we want Stokes data, independent of the choice of $e_1,e_2$,
then we have to divide the space $\mathbb{A}^{2r}$ of the tuples
$(c_1,\dots ,c_{2r})$ by the action of the group $\mathbb{G}_m$. For this action the $c_{2i}$ can be given weight $+1$ and the $c_{2i-1}$ weight $-1$. 

\noindent
(iii) If $\widehat{M}$ is ramified, then there are again $2r$ singular directions in
$[0, 2\pi )$ and Stokes matrices of the form  ${1\ *\choose 0\ 1 }$ and ${1\ 0 \choose *\ 1 }$. The singular directions intertwine. We choose now a basis $e_1,e_2$ of $V$
with $V_{q}=\mathbb{C}e_1,\ V_{-q}=\mathbb{C}e_2$  and $\gamma e_1=e_2,
\gamma e_2=-e_1$. The topological monodromy around $t=0$ is conjugated to the product 
\[ { 0\  -1\choose 1\ \  \ 0 }\cdot {1\ 0\choose c_1\ 1 }\cdot {1\ c_2\choose 0 \ 1 }
\cdots {1\ \ \ 0 \choose c_{2r} \ 0} .\]
In this case one may change the basis $e_1,e_2$ only into $\lambda e_1, \lambda e_2$
with $\lambda \in \mathbb{C}^*$. This does not have an effect on the Stokes data
$(c_1,\dots ,c_{2r})$ and no division by $\mathbb{G}_m$ is needed.
}\hfill $\square$ \end{examples}

\subsection{ The data for global differential modules}
By a global differential module we mean a differential module $M$ over the field 
$K=\mathbb{C}(z)$. We investigate the data that will describe $M$.\\
 The first case that we consider is classical, namely:\\
{\it  The position of the singular points  $\{p_1,\dots ,p_r\}$ of $M$ is fixed and all the singular points are supposed to be regular singular}.

One  introduces the monodromy for  $M$ in the usual way. That is, one chooses a base point  $b\in \mathbb{P}^1\setminus \{p_1,\dots ,p_r\}$ and loops 
 $\alpha _1,\dots ,\alpha _r$ around the singular points, generating the fundamental
 group $\pi _1:=\pi _1(\mathbb{P}^1\setminus \{p_1,\dots ,p_r\},b)$. There is only one relation, namely 
 $\alpha _1\cdots \alpha _r=1$. Then $M$ induces a monodromy homomorphism
 \[mon_M: \pi _1\rightarrow {\rm GL}(V(b))\ ,\] where  $V(b)$ denotes the solution space at $b$. We note that $mon_M (\alpha_i)$ is conjugated to the local monodromy at $p_i$ (formal or topological). A weak solution of the Riemann--Hilbert problem reads 
 (\cite{PS},Thm 6.15): 
 
\begin{proposition} The functor $M\mapsto mon_M$ from the category of the differential
modules with regular singularities at $\{p_1,\dots ,p_r\}$ to the  category of the finite dimensional complex representations of $\pi$, is an equivalence of categories. This 
equivalence respects all constructions of linear algebra, in particular tensor products. 
\end{proposition} 
 
\noindent {\bf Notation}. For any point $p\in \mathbb{P}^1$ we introduce the local parameter $t_p$,
which is $z-p$ if $p\in \mathbb{C}$ and $z^{-1}$ for $p=\infty$. The field $K_p$
is the field of the meromorphic functions at $p$, i.e., $\mathbb{C}(\{t_p\})$ and
$\widehat{K}_p$ is the completion of $K_p$, i.e., $\mathbb{C}((t_p))$. Further 
$O_p\subset K_p$ and $\widehat{O}_p\subset \widehat{K}_p$ are the valuation rings,
i.e., $O_p=\mathbb{C}\{t_p\}$ and $\widehat{O}_p=\mathbb{C}[[t_p]]$.\\

\noindent One associates to a global differential module $M$ with fixed singularities 
$\{ p_1,\dots ,p_r\}$, the  data: the isomorphy classes of the $\{K_{p_i}\otimes _KM \}$ and the monodromy representation $mon _M$ as above. Now we give an example
showing that {\it this is not sufficient for the reconstruction of $M$}.

\begin{example} Two singular points $0$ and $\infty$, both irregular. \\ {\rm 
At both points we prescribe local analytic data for the differential module $M$. In other words, we prescribe the two analytic differential modules $M_0=K_0\otimes M$ and 
$M_\infty=K_\infty \otimes M$.  As we will see in Observations 1.8,
this leads to a connection $(\mathcal{M},\nabla )$ on $\mathbb{P}^1$, where $\mathcal{M}$
is a vector bundle and $\nabla :\mathcal{M}\rightarrow \Omega (k[0]+k[\infty ])\otimes \mathcal{M}$ (for some $k>0$). The restrictions $T_0$ and $T_1$ of this connection to the two open sets
$\mathbb{P}^1\setminus \{0\}$ and $\mathbb{P}^1\setminus \{\infty \}$ are known from
the given data $M_0$ and $M_\infty$. We suppose now that the topological
monodromies of $M_0$ and $M_\infty$ are trivial. Thus the restrictions $T_{0,1}$ and 
$T_{1,0}$ of $T_0$ and $T_1$
to the open set $\mathbb{P}^1\setminus \{0,\infty \}$ are trivial. The glueing is given by
an isomorphism $T_{0,1}\rightarrow T_{1,0}$.  Let $m$ denote the dimension of $M$.
Then an isomorphism is given by a linear bijection $L:\mathbb{C}^m\rightarrow \mathbb{C}^m$. Further $L$ is unique up to multiplication (on the left, respectively on the right) by an automorphism of $M_0$ and $M_\infty$. One can easily produce $M_0$ and $M_\infty$
which have only $\mathbb{C}^*$ as group of automorphisms. Therefore the possible connections
$(\mathcal{M},\nabla )$ and also the possible differential modules $M$ are in bijection with
 ${\rm PGL}(m,\mathbb{C})$.   }\hfill $\square$\end{example}

\begin{definition} Links and the formal and analytic data. {\rm  \\
What is missing is a `link' between the solution space $V(b)$ at the base point $b$ with the (symbolic) solution spaces $V(p_i)$ at the singular points. This idea goes back
to the work of Jimbo--Miwa--Ueno \cite{JMU}.
 We make the following 
construction to remedy this.

As before, $\alpha _1,\dots ,\alpha _r$ are loops starting at $b$ around the singular
points. For each $p_i$ we choose a point $p_i^*$ on the loop close to
 $p_i$ and a line
segment $[p_i^*,p_i]$ which is a  non--singular direction at $p_i$. 
The `{\it link}'
$L_i:V(b)\rightarrow V(p_i)$ is defined by analytic continuation 
from $V(b)$ to $V(p_i^*)$,
followed by the inverse of the multisummation map 
$mults :V(p_i)\rightarrow V(p_i^*)$
in the direction $[p_i^*,p_i]$ (seen as an element in $\mathbb{R}$). 
The role of the monodromy map $mon_M$ is taken over by  links, i.e., 
 the linear bijections
$L_1,\dots ,L_r$. The multisummation $mults :V(p_i)\rightarrow 
V(p_i^*)$ (in the direction
$[p_i,p_i^*]$) is used to identify the two vector spaces.
Then the local topological monodromy $top_i$, along a circle starting 
in $p_i^*$, is expressed 
as a product of the Stokes maps and the formal monodromy at $p_i$. 
The relation $\alpha_1\cdots \alpha_r=1$  translates into
\[ L_r^{-1}\circ top_r\circ L_r \dots L_2^{-1}\circ top_2\circ L_2\circ L_1^{-1}\circ top _1 \circ L_1=1.\]

The `{\it formal and the analytic data}' for $M$ are defined as:
\begin{enumerate}
\item[(1)] The position of the singular points $p_1,\dots ,p_r$;
\item[(2)] for each $i$, the formal structure $(V(p_i),\{V(p_i)_q\},\gamma _i)$ at $p_i$;
\item[(3)] for each $i$, the Stokes maps at $p_i$; 
\item[(4)] the links $L_i:W\rightarrow V(p_i)$.
\item[(5)] These data are supposed to satisfy the relation
\[ L_r^{-1}\circ top_r\circ L_r \dots L_2^{-1}\circ top_2\circ L_2\circ L_1^{-1}\circ top _1 \circ L_1=1.\]
\end{enumerate}
Here $W$ stands for the space $V(b)$. The {\it formal part} of the data is (1) (the position of the singular points) and the eigenvalues $q$ at each singular point.  
The {\it analytic part} of the data is the direct sum decompositions $\oplus _qV(p_i)_q$
 of the spaces $V(p_i)$, including the permutation of the $V(p_i)_q$ induced by $\gamma$; further  (3) and (4), since this combines the links and the Stokes maps. 
 We observe that  these `formal and analytic data' are considered up to the automorphisms of $W$ and of the $V(p_i)$.
 
 One might use $L_1$ to identify $W$ with $V(p_1)$ and then one is only left with links $L_i:V(p_1)\rightarrow V(p_i)$ for $i=2,\dots ,r$.
Another way to reduce the number of links by one, is to choose as base point $b$ the singular point $p_1$ and define links $L_i:V(p_1)\rightarrow V(p_i)$ for $i=2\dots ,r$.
  }\hfill $\square$ \end{definition}

\begin{theorem} For  given  `formal and analytic data', as above, there exists a differential module $M$ over $K=\mathbb{C}(z)$ inducing the data. Moreover $M$ is unique up to isomorphism.
\end{theorem}

\begin{observations} Global differential modules and connections. \\{\rm 
Before giving the proof of Theorem 1.7, we have to make the relation between differential modules $M$
over $K=\mathbb{C}(z)$ and connections $(\mathcal{M},\nabla )$ (with singularities) on $\mathbb{P}^1$ explicit. 

Let a connection $(\mathcal{M},\nabla )$ (with singularities) be given. We note that
we may regard this connection either algebraically or analytically, 
because of the GAGA theorem.
On proper Zariski-open subsets of $\mathbb{P}^1$  we sometimes see $\mathcal{M}$ as an analytic vector bundle. The generic 
fibre $M$ of $\mathcal{M}$ is a vector space of finite dimension over $K$, equipped with a
$\nabla :M\rightarrow \Omega _{K/\mathbb{C}}\otimes M$. After identifying 
$\Omega _{K/\mathbb{C}}$ with $Kdz$, this gives $M$ the structure of a differential module.

On the other hand, let a differential module $M$ be given. This is written as a 
(generic) connection $\nabla :M\rightarrow \Omega _{K/\mathbb{C}}\otimes M$.
Consider a set  $\{p_1,\dots ,p_r\}\subset \mathbb{P}^1$ of points including the singular points of $M$. For each $i$ one chooses an $\widehat{O}_{p_i}$-lattice $\Lambda _i$ in
$\widehat{K}_{p_i}\otimes M$ and let $k_i$ satisfy $\nabla (\Lambda _i)\subset
t_i^{-k_i}dt_i\otimes \Lambda _i$ (where $t_i$ is the local parameter at $p_i$). For 
$p\not \in \{p_1,\dots ,p_r\}$, the module $\widehat{K}_p\otimes M$ is non singular and there is a unique $\widehat{O}_p$-lattice $\Lambda _p$ with 
$\nabla (\Lambda _p)\subset dt_p\otimes \Lambda _p$, where $t_p$ denotes the local
parameter at $p$. Then there exists a unique connection $(\mathcal{M},\nabla )$ on $\mathbb{P}^1$ having the following properties (see \cite{PS}, Lemma 6.16):
\begin{enumerate}
\item $\mathcal{M}(V)\subset M$ for all, non empty, Zariski--open $V\subset \mathbb{P}^1$.
\item There is a basis $e_1,\dots ,e_m$ of $M$ and a non empty Zariski--open subset 
$U\subset \mathbb{P}^1$ such that the restriction of $\mathcal{M}$ to $U$ is the free algebraic vector bundle $O_Ue_1\oplus \cdots \oplus O_Ue_m$.
\item For each $p_i$ one has $\widehat{\mathcal{M}}_{p_i}=\Lambda _i$.
\item For $p\not \in \{p_1,\dots ,p_r\}$ one has $\widehat{\mathcal{M}}_p=\Lambda _p$.
\item $\nabla :\mathcal{M}\rightarrow \Omega (\sum k_i[p_i])\otimes \mathcal{M}$. 
\end{enumerate}  

We still need another ingredient for the proof of Theorem 1.7. Let a differential module $N$ over $K_p=\mathbb{C}(\{t_p\})$ be given and be  written in the form $\nabla :N\rightarrow K_pdt_p\otimes N$.
Choose any $\mathbb{C}\{t_p\}$-lattice $\Lambda \subset N$ and let $k\geq 0$ be such that $\nabla (\Lambda) \subset t_p^{-k}dt_p\otimes \Lambda$. Then the latter map extends to
a connection $(\mathcal{N},\nabla )$, defined on a suitable small disk around $p$ and has the property $\nabla :\mathcal{N}\rightarrow \Omega (k[p])\otimes \mathcal{N}$. We note that this extension depends on the choice of the lattice $\Lambda$ or more precisely on the unique lattice $\Lambda '$ in $\mathbb{C}((t_p))\otimes N$ with
 $\Lambda '\cap N=\Lambda$. }\hfill $\square$ \end{observations}

\noindent {\it Proof of Theorem} 1.7. We use the notation of Definition 1.6. For $r=0$, the data set is empty. The only module $M$ corresponding to
this is the trivial differential module (of the required dimension, say $m$).

 For $r=1$, the data at $p_1$ determines a differential module $M_1$ over $K_{p_1}$.
We choose a lattice $\Lambda \subset M_1$ (say the standard lattice) and then
the connection 
$\nabla _1:\Lambda \rightarrow \Omega (k\cdot [p_1])\otimes \Lambda$ (some $k\geq 0$) extends to a connection $(\mathcal{M}_1,\nabla _1)$ on a small open disk $D$ around 
$p_1$. The topological monodromy around $p_1$ of this connection is trivial. We consider the trivial connection $(\mathcal{M}_0,\nabla _0)$ (of the required 
rank $m$) on $\mathbb{P}^1\setminus \{p_1\}$. The two connections can be glued over 
$D\setminus \{p_1\}$, because of the triviality of $top _1$, and there results a connection
$(\mathcal{M},\nabla )$ on $\mathbb{P}^1$. Its generic fibre $M$ is a differential module
over $K$, inducing the given complete data.

 Let $N$ be another differential module over $K$ inducing the given (formal and analytic) data. Then $K_{p_1}\otimes N$ is isomorphic to $K_{p_1}\otimes M$ and we choose in $K_{p_1}\otimes N$ the lattice which maps to the lattice 
 $\Lambda \subset K_{p_1}\otimes M$. This
 yields a connection $(\mathcal{N},\nabla _N)$ with only $p_1$ as singularity.
Outside $p_1$ the two connections are isomorphic and the same holds above a small enough disk $D$ around $p_1$. The two isomorphisms above $D\setminus \{p_1\}$ will differ by an element in ${\rm GL}_m(\mathbb{C})$ (where $m=\dim M$). The isomorphism between the connections above $\mathbb{P}^1\setminus \{p_1\}$ can be changed by any element in ${\rm GL}_m(\mathbb{C})$. Then, after this change, the two connections are isomorphic and then $N$ is isomorphic to $M$.\\

Now we suppose that $r\geq 2$. The monodromy determines a connection $(\mathcal{M}_0,\nabla _0)$ on $\mathbb{P}^1\setminus \{p_1,\dots ,p_r\}$. The analytic data at $p_i$
determine a differential module over $K_{p_i}$. For this differential module we choose the standard lattice as before. This extends to a connection 
$(\mathcal{M}_i,\nabla _i)$ on a small disk $D_i$ around the point $p_i$.

 Since $top _i$ is conjugated
to the monodromy of the loop $\lambda _i$, we have that the restrictions of 
$(\mathcal{M}_0,\nabla _0)$ and $(\mathcal{M}_i,\nabla _i)$ to $D_i\setminus \{p_i\}$
are isomorphic. A priori, many isomorphisms are possible. However, the link $L_i$ determines the isomorphism. Namely, one takes the isomorphism such that the 
map $V(b)\stackrel{\alpha}{\rightarrow} V(p_i^*)
\stackrel{\beta}{\rightarrow}V(p_i)$ is equal to the given $L_i$, where $\alpha$ is the
analytic continuation for the connection $(\mathcal{M}_0,\nabla _0)$ and 
$\beta$ is the inverse for the multisummation 
$V(p_i)\rightarrow V(p_i^*)$ for the
connection $(\mathcal{M}_i,\nabla _i)$. Glueing yields a connection $(\mathcal{M},\nabla )$
on $\mathbb{P}^1$ and its generic fibre has the required properties.

Consider another differential module $N$ which produces the same (formal and analytic) data. Then $N$ yields a connection $(\mathcal{N},\nabla _N)$. This connection is chosen such that the local connections at the points $p_i$ are standard, as above. 
This connection is, above $\mathbb{P}^1\setminus \{p_1,\dots ,p_r\}$ and above each of the small enough disks $D_i$,
isomorphic to the same items for $(\mathcal{M},\nabla )$. The links $L_i$ imply that these isomorphisms glue to a global isomorphism between 
$(\mathcal{N},\nabla _N)$ and $(\mathcal{M},\nabla )$. Thus $N$ is isomorphic to 
$M$. \hfill $\square$\\

\begin{observations} {\rm  (1) In the construction in the proof of Theorem 1.7 from 
the given formal and analytic data to a connection 
$(\mathcal{M},\nabla )$ on $\mathbb{P}^1$,
one can change the lattices $\Lambda _i$ at the points $p_i$. We note that this corresponds to an elementary transformation in \cite{IIS1}, Section 3.
This will change the connection on $\mathbb{P}^1$. However the corresponding differential module 
does not change.\\
(2) By Proposition 1.4, the links are superfluous in case all the singularities are regular
singular. Another way to see this is to take the standard lattice at each point $p_i$.
Then the glueing of the connection $(\mathcal{M}_0,\nabla _0)$ to the connection
$(\mathcal{M}_i,\nabla _i)$ on a small disk $D_i$ around $p_i$ is {\it unique}, since the
connection  $(\mathcal{M}_i,\nabla _i)$ on $D_i$ and its restriction to 
$D_i\setminus \{p_i\}$ have the same group of automorphisms, namely the elements 
in ${\rm GL}(m,\mathbb{C})$ commuting with the topological monodromy. \\
(3) Theorem 1.7 is the {\it weak solution} of the Riemann--Hilbert problem for differential modules with any type of singularities. }\hfill $\square$ \end{observations}

\begin{definition and examples} The strong Riemann-Hilbert problem.\\ 
{\rm
This problem can be formulated as follows:\\
{\it Let  $M$ be the weak solution for the given formal and analytic data.}\\
{\it  Does there exists a connection $(\mathcal{M},\nabla )$ with generic fibre $M$  and {\rm free} vector bundle $\mathcal{M}$ such that 
 $\nabla :\mathcal{M}\rightarrow \Omega (\sum _p(r^+(p)+1)[p])\otimes \mathcal{M}$?} \\
In the above, the sum $\sum _p$ is taken over the singular points $p$ of $M$, $r(p)$ is the Katz invariant
of $\widehat{K}_p\otimes M$ and $r^+(p)$ is the smallest integer $\geq r(p)$. \\

One observes that in case that all the singularities are regular singular (this means that $r(p)=0$ for every singular point $p$) the above is the classical strong form of the Riemann--Hilbert problem.

In the proof of Theorem 1.7 one can choose at each singular point an invariant lattice
which exists according to Definition and examples 1.2. One arrives at a connection
$(\mathcal{M},\nabla )$ which satisfies all conditions with the exception that  the vector
bundle $\mathcal{M}$ is possibly not free. Under the assumptions of Theorem 1.11,
the invariant lattices can be changed to obtain a free vector bundle. Now we give two
families of examples, closely related to the Painlev\'e equations, where the strong 
Riemann--Hilbert problem has a {\it negative} answer.  \\

\noindent (1) {\it The differential module $M$ is given by the matrix differential equation 
$\frac{d}{dz}+  {0\ f \choose 1\ 0 }$, where $f\in \mathbb{C}[z]$ has degree 3.
For $M$ there is no solution for the strong Riemann--Hilbert problem}.
\begin{proof}
 The only singular point $\infty$ of $M$ has Katz invariant $r=5/2$ and $r^+=3$. {\it Suppose that} $M$ can be represented by a connection $(\mathcal{V},\nabla )$ with 
$\mathcal{V}$ free and $\nabla :\mathcal{V}\rightarrow \Omega (4[\infty ])\otimes \mathcal{V}$. Write $V=H^0(\mathcal{V})$. Then 
$\nabla : V\rightarrow H^0(\Omega (4[\infty ]))\otimes V$ and $\partial :=
\nabla _{\frac{d}{dz}}$ has, with respect to a basis of $V$, the form 
$\frac{d}{dz}+B$ where $B$ is a polynomial matrix of degree $\leq 2$.

For computational convenience we may suppose that $f=f_3z^3+f_1z+f_0$ with $f_3\neq 0$. There exists $A\in {\rm GL}_2(\mathbb{C}(z))$ with  
$A^{-1}(\frac{d}{dz}+{0\ f\choose 1\ 0 })A=\frac{d}{dz}+B$. One easily verifies that
$A\in {\rm GL}_2(\mathbb{C}[z])$ and we may assume that $A$ has determinant 1.
We use the notation ${a\ b \choose c\ d }^t={d\ -b\choose -c \ a }$. Write 
$A=A_0+A_1z+\cdots +A_sz^s$ with constant matrices $A_i$ and $A_s\neq 0$.
Then $A^{-1}=A_0^t+\cdots +A_s^tz^s$ and 
$A^{-1}A'+A^{-1} {0\ f\choose 1\ 0 }A=B$ has degree $\leq 2$. 

Suppose first that $s\geq 2$. We compute the coefficients of $z$-powers in the
expression $A^{-1}A'+A^{-1} {0\ f\choose 1\ 0 }A$.
The coefficient of $z^{2s+3}$ is $A_s^t{0\ f_3\choose 0\ 0}A_s$ is zero and thus
$A_s$ has the form ${*\ *\choose 0 \ 0 }$. The coefficient of $z^{2s+2}$ is then also
zero. The coefficient $A_{s-1}^t{0\ f_3\choose 0 \ 0 }A_{s-1}$ of $z^{2s+1}$ is zero and then $A_{s-1}$ has the form ${*\ *\choose 0\ 0 }$. The coefficient of $z^{2s}$ yields that
$A_s^t{0\ 0\choose 1 \ 0 }A_s=0$. This implies $A_s=0$ in contradiction with the assumption.

In the case $s=1$ one finds again $A_1^t{0\ f_3\choose 0 \ 0 }A_1=0$ 
and observes then that the term 
$(f_3z^3+f_1z)A^{-1} {0\ 1 \choose 0\ 0 }A$ has degree $\leq 2$. This implies that
the term is 0, contradicting that $A$ is invertible. \end{proof}

\noindent {\it Comments}.\\ 
(a) The scalar equation  $(\frac{d}{dz})^2-f$, obtained from $M$ by using
the first basis vector as cyclic vector, has only $\infty$ as singularity, i.e., there is no 
{\it apparent singularity} (see \ref{ss:apparent}). \\
(b) In the above negative answer  for the strong Riemann--Hilbert problem
one can replace $f$ by any polynomial of odd degree $\geq 3$.\\ 
(c) Consider a differential module $M$ of dimension two which has only $\infty$ as singular point
and with $r(\infty )=5/2$ and $\Lambda ^2M\cong {\bf 1}$. Suppose that $M$ can be
represented by a connection $(\mathcal{V},\nabla )$ with $\mathcal{V}\cong
O\oplus O(-2)$. Write more explicitly $\mathcal{V}=Oe_1\oplus O(-2[\infty ])e_2$.
Then $\partial :=\nabla _{\frac{d}{dz}}$ has the form 
$\frac{d}{dz}+{0\ f \choose 1\ 0 }$. One computes that the invariant lattices at
$\infty$ have generators $\{z^n e_1,z^{n-1}e_2 \} $ or $\{z^ne_1, z^{n-2}e_2\}$ over
$\mathbb{C}[[z^{-1}]]$ (with any $n\in \mathbb{Z}$). 
 The connections $(\mathcal{V},\nabla )$  with $\nabla :\mathcal{V}\rightarrow \Omega (4[\infty ])\otimes \mathcal{V}$
representing $M$ are of two types, namely $\mathcal{V}\cong O(n)\oplus O(n-1)$
and $\mathcal{V}\cong O(n)\oplus O(n-2)$.\\
(d) Consider again of differential module $M$ of rank two, $\Lambda ^2M\cong {\bf 1}$,
only $\infty$ as singular point and $r(\infty )=5/2$. Suppose now that the strong
Riemann-Hilbert problem has a {\it positive answer} for $M$. Then $M$ can be represented by a matrix differential equation of the form 
$\frac{d}{dz}+A_0+A_1z+{0\ 1\choose 0\ 0 }z^2$. Using the invariant lattices at
$\infty$ one finds that the vector bundles of the connections $\nabla :\mathcal{V}\rightarrow   \Omega (4[\infty ])\otimes \mathcal{V}$, representing $M$, are of two
types namely isomorphic to $O(n)\oplus O(n)$ or $O(n)\oplus O(n-1)$ (for any
$n\in \mathbb{Z}$. Further $M$ has a cyclic vector (essentially unique) which produces a scalar differential equation with precisely one apparent singularity 
(see \ref{ss:apparent}).\\

\noindent (2) {\it Let $M$ be a 2-dimensional differential module over $\mathbb{C}(z)$
with $\Lambda ^2M={\bf 1}$, $r(0)=r(\infty )=1/2$ and no singularities $\neq 0,\infty$.
Suppose that there exists a connection $\nabla :\mathcal{V}\rightarrow \Omega (2[0]+2[\infty ])\otimes \mathcal{V}$ with generic fibre $M$ and 
$\mathcal{V}\cong O\oplus O(-2)$. Then $M$ can be represented by a matrix differential
equation of the form
\[z\frac{d}{dz}+\left(\begin{array}{cc}0& c_{-1}z^{-1}+c_0+c_1z\\
                                                         1& m \end{array}\right) \mbox{ with }  c_{-1}\neq 0\neq c_1\mbox{ and } m\in \mathbb{Z}.\]
  In particular, the strong Riemann--Hilbert problem has a positive solution for $M$. 
  The {\em special phenomenon} is that the scalar equation, associated to this matrix
  differential equation w.r.t. the first basis vector reads  
  $\delta (\delta -m)-(c_{-1}z^{-1}+c_0+c_1z)$, with $\delta :=z\frac{d}{dz}$, and
  therefore has } no apparent singularities!
                                                      
\begin{proof} Identify $\mathcal{V}$ with $Oe_1\oplus O(-2[\infty ])e_2$. The operator
$\delta :=\nabla _{z\frac{d}{dz}}$ satisfies \\
$\delta e_1\in (\mathbb{C}z^{-1}+\mathbb{C}+\mathbb{C}z)e_1+\mathbb{C}z^{-1}e_2$
and \\
$\delta e_2\in (\mathbb{C}z^{-1}+\mathbb{C}+\mathbb{C}z+\mathbb{C}z^2+\mathbb{C}z^3)e_1+(\mathbb{C}z^{-1}+\mathbb{C}+\mathbb{C}z)e_2$.\\
Since the module is irreducible, we can change $e_2$ into $\lambda e_2+(*+*z+*z^2)e_1$ with suitable $\lambda \in \mathbb{C}^*,\ *\in \mathbb{C}$ and obtain $\delta e_1=z^{-1}e_2$. The condition $\Lambda ^2M={\bf 1}$ implies that
$\delta e_2=(a_{-1}z^{-1}+a_0+a_1z+a_2z^2+a_3z^3)e_1+me_2$ with $m\in \mathbb{Z}$. Conjugation of the corresponding matrix differential equation with ${1\ 0\choose 0\ z }$
yields the matrix differential equation 
$z\frac{d}{dz}+\left(\begin{array}{cc}0& a_{-1}z^{-2}+a_0z^{-1}+a_1+a_2z+a_3z^2\\
                                                         1& m-1\end{array}\right)$. The assumptions
$r(0)=r(\infty )=1/2$ imply $a_{-1}=a_3=0$ and $a_0\neq 0\neq a_2$. This is the required form. The formula for the scalar equation is obvious. \end{proof}  }  \end{definition and examples}

\begin{theorem} Suppose that the differential module $M$ over $K=\mathbb{C}(z)$ is irreducible and
has a {\rm (}regular or irregular{\rm )} singularity  which is unramified.  Then the strong Riemann-Hilbert problem has a solution  for $M$.
\end{theorem}

\begin{proof} We will adapt the proof of Theorem 6.22, \cite{PS} to the present more general situation.
As shown in the proof of Theorem 1.7, there exists a connection $(\mathcal{M},\nabla )$ such that 
$\nabla :\mathcal{M}\rightarrow  \Omega (\sum (k_i+1)[p_i])\otimes \mathcal{M}$, where the $p_i$ are
the singular points of $M$ and $k_i$ is the smallest integer $\geq $ the Katz invariant at $p_i$. The irreducibility of $M$ implies that the 
defect of any $\mathcal{M}$ is bounded by a number only depending on $M$, see
Proposition 6.21, \cite{PS}. 

Let $p=p_1$ be an unramified singular point. 
Then we want to prove the equivalent of Lemma 6.20, \cite{PS}, namely for any integer $N>1$, there 
exists a lattice $\Lambda$ for $\widehat{K}_p\otimes M$ such that $\Lambda$ has a basis 
$e_1,\dots ,e_m$ with the property  $t_p\cdot \nabla e_i=dt_p\otimes ((c_i+a_i)e_i+\sum _{i\neq j}a_{i,j}e_j),$
with $c_i\in \mathbb{C}$; the $a_i\in t_p^{-1}\mathbb{C}[t_p^{-1}]$ belong to the set of the eigenvalues of $M$ at $p$ and $a_{i,j}\in t_p^N\mathbb{C}[[t_p]]$.

It suffices to show the above for an indecomposable direct summand of $\widehat{K}_p\otimes M$. This direct summand has only one eigenvalue and the formal monodromy $\gamma$ has only one Jordan block. Then the  proof of Lemma 6.20, \cite{PS}, yields the required lattice for this indecomposable direct
summand.

 Finally, as in the proof of Theorem 6.22, \cite{PS}, one can change the lattice $\Lambda$ step by step
 to obtain a connection $(\mathcal{M},\nabla )$ where $\mathcal{M}$ has defect 0. Taking the tensor
 product with $O(k[p_1])$ for a suitable $k$ makes $\mathcal{M}$ free. \end{proof}

\section{Families of differential modules }

\subsection{Good families and the monodromy space $\mathcal{R}$}

The aim is to study the formal and analytic data of a {\it family} of differential modules $M(\underline{u})$
depending on some parameters $\underline{u}$. Of course, this notion has to be made explicit. A rough approximation would be a matrix differential equation 
$\frac{d}{dz}+A(z,\underline{u})$ where each entry of the $m\times m$-matrix
$A(z,\underline{u})$ is a rational function in $z$ with coefficients depending analytically
on the parameters $\underline{u}$. We recall that for a point $p$, the local parameter
is $t_p$ (equal to $z-p$ or $z^{-1}$). Further we use the notation of Subsection 1.2:
$V=V(p)$ for the formal solution space at $p$, the eigenvalues are $q_*$ and correspond to subspaces $V_{q_*}$ of $V$. The singular directions at $p$ are defined
in Subsection 1.3. 

In order to have meaningful analytic data (as functions of $\underline{u}$) one  
has to make some  assumptions. A {\it good family} is defined by the properties:
\begin{enumerate}
\item[(1)] The number $r$ of the singular points is fixed. The position of these  points
$\{p_1,\dots ,p_r\}$ may vary, but only slightly.
\item[(2)] For every singular point $p$, the degrees in $t_p^{-1}$ of the eigenvalues $q_i$ and the degrees in $t_p^{-1}$ of the differences $q_i-q_j,\ i\neq j$ is fixed.
\item[(3)] For every singular point $p$ and every eigenvalue $q_i$, the dimension of the space $V_{q_i}$ is fixed.
\item[(4)] The top coefficients $c_{i,j}$ of the $q_i-q_j$ may vary in a 
certain restricted way. Namely, we {\it impose} that any singular 
direction for the singular point $p$ is a singular direction for a 
unique difference $q_i-q_j,\ i\neq j$ of the eigenvalues at $p$ and that
these singular directions vary slightly. As a consequence, the order 
of the directions at $p$ does not change in the family. 
\end{enumerate}

\begin{comments} $\ $ {\rm \\ 
(a) From (1) it follows that one can take a base point $b$ and loops
 $\alpha _1,\dots ,\alpha _r$ around the singular points, valid for 
all $M(\underline{u})$. From a point $p_i^*$ on $\alpha _i$ and close 
to $p_i$, the direction of the line segment $[p_i^*,p_i]$ is non 
singular and lies between the `same' singular directions for the family 
$M(\underline{u})$. This follows from (2), (3) and (4). Suppose that 
the family $M(\underline{u})$ satisfies (1)--(3), and that
for $M(\underline{0})$ every singular direction of each singular point 
belongs to a unique difference $q_i-q_j,\ i\neq j$ of eigenvalues. Then
condition (4) is valid for the restriction of this family to a suitable
neighbourhood  of $\underline{u}=\underline{0}$. \\
(b) Let a differential module $M$ over $K=\mathbb{C}(z)$ be given
 and assume that every  singular direction of each singular point belongs to a unique 
difference of eigenvalues.  {\it We sketch the proof of the 
statement that there exists a local analytic family $M(\underline{u})$,
satisfying (1)--(4), such that $M(\underline{0})=M$}. 

Write $M$ as a matrix differential equation 
$\frac{d}{dz}+\sum _{i=1}^r(\sum _{n=1}^{k_i}\frac{A(i,n)}{(z-p_i)^n})$,
with constant matrices $A(i,n)$. Here, the singular points  
$p_1,\dots ,p_r$ are for notational convenience different from $\infty$ (and thus $\sum _{i=1}^rA(i,1)=0$). We do not impose a condition on $k_i$ in relation with the Katz
invariant at the point $p_i$.  One considers the family 
\[\frac{d}{dz}+A(z,\underline{v}):=\frac{d}{dz}+
\sum _{i=1}^r(\sum _{n=1}^{k_i}\frac{A(i,n)+V(i,n)}{(z-p_i-v_i)^n})\ ,\]
 where the $V(i,n)$ are matrices of indeterminates  and the $v_i$
are also indeterminates. Let $\underline{v}$ denote the collection of
 all indeterminates. We consider this family in a small enough 
polydisk $D$ around $\underline{0}\in \mathbb{C}^N$. The conditions 
(1)--(3) on $\frac{d}{dz}+A(z,\underline{v})$ define a Zariski closed subset $Z\subset 
\mathbb{C}^N$. The subfamily $\frac{d}{dz}+A(z,\underline{u})$ with
$\underline{u}$ belonging to the locally closed set $U=D\cap Z\neq \emptyset$ satisfies
conditions (1)--(3). Further condition (4) is satisfied for this subfamily since $D$ small enough.  

A priori, 
$\underline{0}$ is a singular point of $U$. If needed, one can, by 
resolution of singularities, return to the case that $U$ is a small 
open polydisk around the point $\underline{0}$.\\
(c)  The statement in (b)  justifies the naive way of writing a family, satisfying (1)--(4), locally as  $\frac{d}{dz}+A(z,\underline{u})$.  As mentioned in the introduction, the theory
of Okamoto--Painlev\'e pairs has the aim to improve on this. In this paper however, we will deal with the naive local situation. }\hfill $\square$ \end{comments}

Let $M(\underline{u})$ be a good family of dimension $m$ for $\underline{u}$, close to $\underline{0}$. According to Definition 1.6, the {\it formal data} of the family are the position of the singular points  $\{p_j\}_{j=1}^r$ and the eigenvalues $q$ at the singular point $p_j$. Now we describe the items which do not vary in the family:
\begin{itemize}
\item[(a)] $V(j)$, the formal solution space at $p_j$.
\item[(b)] The direct sum decomposition $V(j)=\oplus _{i\in I_j}V(j,i)$, given by the eigenvalues.
\item[(c)] The dimension of the spaces $V(j,i)$. 
\item[(d)] The permutation $\tau _j$ of the $V(j,i)$, satisfying $\dim V(j,\tau _ji)=\dim V(j,i)$, induced by the action of $\gamma$ on the eigenvalues.
\item[(e)] The order of the singular directions for any $p_j$. This yields
a sequence of Stokes maps $\{ St_k(j) \}_{k=1}^{n_j}$.
\item[(f)]  Each $St_k(j)$ has the form $1+R_k$, with $R_k:= i_t    \circ M(j,k) \circ  pr_s$ with prescribed $s,t\in I_j,\ s\neq t$ (depending on $k$) and $pr_s$ is the projection $V(j)\rightarrow V(j,s)$
(with kernel $\oplus _{h\neq s}V(j,h)$) and $i_t:V(j,t)\rightarrow V(j)$ is the canonical
injection. Moreover,  $M(j,k):V(j,s)\rightarrow V(j,t)$ is a linear map which is not constant
in the family.  
\item[(g)] A vector space $W$ of dimension $m$, representing the solution space $V(b)$ for a given base point $b$.
\end{itemize}

The {\it analytic data} of $M(\underline{u})$ are tuples
 $(\{\gamma _j\}, \{L_j\}, \{St_k(j)\} )$ satisfying:
\begin{enumerate}
\item For each $j$, a map $\gamma _j\in {\rm GL}(V(j))$ with 
$\gamma _j(V(j,i))=V(j,\tau _ji)$ for all $i$.
\item $St_k(j)\subset {\rm GL}(V(j))$ of the form described in (f), determined
by the linear map $M(j,k): V(j,s)\rightarrow V(j,t)$. 
\item Linear bijections  $L_j:W\rightarrow V(j)$ for $j=1,\dots ,r$. 
\item Define $top_j:=\gamma _j\circ St_{n_j}(j)\circ \cdots \circ St_1(j)$.
The data should satisfy the relation
$L_r^{-1}\circ top_r\circ L_r \dots \circ L_1^{-1}\circ top _1 \circ L_1=1.$
\end{enumerate}

Let $AnalyticData$ denote the set of all tuples. This has a natural structure of an
affine variety over $\mathbb{C}$. Two tuples $(\{\gamma _j\}, \{L_j\}, \{St_k(j)\} )$ and 
$(\{\gamma _j'\}, \{L_j'\}, \{St_k(j)'\} )$ are called {\it equivalent}, if there
exists $\sigma _j\in {\rm GL}(V(j)),\ j=1,\dots ,r,\ \sigma \in {\rm GL}(W)$
such that each $\sigma _j$ preserves the direct sum decomposition $\oplus V(j,i)$
and  further $\sigma _j\circ L_j = L_j'\circ \sigma ,\ j=1,\dots ,r$ and 
$\sigma _j\circ \gamma _j=\gamma _j'\circ \sigma _j,\ j=1,\dots ,r$. In other words,
the equivalence relation on $AnalyticData$ is given by the action of the reductive linear
algebraic group $G:={\rm GL}(W)\times \prod _{j,i} {\rm GL}(V(j,i))$.\\
 
The  {\it monodromy space} $\mathcal{R}$  is by definition $Analytic Data / /G$,
the categorical quotient. This is again an affine variety.  In general this quotient is not a geometric one.  In particular, a closed point of $\mathcal{R}$ can correspond to many equivalence classes.  One may use $L_1$ to identify each space $V(1)$ with $W$ to reduce the space $AnalyticData$ and the group $G$ acting on it. \\

The map, which associates to $\underline{u}$ in $D$ (a small polydisk around 
$\underline{0}$) the tuple $(\{\gamma _j\}, \{L_j\}, \{St_k(j)\} )$, is analytic. Indeed, it is rather clear that analytic continuation depends in an analytic way on $\underline{u}$. That the same is valid for multisummation follows from \cite{PS}, Proposition 12.20, p. 314. Thus $D \rightarrow AnalyticData$ is analytic and hence $D\rightarrow \mathcal{R}:=AnalyticData/ / G$ is analytic. The next example illustrates the above for a
relatively simple case.

\begin{example}The monodromy space $\mathcal{R}$ for the local  family 
$M(\underline{u})$ with $M(\underline{0})$ given by the matrix equation 
\[z\frac{d}{dz}+ 
\left(\begin{array}{ccc}  z^{-1} +a_1& 0 & 0\\
                            0& \omega z^{-1}+a_2&0\\ 0&0& \omega ^2z^{-1}+a_3\end{array}\right), \mbox{ where } \omega =e^{2\pi i/3} .\]
{\rm A good choice (compare \cite{PS}, 12.3) for the family $M(\underline{u})$ is 
\begin{footnotesize}
$$
z\frac{d}{dz}+ 
\left(\begin{array}{ccc}  (1+u_1)z^{-1} +a_1+u_2& u_3 & u_4\\
                            u_5& (1+u_6)\omega z^{-1}+a_2+u_7&u_8\\ u_9&u_{10}& (1+u_{11})\omega ^2z^{-1}+a_3+u_{12}\end{array}\right)  .
$$
\end{footnotesize}

The singular points are $z=0,\ \infty$ and $r(0)=1,\ r(\infty)=0$.  
The space $AnalyticData$ consist of the formal monodromy and six Stokes matrices at $0$, the link between $0$ and $\infty$ and the formal (=topological) monodromy at $\infty$.  This link and the topological monodromy at $\infty$ are determined by the data at $0$ up to an automorphism of the solution space at $\infty$. 

The eigenvalues 
at 
$z=0, \underline{u}=\underline{0}$ are $q_1=z^{-1},\ q_2=\omega z^{-1},\ q_3=\omega ^2z^{-1}$. 
\par\noindent
The order of the six singular directions in $\mathbb{R}/2\pi \mathbb{Z}$  is given by the differences $q_1-q_2, q_1-q_3, q_2-q_3, q_2-q_1, q_3-q_1, q_3-q_2$.  The topological monodromy $top _0$ at $z=0$ is then the following product of matrices
\begin{footnotesize}
\[\left(\begin{array}{ccc} \alpha _1 &0 &0 \\0 &\alpha _2 &0  \\ 0 &0 & \alpha _3\end{array}\right)\cdot
\left(\begin{array}{ccc}1 &0 &0 \\ 0 &1 & c_6 \\ 0 &0  &1 \end{array}\right)\cdot
 \left(\begin{array}{ccc} 1 &0  & c_5 \\ 0 & 1 & 0 \\ 0 & 0& 1\end{array}\right)\cdot 
\left(\begin{array}{ccc} 1 & c_4 &0 \\ 0& 1& 0\\  0& 0& 1\end{array}\right) \] \[
\cdot \left(\begin{array}{ccc}  1& 0& 0 \\ 0 &1 & 0 \\ 0 & c_3 &1 \end{array}\right)\cdot
\left(\begin{array}{ccc} 1 &0 &0 \\ 0 & 1& 0 \\ c_2& 0& 1\end{array}\right)\cdot 
\left(\begin{array}{ccc}  1&0 &0 \\ c_1 &1 & 0 \\ 0& 0& 1\end{array}\right) .\]
\end{footnotesize}
The entries $(\alpha _1,\alpha _2,\alpha _3)$ of the first matrix, the formal monodromy, are 
\[ (\alpha _1,\alpha _2,\alpha _3)=(e^{2\pi i(a_1+u_2) },e^{2\pi i(a_2+u_7) },e^{2\pi i (a_3+u_{12}) }).\]

 The
$c_1,\dots ,c_6$ are analytic functions of $\underline{u}$, produced by multisummation
(in this case just Borel summation). 
The topological monodromy $top_\infty$ at $\infty$ is
conjugated to $top_0$.
 Under the condition that $a_i-a_j\not \in \mathbb{Z}$ for $i\neq j$, one has that $top_\infty$ is equal to $e^{2\pi iA}$ with
 \begin{footnotesize}
\[ A= 
\left(\begin{array}{ccc}  a_1+u_2& u_3 & u_4\\
                            u_5& a_2+u_7&u_8\\ u_9&u_{10}& a_3+u_{12}\end{array}\right)  .\]
\end{footnotesize}

The group 
\[G:=\{ \left(\begin{array}{ccc} t_1&0&0\\ 0&t_2&0\\ 0&0& t_3\end{array}\right) \ |\
t_1t_2t_3=1\} \subset {\rm GL}(V(0))\]
acts, by conjugation, on the data $(\alpha _1,\alpha _2,\alpha _3,c_1,\dots ,c_6)$. 
Thus $\mathcal{R}$ is the affine space with coordinate ring 
$\mathbb{C}[\alpha _1,\alpha _1^{-1},\dots ,
\alpha _3,\alpha _3^{-1},c_1,\dots ,c_6]^G$. A computation shows that this ring
is $\mathbb{C}[\alpha _1,\alpha _1^{-1},\dots ,\alpha _3,\alpha _3^{-1},x_{14},x_{25},
x_{36}, x_{135},x_{246}]$ where the only relation is $x_{135}x_{246}-x_{14}x_{25}x_{36}
=0$. Here $x_{14}=c_1c_4,\ x_{25}=c_2c_5,\ x_{135}=c_1c_3c_5$ et cetera. It is natural to see the eigenvalues
$(\alpha _1,\alpha _2 ,\alpha _3)$ of the formal  monodromy as a parameter space  $\mathcal{P}$. The fibers of $\mathcal{R}\rightarrow \mathcal{P}$ are isomorphic to the 4-dimensional affine space with coordinate ring  $\mathbb{C}[ x_{14},x_{25}, x_{36}, x_{135},x_{246}]$ with relation  $x_{135}x_{246}-x_{14}x_{25}x_{36}=0$. The singular locus of this space has three components, they are the image of the locus where the
differential equation is reducible $\{\underline{u}|\ M(\underline{u})\mbox{ is reducible}\}$.

The group $G$ also acts on the local family $M(\underline{u})$ and we
obtain a local Riemann--Hilbert morphism 
$RH:\{\underline{u}\in \mathbb{C}^{12}| \ \|\underline{u}\| <\epsilon \}// G\rightarrow \mathcal{R}$. This map does not depend on 
the coefficients $(1+u_1), (1+u_6)\omega ,(1+u_{11})\omega ^2$ of $q_1,q_2,q_3$.
The fibres of $RH$ are, by definition, the isomonodromic families. They are parametrized
by the three variables $t_1=:u_1,\ t_2:=u_6,\ t_3:=u_{11}$. Using  
$z\mapsto  \lambda z$, one may normalize to $(1+u_{11})=1$ and thus the isomonodromic family is parametrized by $t_1,t_2$.  One expects that a
suitable expression in the other $u_i$ satisfies a Painlev\'e type of partial differential
equations w.r.t. the variables $t_1,t_2$.  In fact it is possible to convert the situation into 
a one variable case for  PVI (cf. \cite{Boalch}). }\hfill $\square$ \end{example}

\begin{remarks} The papers of M.~Jimbo, T.~Miwa and K.~Ueno.\\
{\rm The above introduction of families of differential modules and their formal and analytic data can be
seen as an extension of the papers \cite{JMU}, \cite{JM}, which we will describe now, using our terminology.

 In \cite{JMU}, \cite{JM} the base point $b$ is taken to be $\infty$ and this point is supposed to
be (irregular) singular. Further the irregular singularities $p$ are of a simple kind, namely all
the eigenvalues (generalized exponents) $q_i$ are  in $t_p^{-1}\mathbb{C}[t_p^{-1}]$, and all
$q_i$ and $q_i-q_j$ for $i\neq j$ have the same degree in $t_p^{-1}$ (there is one exception, related to the Painlev\'e I equation).  We note that Example 2.2 is of the type considered in \cite{JMU}.  In particular, Borel summation or, better, $k$-summation is sufficient  for the asymptotic analysis of the singularity. 
A theorem of Y.~Sibuya \cite{Sib} gives the required
input from asymptotics.  The `links', that we defined, are present in their work and the family of linear differential  equations is presented as a matrix differential equation $\frac{d}{dz}+A(z,\underline{u})$. 
The origin of the examples, in the appendix of \cite{JM}, of families related to Painlev\'e I--VI, is probably classical (discovered by R.~Fuchs \cite{F}, P.~Painlev\'e \cite{P}, R.~Garnier \cite{Garnier} et al.). Another source for similar examples are the ones discovered by H.~Flaschka and A.C.~Newell \cite{FN}.  Later work of  B.~Malgrange \cite{Mal1, Mal2}, clarifies and extends the papers \cite{JMU}, \cite{JM}. 

\smallskip
The new tool `multisummation' and the precise construction of the Stokes matrices, enables to generalize the work of 
Jimbo, Miwa and Ueno. Especially, as we will show in the next subsection, a `complete' list of the equations related to 
Painlev\'e I--VI can be derived. Further, the monodromy spaces
$\mathcal{R}$ for the analytic data can now be computed 
and studied  in detail. }\hfill $\square$ \end{remarks}

\subsection{Finding the list. Tables for connections and $\mathcal{R}$}

 We consider a local family $M(\underline{u})$ of differential modules, represented by a
  matrix equation  $\frac{d}{dz}+A(z,\underline{u})$ where $A(z,\underline{u})$  is a
   $2\times 2$-matrix 
 with trace 0 and $\underline{u}$ lies in a small polydisk $D$ around $\underline{0}$. 
 The possibilities of the formal structure at the singular points is given in Examples 1.1. 
  The local Riemann--Hilbert map  $RH:D\rightarrow  \mathcal{R}$ forgets the formal
   data, namely the position of the set of singular points $S$ and the coefficients of
    the eigenvalues at the irregular singular points. 

The position of the points $S$ contributes  $\max (-3+\#S,0)$ to the dimension of the fibre, because of the automorphisms of $\mathbb{P}^1$. A singular point $p$
with Katz invariant $r(p)$ contributes to the fibre the dimension $r(p)$ if $r(p)\in \mathbb{Z}_{\geq 0}$ and $r^+(p)=\frac{1}{2}+r(p)$ if $r(p)\in \frac{1}{2}+\mathbb{Z}_{\geq 0}$. 
Further, in the space $D$ one divides by the action of a subgroup of the automorphisms of $\mathbb{P}^1$.\\
{\it The requirement  that the fibres of $RH$ have dimension 1 produces the list}:\\
 $\# S>4$ is excluded.  \\
 $\#S=4$, then  $S=\{0,1,\infty ,t\}$, only regular singular points, i.e., all $r(p)=0$. \\
 $\#S=3$, then $S=\{0,1,\infty \}$ and only one irregular point with $r(p)\in \{1,\frac{1}{2}\}$. 
 $\#S=2$, then $S=\{0,\infty \}$, the contribution of the singular points to the dimension of the fibre is 2, 
 since we divide by the group $z\mapsto a z$.\\
$\#S=1$, then $S=\{\infty \}$, the contribution of the singular point to the dimension of the fibre is 3, since we divide by the group
 $z\mapsto az+b$.\\

Columns 3--6 of the next table present the ten resulting cases. In the second 
column one finds the classification of the related Painlev\'e equation (some classes 
are divided into subclasses). The first column gives the extended Dynkin diagram of the corresponding Okamoto--Painlev\'e pair (see the Introduction). The space $\mathcal{R}$ is mapped to a space of parameters $\mathcal{P}$ (related to the parameters spaces of the  Painlev\'e equations) consisting of traces or eigenvalues of the 
matrices involved in the construction of $\mathcal{R}$.

\begin{table}[h]

\begin{center}
\begin{tabular}{|c|c||c|c|c|c|c|c|} \hline 
Dynkin &Painlev\'e equation & $r(0)$ & $r(1)$ & $r(\infty)$ & $r(t)$ & 
$\dim \mathcal{P}$ \\ \hline 
 $\tilde{D}_4$ & PVI &  0 & 0 & 0 & 0 & 4 \\ \hline 
$\tilde{D}_5$ & PV & 0 & 0 & 1 & - & 3 \\ \hline 
$\tilde{D}_6$  & ${\rm PV}_{\rm deg}$= PIII(D6) & 0 & 0 & 1/2 & - & 2  \\ \hline 
$\tilde{D}_6$ & PIII(D6) & 1 & - &  1 & - & 2  \\ \hline  
$\tilde{D}_7$  & PIII(D7) & 1/2& - & 1 & - & 1 \\ \hline 
$\tilde{D}_8$  & PIII(D8) & 1/2 & - & 1/2 & - & 0  \\ \hline  
$\tilde{E}_6$ & PIV     & 0 & - & 2 & - & 2 \\ \hline
$\tilde{E}_7$ & PII & 0 & - & 3/2 & - & 1\\ \hline  
$\tilde{E}_7$ & PII & - & - & 3 & - & 1 \\ \hline 
$\tilde{E}_8$ & PI  &- & - & 5/2 & -& 0  \\ \hline 
\end{tabular}
\end{center}
\vspace{0.5cm}
\caption{Classification of Families}
\label{tab:dynkin}
\end{table}

We will not number these ten families, but indicate them by their  Katz invariants,
e.g., $(0,0,0,0),\ (0,0,1),\dots ,\ (0,-,3/2 ),\ (-,-,3),\ (-,-,5/2)$. For a differential module $M$ corresponding to one of the types, the strong Hilbert-Riemann problem has a positive answer, except possibly for $(1/2,-,1/2)$ and $(-,-,5/2)$ (see Definitions and examples 1.10). For these two types we only consider the modules $M(\underline{u})$ for which the strong Riemann--Hilbert problem does have a positive answer.  The strong Riemann-Hilbert problem for a family $M(\underline{u})$ is more subtle. It seems that connections on the vector bundle  $O\oplus O(-1)$ is better adapted to families. For
the Painlev\'e VI case this is type of vector bundle is considered in \cite{IIS1,IIS2}.
Here however we will represent  a family $M(\underline{u})$ by connections on $O\oplus O$. This defines, in general, an affine Zariski open subset of the space of all
connections. However, the monodromy space $\mathcal{R}$ classifies the analytic
data (modulo some equivalence) for the complete space of all connections.
For each type there are many possibilities. We will make choices that are helpful for the computation of the Painlev\'e equations and are moreover close to classical formulas. 

 It turns out that $\mathcal{R}\rightarrow \mathcal{P}$ is a {\it family of affine cubic surfaces}. There are two sources for the singularities of the fibres. The first 
one is {\it reducibility} of systems and is connected with the singularities of $\mathcal{R}$ itself. The other source is  {\it resonance}, i.e., at least one of the matrices involved  in the construction of $\mathcal{R}$ has a difference of eigenvalues belonging to $\mathbb{Z}\setminus \{0\}$. Section 3 provides the computations of the families  $\mathcal{R}\rightarrow \mathcal{P}$.    We will also describe the corresponding one-dimensional families of 
differential modules $M(t)$. This subsection ends with a list indicating the families of connections and presenting the families $\mathcal{R}\rightarrow \mathcal{P}$ of affine cubic surfaces by an equation. The monodromy space $\mathcal{R}$ for
$(0,0,0,0)$ is classical(cf.\cite{FK65,Iw1}), the others seem to be new.

\bigskip
\noindent 
{\bf (0,0,0,0)}.\  \ PVI. \hfill $\frac{d}{dz}+\frac{A_0}{z}+\frac{A_1}{z-1}+\frac{A_t}{z-t}, \mbox{ all } tr(A_*)=0. $\\ 

\noindent 
$ x_1 x_2 x_3 + x_1^2 + x_2^2 + x_3^2 - s_1 x_1 - s_2 x_2- s_3 x_3 +  s_4=0, $ with \\
$s_i  =  a_i a_4 + a_j a_k , \quad (i, j, k) = \mbox{a cyclic permutation of  }  (1, 2, 3),
$\\ 
$s_4  =  a_1a_2 a_3 a_4 + a_1^2 + a_2^2 + a_3^2 + a_4^2 - 4 
\mbox{ with } a_1, a_2, a_3, a_4 \in \mathbb{C}.$
\bigskip 

\noindent  {\bf (0,0,1)}. \ \  {PV}. \hfill  $\frac{d}{dz} +\frac{A_0}{z}+\frac{A_1}{z-1}+
t/2\cdot {-1\ \  0\choose 0 \ 1 }  , \mbox{ all } tr(A_*)=0.$\\

\noindent 
$x_1x_2x_3+x_1^2+x_2^2-(s_1+s_2s_3)x_1-(s_2+s_1s_3)x_2-s_3x_3+s_3^2+
s_1s_2s_3+1=0$ with $s_1,s_2\in \mathbb{C},\ s_3\in \mathbb{C}^*$.

\bigskip
\noindent {\bf (0,0,1/2)}.\ \ {${\rm PV}_{\rm deg}$}. \hfill $ \frac{d}{dz}+\frac{A_0}{z}+\frac{A_1}{z-1}+
{0\ t^2 \choose 0 \ \ 0},
\mbox{ all }tr (A_*)=0.  $\\

\noindent 
  $x_1 x_2 x_3+x_1^2+x_2^2 +s_0 x_1+s_1 x_2+1=0$ with $s_0, s_1 \in \mathbb{C}$. 

\bigskip
 \noindent {\bf (1,-,1)}.\ \  {PIII(D6)}. \hfill  $z\frac{d}{dz}+A_0z^{-1}+A_1+{\frac{t}{2}\ \ 0\choose 0\ -\frac{t}{2} }z,\mbox{ all }
 tr(A_*)=0.  $\\
 
 \noindent 
  $ x_1x_2x_3+x_1^2+x_2^2+(1+\alpha \beta )x_1+(\alpha +\beta )x_2
  +\alpha \beta =0$ with $\alpha ,\beta \in \mathbb{C}^*$.

\bigskip
 \noindent {\bf (1/2,-,1)}.\ \  {PIII(D7)}. \hfill $z\frac{d}{dz}+ A_0z^{-1}+A_1+{\frac{t}{2} \ 0\choose 0 \ -\frac{t}{2} }z,
 \mbox{ all } tr(A_*)=0. $\\ 
 
 \noindent 
 $x_1x_2x_3+x_1^2+x_2^2+\alpha x_1+ x_2 =0$ with $\alpha \in \mathbb{C}^*$.

\bigskip

 \noindent  {\bf (1/2,-,1/2)}.\ \  {PIII(D8)}. \hfill $z\frac{d}{dz}+{0 \ 0\choose -q \ 0} z^{-1} +
 {\frac{p}{q}\ -\frac{t}{q} \choose 1\ -\frac{p}{q} }+{0\ 1 \choose 0 \ 0 }z.  $\\

   \noindent $x_1x_2x_3+x_1^2-x_2^2-1  =0$. 
   
   \bigskip
   
 \noindent {\bf (0,-,2)}.\ \ {PIV}. \hfill $z\frac{d}{dz}+A_0+A_1z+{1\ \ 0\choose 0\ -1 }z^2.   $\\
  
  \noindent 
  $x_1x_2x_3+x_1^2-(s_2^2+s_1s_2)x_1-s_2^2x_2-s_2^2x_3+s_2^2+s_1s_2^3=0 $ 
with $s_1\in \mathbb{C},\ s_2\in \mathbb{C}^*$.

\bigskip

  \noindent {\bf (0,-,3/2)}.\ \  {PIIFN}. \hfill $z\frac{d}{dz}+A_0+{0\ t+q \choose 1\ \ \ 0}z+{0\ 1\choose 0\ 0 }z^2  $\\
  
  \noindent $x_1x_2x_3+x_1-x_2 + x_3 +s=0 $, with $s \in \mathbb{C}$. 
  
  \bigskip
  
 \noindent {\bf (-,-,3)}.\ \  {PII}. \hfill $\frac{d}{dz}+A_0+A_1z+{1\ \ 0\choose 0\ -1 }z^2,
  \mbox{ all } tr(A_*)=0.  $\\  

\noindent  $x_1x_2x_3-x_1 -\alpha x_2- x_3+\alpha +1 =0$ with $\alpha \in \mathbb{C}^*$.
\bigskip

\noindent {\bf (-,-,5/2)}.\ \  {PI}. \hfill $\frac{d}{dz}+{p\ t+q^2\choose -q\ \ \ -p }+{0\ q\choose 1\ 0 }z
+{0\ 1\choose 0\ 0 }z^2.   $\\ 

\noindent $x_1x_2x_3+x_1+x_2+1=0$.\\
\vspace{0.5cm}
 
 Table of the equations of  the monodromy spaces for the 10 families.\\ 


\section{Computation of the monodromy spaces}

\subsection{Family $(0,0,0,0)$ and Painlev\'e PVI} 
For completeness we describe this classical family. The family of 
differential modules is represented by 
$\frac{d}{dz}+A(z,t):=\frac{d}{dz}+\frac{A_0}{z}+\frac{A_1}{z-1}+\frac{A_t}{z-t}$ with 
constant $2\times 2$ matrices having trace 0. Dividing by  the action,
by conjugation, of ${\rm PSL}_2$ one finds a moduli space $\mathcal{M}$
(say the categorical quotient) of differential modules with dimension 
$7$. 

The monodromy data are given by the tuples 
$(M_1,M_2,M_3,M_4)\in {\rm SL}_2^4$ satisfying $M_1\cdots M_4=1$. This 
defines an affine space of dimension 9. The categorical quotient 
$\mathcal{R}$ of this space under the action, by conjugation with 
${\rm PSL}_2$, has dimension 6. The fibres of 
$RH:\mathcal{M}\rightarrow \mathcal{R}$ are  parametrized by
$t\in \mathbb{P}^1\setminus \{0,1,\infty \}$.

The coordinate ring of $\mathcal{R}$ is generated over $\mathbb{C}$ by 
  $x_1,x_2,x_3,a_1,a_2,a_3,a_4$ with $a_i=tr(B_i)$ and 
$x_1=tr(B_2B_3),\ x_2=tr(B_1B_3),\ x_3=tr(B_1B_2)$. There is only one 
relation (\cite{FK65, Iw2}), namely (as in the list)
 \[ x_1x_2x_3+x_1^2+x_2^2+x_3^2-s_1x_1-s _2x_2
-s_3x_3+s_4=0\ .\]
 The morphism $\mathcal{R}\rightarrow \mathcal{P}:=\mathbb{C}^4$, given by 
$(x_1,\dots ,a_4)\mapsto (s _1,\dots ,s _4)$, is a family of
 affine cubic surfaces with `three lines at infinity'. For information concerning the singularities of $\mathcal{R}$ and of the fibres  we refer to [\cite{Iw1}, \cite{Iw2}, \cite{IISA}].

\subsection{Family $(0,0,1)$ and Painlev\'e  PV}

For a differential module  of type $(0,0,1)$, the strong Riemann-Hilbert problem has
a positive answer. Indeed, the lattices at $0$ and $1$ can be chosen arbitrary. By tradition one supposes that the corresponding local exponents are $\pm \theta _0/2$
and $\pm \theta _1/2$. From Definition and examples 1.2 one concludes that there exists a unique lattice 
at $\infty$ leading to a free vector bundle. By tradition, the generalized local exponents
at $\infty$ are $\pm (t z+\theta _\infty)/2$. The module is then represented by the
matrix differential equation $\frac{d}{dz}+\frac{A_0}{z}+\frac{A_1}{z-1}+A_\infty$, for certain constant matrices  $A_0,A_1,A_\infty$ with trace 0.
  $A_\infty$ is normalized by $A_\infty =t/2\cdot {-1\  0\choose 0\ \ 1}$ with 
$t\in \mathbb{C}^*$. Further $-\theta _i^2/4$ is the determinant of $A_i$ for $i=0,1$ and, $\theta _\infty$
is the $(1,1)$ entry of $A_0+A_1$.

\subsubsection{The moduli space $\mathcal{R}$ for the analytic data}
 The symbolic solution space $V$ at $\infty$ is written as $V_q\oplus V_{-q}$. Let $e_1, e_2$ be basis
 vectors for $V_q$ and $V_{-q}$. Starting at $\infty$ one makes loops around $0$ and $1$, producing 
 monodromy matrices $M_1,M_2$, with respect to the basis $\{e_1,e_2\}$. Let $M_\infty$ be the topological
 monodromy at $\infty$, then we have the relation $M_1M_2M_\infty =1$.
Further $M_\infty ={\alpha \ \ \  0\choose 0\ \ \alpha ^{-1} }{1\ \ 0\choose f_1\ \ 1 }{1\  \ f_2\choose 0\ \ 1 }$,
where the first matrix is the formal monodromy and the others are the two Stokes matrices. In the sequel, we will
eliminate the choice of the basis vectors $e_1$,  $e_2$ of $V_q$ and $V_{-q}$ for the matrix equation.
 One considers the isomorphism  $\mathbb{C}^*\times \mathbb{C}\times \mathbb{C}\rightarrow 
\{ {a\ b \choose c\ d }\in {\rm SL}_2\ | \ a\neq 0\}$, given by 
\[(\alpha , f_1,f_2)\mapsto {\alpha \ \ \ 0\choose 0\ \ \alpha ^{-1} }{1\ \ 0\choose f_1\ \ 1 }{1\ \ f_2\choose 0\ \ 1 }
=
\left(\begin{array}{cc}\alpha   &\alpha f_2  \\ \alpha ^{-1}f_1  &\alpha ^{-1}(1+f_1f_2)  \end{array}\right)\  .\]
One concludes that the matrices $M_j={a_j\  \ b_j\choose c_j\  \ d_j }\in {\rm SL}_2$ for $j=1,2$ determine
$M_\infty$ and that $c_1b_2+d_1d_2=\alpha $ is non zero. Therefore, the pairs $M_1,M_2$ that occur
define an affine variety with coordinate ring  
\[R=\mathbb{C}[a_1,\dots ,d_2,\frac{1}{c_1b_2+d_1d_2}]/(a_1d_1-b_1c_1-1,
a_2d_2-b_2c_2-1)\ .\]
The group $\mathbb{G}_m=\{{c\ 0\choose 0\ 1 }|\ c\in \mathbb{C}^*\}$ acts on $V$ and thus on the matrices
$M_1,M_2$. For this action the weights are: $+1$ for $b_1,b_2$; $-1$ for $c_1,c_2$ and $0$ for
 $a_1,d_1,a_2,d_2$. The subring $R_0$ of $R$, consisting of the invariants under the action of $\mathbb{G}_m$ is the subring
consisting of the elements of weight 0. The moduli space $\mathcal{R}$ for the analytic data is
${\rm Spec}(R_0)$.

For the calculation of $R_0$ we may at first forget the localization
at the degree 0 element $c_1b_2+d_1d_2$. Now, using the two relations,
we find that $R$ has a basis over $\mathbb{C}$, consisting of the monomials
\[ a_1^*a_2^*d_1^*d_2^*b_1^{n_1}b_2^{n_2}c_1^{m_1}c_2^{m_2}
\mbox{ with } n_1m_1=0,\ n_2m_2=0\mbox{ and any integers }  *\geq 0 .\]
It follows that $R_0$ is equal to
$\mathbb{C}[a_1,d_1,a_2,d_2,b_1c_2,b_2c_1,\frac{1}{b_2c_1+d_1d_2}]$, where the
six generators have only one relation namely
$b_1c_2\cdot b_2c_1=(-1+a_1d_1)(-1+a_2d_2)$. The {\it singular locus of $\mathcal{R}$}
is given by the additional equations $0=b_1c_2=b_2c_1=a_1d_1-1=a_2d_2-1$. One observes that
this describes the {\it reducible analytic data}, given by  $b_1=b_2=0$ or $c_1=c_2=0$ (or equivalently
the corresponding reducible differential equations).  The coordinate ring of the singular locus of $\mathcal{R}$ is  $\mathbb{C}[d_1,d_1^{-1},d_2,d_2^{-1}]$.

\bigskip
Introduce new variables $s_1:=a_1+d_1,\ s_2:=a_2+d_2,\ s_3:=b_2c_1+d_1d_2$, i.e., the traces of $M_1,M_2$ and the eigenvalue $\alpha$ of the formal
monodromy at $\infty$ and the new variable $d_3:=b_1c_2+d_1d_2-s_2d_1-s_1d_2+s_1s_2+s_3$. Exchange these variables against $a_1,a_2,b_2c_1$
and $b_1c_2$. Then the ring $R_0$ obtains the form
$R_0=\mathbb{C}[d_1,d_2,d_3,s_1,s_2,s_3,s_3^{-1}]/(R(s_1,s_2,s_3))$, where $R(s_1,s_2,s_3)$ is equal to
\[d_1d_2d_3+d_1^2+d_2^2-(s_1+s_2s_3)d_1-(s_2+s_1s_3)d_2-s_3d_3 +s_3^2+s_1s_2s_3+1 \ .\]
In the sequel we will write $x_i=d_i$ for $i=1,2,3$.
The inclusion 
\[\mathbb{C}[s_1,s_2,s_3,s_3^{-1}]\subset \mathbb{C}[x_1,x_2,x_3,s_1,s_2,s_3,s_3^{-1}]/(R(s_1,s_2,s_3)) \]
induces a surjective morphism
\[ \pi :\mathcal{R}\rightarrow  \mathcal{P}=\mathbb{C}\times \mathbb{C}\times \mathbb{C}^*={\rm Spec}(\mathbb{C}[s_1,s_2,s_3,s_3^{-1}]) \ ,\]
which maps a given tuple  $(M_1,M_2,M_\infty )$ to $(s_1,s_2,s_3)$. Thus $\pi :\mathcal{R}\rightarrow \mathcal{P}$ is a family of affine cubic surfaces with equation $F=0$ with
 \[F=x_1x_2x_3+x_1^2+x_2^2-(s_1+s_2s_3)x_1-(s_2+s_1s_3)x_2-s_3x_3 +s_3^2+s_1s_2s_3+1\ .\]
We note that this equation (or the cubic surface) has a symmetry, given by 
interchanging $(x_1,s_1)$ and $(x_2,s_2)$ (i.e., interchanging $M_1,M_2$).

\subsubsection{The singularities of $\mathcal{R}$ and the fibres of $\mathcal{R}\rightarrow \mathcal{P}$}
The inclusion of the singular locus ${\rm Spec}(\mathbb{C}[x_1,x_1^{-1},x_2,x_2^{-1}])$ of $\mathcal{R}$ into  $\mathcal{R}$ has the explicit form
\[(x_1,x_2)\in (\mathbb{C}^*)^2\mapsto (x_1,x_2, x_1x_2+x_1^{-1}x_2^{-1},x_1+x_1^{-1},
x_2+x_2^{-1},x_1x_2)\in \mathcal{R}(\mathbb{C})\ .\]
The image of the induced morphism  ${\rm Spec}(\mathbb{C}[x_1,x_1^{-1},x_2,x_2^{-1}])\rightarrow \mathcal{P}$
lies in  $\mathcal{P}_{red}:={\rm Spec}(\mathbb{C}[s_1,s_2,s_3,s_3^{-1}]/(R_1))$, where $R_1$ is the irreducible element
$R_1=(s_3 +s_3^{-1})^2-s_1s_2(s_3+s_3^{-1})+s_1^2+s_2^2-4$.  More precisely, since $R_1$ is irreducible, one has  an inclusion 
$\mathbb{C}[s_1,s_2,s_3,s_3^{-1}]/(R_1)\rightarrow \mathbb{C}[x_1,x_1^{-1},x_2,x_2^{-1}]$, given by $s_1\mapsto  x_1+x_1^{-1},\ s_2\mapsto x_2+x_2^{-1}, \ s_3\mapsto x_1x_2 $. This identifies the first ring with the subring $\mathbb{C}[x_1+x_1^{-1},x_2+x_2^{-1}, x_1x_2, x_1^{-1}x_2^{-1}]$ of the second one.
It easily follows that $\mathbb{C}[x_1,x_1^{-1},x_2,x_2^{-1}]$ is the normalization of $\mathbb{C}[s_1,s_2,s_3,s_3^{-1}]/(R_1)$ in its field of fractions.

The singular locus of $\mathcal{P}_{red}$ itself is easily computed to be the union  of two disjoint components given by the ideals $(s_3-1,s_1-s_2)$ and $(s_3+1,s_1+s_2)$.  
The map $\tau :{\rm Spec}(\mathbb{C}[x_1,x_1^{-1},x_2,x_2^{-1}])\rightarrow \mathcal{P}_{red}$ is an isomorphism outside
the singular locus of $\mathcal{P}_{red}$ and further satisfies:

\[\tau ^{-1}(s_1,s_1,1)=\{(\frac{1}{2}  \left(s_1\pm\sqrt{s_1^2-4}\right),  \frac{1}{2}
   \left(s_1 \mp \sqrt{s_1^2-4}\right))\}\]
\[\mbox{ for } s_1\neq \pm 2 \mbox{ and }
\tau ^{-1}(\pm 2,\pm 2,1)=\pm (1,1);\]
\[\tau ^{-1}(s_1,-s_1,-1)=\{(\frac{1}{2}   \left(s_1\pm\sqrt{s_1^2-4}\right),  -\frac{1}{2}
   \left(s_1 \mp \sqrt{s_1^2-4}\right))\}\] 
   \[ \mbox{ for }s_1\neq \pm 2 \mbox{ and }
\tau ^{-1}(\pm 2,\mp 2,-1)=(\pm 1,\mp 1). \]

If for a fixed point $p\in \mathcal{P}$, the fibre $\pi ^{-1}(p)$ has a singular point, then  the ideal
$(\frac{d}{dx_1}F(x_1,x_2,x_3,p),\frac{d}{dx_2}F(x_1,x_2,x_3,p),\frac{d}{dx_3}F(x_1,x_2,x_3,p))$ is not the unit ideal and
it follows that the ideal  $I:=(F,\frac{d}{dx_1}F,\frac{d}{dx_2}F,\frac{d}{dx_3}F) \cap \mathbb{C}[s_1,s_2,s_3,s_3^{-1}]$ 
lies in the maximal ideal of $\mathbb{C}[s_1,s_2,s_3,s_3^{-1}]$ defined by the point $p$. 
Using a Gr\"obner basis one verifies that $I$ is generated by  $(s_1^2-4)(s_2^2-4)R_1(s_1,s_2,s_3)$. Thus singular points in 
$\pi ^{-1}(p)$ occur for $p$ lying on one of the five divisors on $\mathcal{P}$ defined by $s_1=\pm 2,\ s_2=\pm 2, R_1=0$. The first four divisors correspond to resonance for the matrices $M_1,M_2$ and the last one to reducibility.  
 A singular point in $\pi^{-1}(p)$, with $p$ lying on only one of the divisors
has type $A_1$. If $p$ lies on more than one divisor, the singularity type can be different. The following table, of importance
for the comparison with the Okamoto-Painlev\'e pairs, gives the rather complicated structure of the singularities of the fibres (see Table \ref{tab:sing-p4}). 

\small
\begin{table}
\begin{center}
\begin{tabular}{|      c         |         c       |          c      ||  c                                              |     c   |} \hline 
                                     &                 &              &                                                   & \\
 $s_1$                          &     $s_2$    &   $R_1$  & Singular points                                 & Type of the  \\
                                     &                   &               & $(x_1, x_2, x_3)$                     & singularities \\
 \hline 
 \hline
                      $2$         & $\not= \pm 2$ & $\neq 0$ &    $ (1, s_3, s_2)$          & $A_1$ \\
 \hline 
                      $2$         & $2$             & $\neq 0$     &    $(1, s_3, 2), (s_3,1,2) $           & $A_1+A_1 $ \\ 
 \hline 
                      $2$         &        $2$     &      $0$         &    $(1, 1, 2) $                & $A_3$ \\ 
 \hline 
                      $2$         & $-2$           & $\neq 0$      &    $(-s_3, -1, -2), (1,s_3,-2) $       & $A_1+A_1$ \\ 
  \hline 
                      $2$         & $-2$           & $0$              &    $(1, -1, -2) $                & $A_3$ \\ 
 \hline 
                      $2$         & $\not=\pm 2$&  $0$          &   $(1, s_3, s_3+s_3^{-1}) $ & $A_2$ \\
 \hline 
 \hline 
                       $-2$       & $\not= \pm 2$ & $\neq 0$  &  $ (-1, -s_3, -s_2)$          & $A_1$ \\
  \hline 
                       $-2$       &  $2$          & $\neq 0$       &   $(-1, -s_3, -2), (s_3,1,-2)  $        & $A_1+A_1 $ \\ 
 \hline 
                      $-2$       & $2$           & $0$          &     $(-1, -1, -2) $                    & $A_3$ \\
 \hline 
                      $-2$        & $-2$         & $\neq 0$ &  $(-1, -s_3, 2), (-s_3,-1,-2) $                    & $A_1+A_1$ \\ 
\hline 
                       $-2$       & $-2$          & $0$          &        $(-1, -1, 2) $                   & $A_3$ \\ 
\hline 
                       $-2$       &$\not=\pm 2$ & $0$       &   $(-1, -s_3, s_3+s_3^{-1}) $ &  $A_2$ \\ 
\hline 
\hline  
              $\not=\pm 2$ & $2$         & $\neq 0$        &  $(s_3,1,s_)$            &      $A_1 $ \\ 
\hline 
           $\not=\pm 2$ &   $2$        & $0$ &   $(s_3, 1, s_3+s_3^{-1})$  & $A_2 $ \\
 \hline
          $\not=\pm 2$ & $-2$        & $\neq 0$ &   $(-s_3, -1, -s_1)$             & $A_1 $ \\
  \hline 
         $\not=\pm 2$ & $-2$        & $0$      &  $(-s_3, -1, s_3+s_3^{-1})$   & $A_2 $ \\
 \hline
          $\not= \pm 2$ & $\not= \pm 2$ & $0$ &    $(a_1, a_2,  a_3)$ & $A_1$ \\ 
\hline 
 $\not= \pm 2$ & $s_1$ & $0$ &    $(\alpha, \beta, 2), (\beta ,\alpha ,2)$ & $A_1 + A_1$   \\   
 \hline  
          $\not= \pm 2$ & $-s_1$ & $0$ &  $(\alpha, -\beta, 2),(-\beta , \alpha ,2)$ & $A_1 + A_1$ \\   
\hline    
\end{tabular}
\vspace{0.2cm}
\end{center}
\flushleft
{\it This table of the singularities of the fibres} uses the notation \begin{small}
\[(a_1, a_2,  a_3)=(\frac{s_3^2-1}{s_2 s_3 -s_1}, 
\frac{s_3(s_2 s_3 -s_1)}{s_3^2-1}, s_3+s_3^{-1})\mbox{ and } \]\[ \alpha= \frac{1}{2}
   \left(s_1+\sqrt{s_1^2-4}\right), \beta= \frac{1}{2}  \left(s_1-\sqrt{s_1^2-4}\right) \ . \] \end{small}
As usual, the symbol  $A_n,\ n\geq 1$ stands for the surface singularity given by the local equation
$x^2+y^2+z^{n+1}=0$.
\vspace{0.5cm}
\caption{Singularities for the monodomy spaces for PV.}
\label{tab:sing-p4}
\end{table}

\normalsize

\subsection{Family $(0,0,1/2)$ and Painlev\'e  ${\bf  PV}_{\bf deg}$ }

A differential module of this type is irreducible and by Theorem 1.11 can be represented by a matrix differential equation of the form  $\frac{d}{dz}+\frac{A_0}{z}+\frac{A_1}{z-1}+A_\infty$   with $tr(A_0)= tr(A_1)=0$ and $A_\infty$ nilpotent. The generalized eigenvalues at $\infty$ are $\pm t \cdot z^{1/2}$  and $t\in \mathbb{C}^*$. One may normalize by  $A_\infty ={0\ t^2 \choose 0 \ 0 }$.

\smallskip
For the computation of {\it monodromy space} $\mathcal{R}$ we give the solution space $V$ at $\infty$ a basis $e_1,e_2$ such that  $V_q=\mathbb{C}e_1,\ V_{-q}=\mathbb{C}e_2, \gamma e_1=e_2,\ \gamma e_2=-e_1$.  Let $M_0,M_1,M_\infty$ denote the topological monodromies at $0,1,\infty$ on the basis $e_1,e_2$. Then $M_\infty ={0\ -1\choose 1\ \ 0}{1\ 0 \choose e\ 1}$ and one finds $M_0M_1{-e \ -1\choose1\ \ \ 0 }=1$. Changing the basis at $\infty$ does not effect these data. Therefore $\mathcal{R}$ has dimension $3+3+1-3=4$ (for $M_0,M_1,M_\infty$ and the 3 equations). One considers the map  $\mathcal{R}\rightarrow \mathcal{P}:=\mathbb{C}\times \mathbb{C}$ which sends the tuple to $(s_0,s_1):=(tr(M_0),tr(M_1))$.  \smallskip

Write $M_1={a_1 \ b_1 \choose  c_1  \ d_1} $.
 The equation $M_0M_1M_\infty =1$ determines $M_0$ in terms of $M_1,M_\infty$. In particular, $s_0=-c_1+b_1+a_1e$. Thus $\mathcal{R}$ is the space, given by the 5 variables $a_1,b_1,c_1,d_1,e$  and the equation $a_1d_1-b_1c_1=1$. Use $s_0$ and 
 $s_1=a_1+d_1$ to eliminate $c_1$ and  $d_1$. Then the single equation between $b_1,a_1,e,s_0,s_1$ reads $a_1b_1e+a_1^2+b_1^2-a_1s_1-b_1s_0+1=0$. With the choice $x_1=-b_1,\ x_2=-a_1,\ x_3=e$ this equation is 
 \[x_1x_2x_3+x_1^2+x_2^2+s_0x_1+s_1x_2+1=0 \mbox{ and shows that }
 \mathcal{R}\rightarrow \mathcal{P}\] is a family of affine cubic surfaces.   
 We note that there are no reducible cases and that $\mathcal{R}$ is nonsingular.
The singularities of the fibres occur only for the loci $s_0=\pm 2$ and/or $s_1=\pm 2$, corresponding to resonance. The fibres for  $(s_0,s_1)=(\pm 2 ,\neq \pm 2)$ and 
$(s_0,s_1)=(\neq \pm 2,\pm 2)$ contain one singular point and the fibers for
$(s_0,s_1)=(\pm 2,\pm 2)$ contain two singular points. All these surface singularities
are of type $A_1$.  

\subsection{Family $(1,-,1)$ and Painlev\'e PIII(D6) }

Due to the ample choice of invariant lattices at $0$ and at $\infty$, any differential  module of this type can be represented by a matrix differential equation
$z\frac{d}{dz}+A_{0}z^{-1}+A_1+A_2z$. By a transformation $z\mapsto \lambda z$
one arrives at eigenvalues $\pm \frac{t}{2} z^{-1}$ at $0$ and $\pm \frac{t}{2} z$ at $\infty$ with $t\in
\mathbb{C}^*$. Moreover one can normalize such that $A_2=\frac{t}{2}{1\ \ 0\choose 0\ -1 }$.
There are more normalizations possible. 

The affine space $AnalyticData$ is described as follows.\\
The formal solution space $V(0)$ at $0$ is given a basis $e_1,e_2$ such that the
formal monodromy,  the Stokes matrices and the topological monodromy $top_0$ have the form
\[ \left(\begin{array}{ll}   \alpha & 0  \\ 0 & \alpha ^{-1}  \\  
\end{array}\right),\ 
\left(\begin{array}{ll}   1  & 0  \\ a_1 & 1  \\  \end{array}\right), \
\left(\begin{array}{ll}   1  & a_2  \\ 0 & 1  \\  \end{array}\right),\
\left(\begin{array}{ll}   \alpha  & \alpha a_2 \\
 \alpha ^{-1} a_1& \alpha ^{-1}(1+a_1a_2)  \\  \end{array}\right).\]
The last matrix is written as ${m_1\ m_2\choose m_3\ m_4 }$. It is characterized by
$m_1\neq 0,\ m_1m_4-m_2m_3=1$ and it determines $\alpha , a_1, a_2$.  
Moreover, $e_1\wedge e_2$ is a fixed global solution of the second exterior power.

The formal solution space $V(\infty )$ at $\infty$ is provided with a basis
$f_1,f_2$, such that $f_1\wedge f_2$ is again this fixed global solution and the
formal monodromy, the Stokes maps and the topological monodromy $top_\infty$ have the matrices
\[ \left(\begin{array}{ll}   \beta & 0  \\ 0 & \beta ^{-1}  \\
  \end{array}\right),\ 
\left(\begin{array}{ll}   1  & 0  \\ b_1 & 1  \\  \end{array}\right),\ 
\left(\begin{array}{ll}   1  & b_2  \\ 0 & 1  \\  \end{array}\right),\
\left(\begin{array}{ll}   \beta  & \beta b_2 \\
 \beta ^{-1} b_1& \beta ^{-1}(1+b_1b_2)  \\  \end{array}\right).\]

The link $L:V(0)\rightarrow V(\infty )$ satisfies:\\
(i) $top _\infty \circ L=L\circ top_0$, this follows from $M_1M_\infty =1$. \\
(ii) $L$ maps $e_1\wedge e_2$ to $f_1\wedge f_2$. Thus the matrix
 ${\ell _1\ \ell _2\choose \ell _3\ \ell_4 }$ of $L$ w.r.t. the given bases has determinant 1.\\
One uses (i) to forget the data for $\infty$. The coordinate ring for $AnalyticData$ is
the localization of $\mathbb{C}[m_1,\dots ,m_4,\ell_1,\dots ,\ell _4]/(m_1m_4-m_2m_3-1,
\ell_1\ell_4-\ell _2\ell _3-1)$, given by $0\neq \alpha=m_1$ and 
$0\neq \beta =\ell_1\ell_4m_1+\ell_2l_4m_3-\ell_1\ell_3m_2-\ell_2\ell_3m_4$.

The monodromy space $\mathcal{R}$ is obtained by dividing $AnalyticData$ by 
the action of the elements $(\gamma , \delta )\in \mathbb{G}_m\times \mathbb{G}_m$,
which is induced by the base change $e_1,e_2,f_1,f_2\mapsto \gamma e_1,\gamma ^{-1}e_2,\delta f_1,\delta ^{-1}f_2$.

The new matrices are
$\left(\begin{array}{ll}  m_1 & \gamma ^2 m_2  
\\ \gamma ^{-2}m_3 & m_4  \\  \end{array}\right)$
and 
$\left(\begin{array}{ll}\gamma ^{-1}\delta   \ell_1 &\gamma \delta \ell_2  
\\\gamma ^{-1}\delta ^{-1} \ell_3 & 
\gamma \delta ^{-1} \ell_4  \\  \end{array}\right)$.\\

The ring of invariants for the action of $\mathbb{G}_m\times \mathbb{G}_m$ is
computed to be (a localization of)  $\mathbb{C}[m_1,m_4, \ell_1\ell_4,m_2\ell_1\ell_3,m_3\ell_2\ell_4]$. We note that $m_2m_3$ and $\ell_2\ell_3$ are omitted because of the determinant =1 relation. There is only one relation between these five generators  namely (recall $\alpha =m_1$)
\[(m_2\ell_1\ell_3)\cdot (m_3\ell_2\ell_4)+(-\alpha m_4+1)
\cdot (\ell_1\ell_4)\cdot (\ell_1\ell_4-1)=0 .\]
 
Writing  $y_1:=\ell_1\ell_4,\ y_2:=m_2\ell_1\ell_3,\ y_3:=m_3\ell_2\ell_4$ and using the formula for $\beta$ one  obtains the equation and the formula
\[y_2y_3+(-\alpha m_4+1)y_1(y_1-1)=0 \mbox{ and }
 \beta = \alpha y_1+y_3-y_2-(y_1-1)m_4 . \]
Using the formula for $\beta$ one eliminates $y_3$ and obtains the equation
\[y_2 (\beta -\alpha y_1+y_2+(y_1-1)m_4) +(-\alpha m_4+1)y_1(y_1-1)=0 .\]
For fixed $\alpha, \beta$ this is a cubic equation in $y_1,y_2, m_4$. After 
a series of simple transformations, one obtains the following equation for 
$\mathcal{R}\rightarrow \mathcal{P}=\mathbb{C}^*\times \mathbb{C}^*$ 
\[x_1x_2x_3+x_1^2+x_2^2+(1+\alpha \beta )x_1+(\alpha +\beta )x_2+\alpha \beta =0.\]
The discriminant of $\mathcal{R}\rightarrow \mathcal{P}$ has the formula 
$(\alpha -\beta )^2(\alpha \beta -1)^2$ and therefore the fiber above 
$(\alpha ,\beta)$ with $\alpha \neq \beta ,\beta ^{-1}$ is non singular.
The singular locus of $\mathcal{R}$ consists
of the two non intersecting lines
 \[ L_1:\ \ \alpha =\beta,\ (x_1,x_2,x_3)=(0,-\alpha ,\alpha +\alpha ^{-1} ) \mbox{ and }\]
 \[ L_2:\ \ \alpha =\beta ^{-1},\ (x_1,x_2,x_3)=(-1,0,\alpha +\alpha ^{-1}).\]
They correspond to the reducible connections (or equivalently reducible monodromy data).
All the singularities of the fibres are obtained by intersecting with $L_1$ or $L_2$.
The corresponding surface singularities are of type $A_1$. If 
$\alpha \neq \pm 1$ and $\beta=\alpha ^{\pm 1}$, then there is only one singular point
in the fiber. If $\alpha =\beta =\pm 1$, then the fiber has two singular points.

\subsection{Family $(1/2,-,1)$ and Painlev\'e PIII(D7) }

By Theorem 1.11, any differential module of this type is represented by a matrix
differential equation $z\frac{d}{dz}+A_{0}z^{-1} + A_1+A_2 z$. One may normalize
$A_2={\frac{t}{2} \ \ \  0 \choose 0 \ -\frac{t}{2} }$. After a transformation $z\mapsto \lambda z$
one may suppose that the eigenvalues at $0$ are $\pm z^{-1/2}$ and 
$\pm \frac{t}{2} \cdot z$ at $\infty$. Assuming that $A_0$ and $A_2$ have no common eigenvector
leads to the explicit  family 
\[z\frac{d}{dz}+A_0z^{-1}
+A_1  + {\frac{t}{2} \ \ \  0 \choose 0 \ -\frac{t}{2} }z  .\]

For the description of the space $AnalyticData$, the formal solution space $V(0)$ at
$0$ is given the  basis $e_1,e_2$ for which the formal monodromy, the Stokes matrix and topological monodromy $top_0$ have the matrices
\[ {0\ -1\choose 1 \ \ 0},\ {1\  0\choose e\ 1}, \ {-e \ -1 \choose 1\ \ \ 0 }.\]
The formal solution space $V(\infty )$ at $\infty$ is given a basis $f_1,f_2$ for which 
the formal monodromy, the Stokes maps and the topological monodromy $top_\infty$ have
the matrices
\[{\alpha \ \ 0 \choose 0\ \ \alpha ^{-1}},\ {1\ \ \ 0\choose c_1\ 1 },\ 
{1\ c_2\choose 0\ \ \ 1 },\ \left(\begin{array}{cc}\alpha & \alpha c_2\\
\alpha ^{-1}c_1&\alpha ^{-1}(1+c_1c_2)\end{array}\right)\ .\] 
One writes $top_\infty ={a\ b \choose c\ d}$ with $a\neq 0$ and determinant 1.
It is assumed that $e_1\wedge e_2$ and $f_1\wedge f_2$ are the same global solution
of the second exterior power of the differential equation.
The link $L:V(0)\rightarrow V(\infty )$ has therefore a matrix
${\ell_1\ \ell_2\choose \ell_3\ \ell_4 }$ with determinant 1. 
 
 The equation  $top_0\cdot L^{-1}top_\infty L=1$ can be written as 
${a\ b\choose c\ d}= L {0\ \ 1\choose -1\ -e }L^{-1}$. This eliminates
 ${a\ b \choose c\ d }$ (and thus $\alpha ,c_1,c_2$). The coordinate ring of
 $AnalyticData$ is therefore $\mathbb{C}[\ell _1,\dots ,\ell _4,e]/(\ell _1\ell _4-\ell _2\ell _3-1)$. The elements $\mu \in \mathbb{G}_m$ act on $AnalyticData$ by the base change $f_1,\ f_2\mapsto \mu f_1,\ \mu ^{-1}f_2$. The elements  
 $\ell_1,\ell_2,\ell_3,\ell_4$ have weights $-1,-1,1,1$ for this action.

The coordinate ring of $\mathcal{R}$ is generated by the variables
 $e,\ell_{13},\ell_{14}, \ell_{23},\ell_{24}$ where $\ell_{ij}:=\ell_i\ell_j$. There are 
two relations, namely $\ell_{14}-\ell_{23}=1$ and $\ell_{14}\ell_{23}=\ell_{13}\ell_{24}$.
$\mathcal{R}$ has dimension $3$. The map $\mathcal{R}\rightarrow \mathcal{P}=\mathbb{C}^*$ is defined by: an element in $\mathcal{R}$ is mapped to 
$\alpha =-\ell_{24}-\ell_{13}-\ell_{23}e$, one of the eigenvalues of the 
formal monodromy at $\infty$. Eliminate 
$\ell_{14}=\ell_{23}+1$. Then we have the equation 
$(\ell_{23}+1)\ell_{23}+\ell_{13}(\alpha +\ell_{13}+\ell_{23}e)=0$ 
(here $\ell_{24}$ is eliminated). We obtain a family $\mathcal{R}\rightarrow \mathcal{P}=\mathbb{C}^*$ of non singular affine cubic surfaces given by the equation
$ \ell_{13}\ell_{23}e+\ell_{13}^2+\ell_{23}^2+\alpha \ell_{13}+\ell_{23} =0.$

\subsection{Family $(1/2,-,1/2)$ and Painlev\'e PIII(D8)}
 We consider differential modules of this type for which the
 strong Riemann--Hilbert problem has a solution (see Definition and examples 1.10, part (2)).
 Then a matrix differential equation $z\frac{d}{dz}+A_{0}z^{-1}+A_1+
 A_2z$, with nilpotent $A_0$ and $A_2$, represents the module. Further assuming that the eigenvectors of $A_0$ and $A_2$
are distinct one can normalize to an equation of the form (see \ref{ss:p3-d8})
\[z\frac{d}{dz}+{0 \ 0\choose -q \ 0} z^{-1} +
 {\frac{p}{q} \ -\frac{t}{q} \choose 1\ -\frac{p}{q} }+{0\ 1 \choose 0\ 0 }z.  \]

The space $AnalyticData$ is build as follows.
The formal solution space $V(0)$ at $0$ is given a basis $e_1,e_2$, unique up to multiplication by the same constant, such that
$V(0)_{z^{-1/2}}=\mathbb{C}e_1,\ V(0)_{-z^{-1/2}}=\mathbb{C}e_2$, such that the formal monodromy has matrix ${0\ -1\choose 1\ \ 0 }$. There is  one Stokes  matrix ${1\ 0\choose a\ 1}$. The topological monodromy at $0$ is the product, 
i.e., ${-a\ -1 \choose 1\ \ \ 0 }$.\\

At $\infty$, one has similarly a basis $f_1,f_2$ for $V(\infty )$ with 
topological monodromy 
${0\ -1\choose 1\ \ 0 }{1\ 0\choose b\ 1 }={-b\ -1\choose 1\ \ 0 }$.
The matrix of the link $L:V(0)\rightarrow V(\infty )$ with respect to these 
basis satisfies $L:e_1\wedge e_2\mapsto f_1\wedge f_2$, because
we assume, as we may, that $e_1\wedge e_2$ and $f_1\wedge f_2$ are the same global solution of the second exterior power. Thus
$L=:{\ell_1\ \ell_2\choose \ell_3\ \ell_4 }$ has determinant 1. The  identity  \[{\ell_1\ \ell_2\choose \ell_3\ \ell_4 } 
{-a\ -1\choose 1\ \ \ \ 0}={-b\ -1\choose 1\ \ \ \ 0 }
{\ell_1\ \ell_2\choose \ell_3\ \ell_4 } \]
describes the generators and relations of the coordinate ring of
$AnalyticData$. The only admissible bases change is
$e_1,e_2,f_1,f_2\mapsto \lambda e_1,\lambda e_2,\lambda f_1,\lambda f_2$ with $\lambda \in \mathbb{C}^*$ acts trivially on
$AnalyticData$ and thus this space coincides with $\mathcal{R}$.

After elimination of $b,\ell_1,\ell_3$ one is left with the variables 
$a,\ell_2,\ell_4$ and one equation, namely 
$a\ell_2\ell_4+\ell_4^2-\ell_2^2-1=0$, or in other
variables 
\[x_1x_2x_3+x_1^2-x_2^2-1=0\ .\]  This defines a non singular affine cubic surface.

\subsection{Family $(0,-,2)$ and Painlev\'e PIV}

The singularity at $\infty$ of a module of this type guarantees that
the strong Riemann-Hilbert problem has a solution. 
There is a corresponding matrix differential equation which can be normalized to $z\frac{d}{dz}+A_0+A_1z+{1\ \ 0\choose 0\ -1}z^2$
Using the transformation $z\mapsto \lambda z$ one may
suppose that the eigenvalues at $\infty$ are  $\pm (z^2+\frac{t}{2}\cdot z)$. 
The ingredients for $AnalyticData$ are the following.

The symbolic solution space at $\infty$ is written as 
$V_q\oplus V_{-q}$ and one takes a basis $\{e_1\}$ and $\{e_2\}$ for $V_q$ and
$V_{-q}$. With respect to the basis $\{e_1,e_2\}$
the topological monodromy $top_\infty$ at $\infty$ has the form

\[top_\infty =\left(\begin{array}{cc}\alpha &0\\0&\alpha ^{-1} \end{array}\right)
\left(\begin{array}{cc}1&0\\a_1 &1\end{array}\right)
\left(\begin{array}{cc}1&a_2\\0&1 \end{array}\right)
\left(\begin{array}{cc}1&0\\a_3&1 \end{array}\right)
\left(\begin{array}{cc}1&a_4\\0&1\end{array}\right)\ ,\]
where the first matrix is the formal monodromy and the others are the 4 Stokes
matrices. Let $top_0$ denote the monodromy at $0$ written on the basis $e_1,e_2$.The condition $top_0\cdot top_\infty=1$ implies that $top_\infty$ determines $top_0$. The coordinate ring of $AnalyticData$ is
$\mathbb{C}[\alpha, \alpha ^{-1},a_1,\dots ,a_4]$. An element 
$\lambda \in \mathbb{G}_m$ acts on $AnalyticData$ by the base change $e_1,e_2\mapsto \lambda e_1,\lambda ^{-1}e_2$. The weights
of $\alpha ,a_1,a_2,a_3,a_4$ for this action are $0,+1,-1,+1,-1$.
Therefore $\mathcal{R}$  has coordinate ring 
$\mathbb{C}[\alpha ,\alpha ^{-1},a_{12},a_{14},
a_{23},a_{34}]$, where $a_{ij}:=a_ia_j$. There is only one relation namely $a_{12}a_{34}-a_{14}a_{23}=0$.\\

The {\it singular points} of $\mathcal{R}$ are
given by the equations $a_{12}=a_{14}=a_{23}=a_{34}=0$. This coincides with the locus where the monodromy data (or equivalently
the differential modules) are reducible  (namely $a_1=a_3=0$
or $a_2=a_4=0$).\\

The morphism $\mathcal{R}\rightarrow \mathcal{P}:=\mathbb{C}\times 
\mathbb{C}^*$, where $\mathcal{P}$ is a space of parameters, is  given by
$(\alpha ,a_1,\dots ,a_4)\mapsto (tr (top_1),\alpha )$. Now
$tr(top_1)=tr(top_\infty )$ and 
\[tr(top_\infty )=\alpha (1+a_{23})+\alpha ^{-1}(a_{14}+a_{34}+a_{12}
+a_{12}a_{34}+1)\ .\]

Write $s_2=\alpha$ and $s_1=tr (top_\infty )$ and exchange  
$a_{23}$ with $s_1$ by the formula
\[a_{23}=s_2^{-1}s_1-s_2^{-2}(a_{14}+a_{34}+a_{12}+a_{12}a_{34}+1)-1\ .\]
Then the coordinate ring of $\mathcal{R}$ has the form
$\mathbb{C}[s_1,s_2,s_2^{-1},a_{12},a_{14},a_{34}]$ and there is one relation,
namely
\[a_{12}a_{14}a_{34}+(s_2^2a_{12}a_{34}+a_{14}^2+a_{14}a_{12}+a_{14}a_{34})+
a_{14}(1+s_2^2-s_1s_2)=0\ . \]

One makes the following substitutions 
\[a_{14}=x_1-s_2^2,\ a_{12}=x_2-1,\ a_{34}=x_3-1 \mbox{ and the relation $R$
  reads}\]
\[x_1x_2x_3+x_1^2-(s_2^2+s_1s_2)x_1-s_2^2x_2-s_2^2x_3+s_2^2+s_1s_2^3 =0 \  ,\]
and  thus $\mathcal{R}\rightarrow \mathcal{P}= \mathbb{C}\times \mathbb{C}^*$ is a family of affine cubic surfaces.

\subsubsection{Singular loci of $\mathcal{R}$
 and the fibres of $\mathcal{R}\rightarrow \mathcal{P}$}
We have already remarked that the  singular points of $\mathcal{R}$ correspond to
reducibility  and are given by 
$x_1=s_2^2,\ x_2=1,\ x_3=1,\ s_1=s_2+s_2^{-1}$. \\
For a fixed $s=(s_1,s_2)\in \mathcal{P}$, the singular locus of the fibre 
is given by the (relative) Jacobian ideal, generated by  
$R, \partial R/\partial x_1, \partial R/\partial x_2, \partial R/\partial x_3 $. A Gr\"obner
basis for this ideal produces the following results.

The fiber  has singular points  if and only if $s$ satisfies the equation
\[ \Delta(s):=(s_1-2)(s_1+2)(s_2^2-s_1 s_2+1)=0.   \]
We define three divisors of $\mathcal{P}$ 
by $D_1^{\pm} = \{ s_1 = \pm 2 \}, D_{red} =\{ s_2^2-s_1 s_2+1=0 \}$.  We have seen that $D_{red}$ corresponds to the locus of the {\it reducible differential equations}. 
Further $s_1=e^{\pi i\theta _0}+e^{-\pi i\theta _0}$, where $\pm \theta _0/2$ are the local exponents at $z=0$ of the differential equation. Thus $s_1=\pm 2$ corresponds to the  {\it resonant case} $\theta _0\in \mathbb{Z}$. The table gives the singularities and their type of the fibres.   
\begin{center}
\begin{tabular}{|c|c|c|c|} \hline 
  
 &  $s=(s_1, s_2)$ & $(x_1, x_2, x_3)$ &  Type \\ 
  \hline   
$D_1^{+}$ not $D_{red}$  & $(2, s_2), s_2 \not=1$ & $(s_2, s_2, s_2)$ 
& $A_1$ \\ \hline 
$D_1^{-}$ not $D_{red}$&$(-2, s_2), s_2 \not= -1$ & $(-s_2, -s_2, -s_2)$ 
& $A_1$ \\ \hline
 $D_{red}$  not $D_1^{\pm}$ &$(s_2 +s_2^{-1}, s_2), s_2 \not=\pm 1$ 
 & $(s_2^2, 1, 1)$ 
& $A_1$ \\ \hline
$D_1^{+} \cap D_{red} $  & $(2, 1)$ & $(1, 1, 1)$ & $A_2$ \\ 
\hline 
 $D_1^{-} \cap D_{red} $  &   $(-2, -1)$ & $(1, 1, 1)$ & $A_2$ \\ 
\hline 
\end{tabular}
\end{center}

\subsection{Family $(0,-,3/2)$ and Painlev\'e PIIFN }

A module of this type can be represented, by Theorem 1.11, by
a matrix differential equation $z\frac{d}{dz}+A_0+A_1z+A_2A_2$
with $A_2$ nilpotent. One can use the transformation $z\mapsto \lambda z$ and choose a basis such that the explicit form is
$z\frac{d}{dz}+{a\ \ c\choose -b\ -a }+{0\ t+b\choose 1\ \ \ 0 }z+
{0\ 1\choose 0\ 0 }z^2$. The eigenvalues at $\infty$ are
$\pm (z^{3/2}+\frac{t}{2}\cdot z^{1/2})$. 

 The space $AnalyticData$ is formed as follows. The formal solution
 space $V$ at $\infty$ is given a basis $e_1,e_2$ such that the formal
 monodromy and the three Stokes maps have the matrices
 \[ {0\ -1\choose 1\ \  \ 0 }, \ {1\ \ 0\choose a_1\ 0 },\  {1\ a_2\choose 0\ \ 1 },\  {1\ \ 0\choose a_3\ 1}.\ \]
 The topological monodromy $top_\infty$ is the product of these matrices and $top_0$ is the inverse of $top_\infty$. Further,
 the base change $e_1,e_2\mapsto \lambda e_1,\lambda e_2$
 does not effect the matrices. It follows that the coordinate ring of 
 $\mathcal{R}$ is $\mathbb{C}[a_1,a_2 ,a_3]$. One computes that
 the trace $s$ of the topological monodromy at $0$ is
 $s=-a_1a_2a_3-a_1+a_2-a_3$. The map $\mathcal{R}\rightarrow 
 \mathcal{P}$ is given by $(a_1,a_2,a_3)\mapsto s$.
 Thus $\mathcal{R}\rightarrow \mathcal{P}$ is a family of affine cubic
 surfaces given by the equation $a_1a_2a_3+a_1-a_2+a_3+s=0$.

The {\it singularities of the fibres}  occur only for the resonant case 
$\theta _0\in \mathbb{Z}$, where $\pm \theta _0/2$ are the local exponents at $z=0$. Since $s=e^{\pi i\theta _0}+e^{-\pi i\theta _0}$, 
this corresponds to $s=\pm 2$. For $s=2$ one finds one singular point $(a_1,a_2,a_3)=(-1,1,-1)$ and for $s=-2$ one singular point
$(a_1,a_2,a_3)=(1,-1,1)$. The type of the singularity is $A_1$ in both cases.

\subsection{Family $(-,-,3)$ and Painlev\'e PII}

{\it The family of connections}.
A differential module of this type can be represented by a matrix differential equation $\frac{d}{dz}+A_0+A_1z+A_2z^2$, because of 
the singularity at $\infty$. Using a transformation $z\mapsto \lambda z+
\mu$ and by choosing a suitable basis one arrives at the explicit form,
 having eigenvalues $\pm (z^3+\frac{t}{2}\cdot z)$ at $\infty$, namely
 \[\frac{d}{dz}+\left(\begin{array}{cc}a_{10}+z^2& a_{21}z+a_{20}\\ 
a_{31}z+a_{30}&
-a_{10}-z^2\end{array}\right)\mbox{ and } t=a_{10}+a_{21}a_{31}/2. \]
The group $\mathbb{G}_m$ acts by conjugation, in fact by 
$a_{21}z+a_{20}\mapsto \lambda (a_{21}z+a_{20})$ and 
$a_{31}z+a_{30}\mapsto \lambda ^{-1}(a_{31}z+a_{30})$. In general,
this cannot be used to normalize the equation even further. (See Subsection \ref{ss:p2-e7}). \\

The space $AnalyticData$ consists of the formal monodromy and
six Stokes matrices.  The formal solution space $V$ at $\infty$ is given a basis  $e_1, e_2$  such that the formal monodromy and the six Stokes maps have the matrices
\[ \left(\begin{array}{cc}\alpha &0\\ 0&\alpha ^{-1}\end{array}\right),\
\left(\begin{array}{cc}1 &0\\ b_1&1\end{array}\right),\
\left(\begin{array}{cc}1 &b_2\\ 0&1\end{array}\right), \cdots ,
\left(\begin{array}{cc}1&b_6\\ 0&1\end{array}\right). \]
The product of all of them is the topological monodromy at $\infty$
and hence is equal to ${1\ 0\choose 0\ 1 }$. The coordinate ring of
$AnalyticData$ is therefore generated by $\alpha ,\alpha ^{-1},b_1,\dots ,b_6$ and the matrix identity defines the ideal of the relations
$I\subset \mathbb{C}[\alpha ,\alpha ^{-1},b_1,\dots ,b_6]$.
The basis $e_1,e_2$ is unique up to the action of the elements $\lambda \in \mathbb{G}_m$, given by $e_1, e_2\mapsto \lambda e_1,
\lambda ^{-1}e_2$.

Call the  six Stokes matrices $M_1,\dots ,M_6$. Then $M_3M_4M_5M_6$ is equal to 
\[\left(\begin{array}{cc}\alpha ^{-1}(1+b_1b_2)&-\alpha b_2\\
-\alpha ^{-1}b_1&\alpha \end{array}\right)\mbox{
 which is the inverse of } \left(\begin{array}{cc}\alpha& 0\\ 0& \alpha ^{-1}
\end{array}\right) M_1M_2.\]
 We note that the product of the three matrices
determines $\alpha ,b_1,b_2$. Further one computes that 
$\alpha =b_3b_6+(1+b_3b_4)(1+b_5b_6)$. Thus the coordinate ring of $AnalyticData$ is $\mathbb{C}[b_3,b_4,b_5,b_6,\frac{1}{b_3b_6+(1+b_3b_4)(1+b_5b_6)}]$.
For the group $\mathbb{G}_m$ the variables $b_1,b_3,b_5$ have weight $-1$,
the variables $b_2,b_4,b_6$ have weight $+1$ and $\alpha$ has weight $0$. Write $b_{ij}:=b_ib_j$ for $i<j$. Then the coordinate ring of $\mathcal{R}$ is the ring of the $\mathbb{G}_m$-invariant elements and this is 
\[\mathbb{C}[b_{34},b_{36},b_{45},b_{56},\frac{1}{b_{36}+(1+b_{34})
(1+b_{56})}]\ .\]
 There is only one relation, namely $b_{34}b_{56}=b_{36}b_{45}$. We use the 
identity $\alpha =b_{36}+(1+b_{34})(1+b_{56})$ to exchange $b_{36}$ with
$\alpha$. Then the coordinate ring of $\mathcal{R}$ has the form 
$\mathbb{C}[\alpha ,\alpha ^{-1},b_{34},b_{45},b_{56} ]$ and there is only
one relation now. Define $x_1=b_{34}+1=tr(M_3M_4)-1, \ 
x_2=b_{45}+1=tr(M_4M_5)-1,\ x_3=b_{56}+1=tr(M_5M_6)-1$.
Then this relation reads  $x_1x_2x_3-x_1-\alpha x_2-x_3+  \alpha +1=0$  and defines a family $\mathcal{R}\rightarrow \mathcal{P}=
\mathbb{C}^*$ of cubic surfaces.

The locus in the affine space $AnalyticData$ of reducible data has two components. The first one is given by
$\alpha =1,\ b_1=b_3=b_5=0$ and the second one by $\alpha =1,\ b_2=b_4=b_6=0$.
These loci are mapped to the unique singular point 
$\alpha =1,\ x_1=x_2=x_3=1$ of $\mathcal{R}$. 

For $\alpha \neq 1$, the affine cubic surface, given by the above equation, has no singularities. The infinite part of the cubic surface consists of three lines. The three intersection points of these lines are the infinite  singularities. The cubic surface for $\alpha =1$ has one extra singular point, namely $x_1=x_2=x_3=1$. (This cubic surface is the Cayley surface). The type of the surface singularities  is $A_1$.

\subsection{Family $(-,-,5/2)$ and Painlev\'e  PI}

According to Definition and examples 1.10, a differential module of this type need not have a solution for the strong Riemann-Hilbert problem.
We deal here with the modules for which there is a solution, i.e., are
represented by a matrix differential equation 
$\frac{d}{dz}+A_0+A_1z+A_2z^2$ with nilpotent $A_2$ which can be normalized into  ${0\ 1\choose 0\ 0 }$. The map
$z\mapsto \lambda z+\mu$ is used to normalize the eigenvalues at
$\infty$ to $\pm (z^{5/2}+\frac{t}{2}\cdot z^{1/2})$.  Conjugation with a constant
matrix of the form ${1\ *\choose0\ 1 }$ leads to the normalization
\[\frac{d}{dz}+\left(\begin{array}{cc}p & t+q^2\\ -q&-p\end{array}\right)+
\left(\begin{array}{cc}0&q\\1&0\end{array}\right)z
+\left(\begin{array}{cc}0&1\\0&0\end{array}\right)z^2\ .\]  

The space $AnalyticData$ is given by the formal monodromy and 5 Stokes maps which are on a basis $e_1,e_2 $ of the formal solution 
 space at $\infty$ given by the matrices
 \[{0\ -1\choose 1\ \ \ 0 },\  {1\ 0\choose a_1\ 1}, \
{1\ a_2 \choose 0\ 1},\ {1\ 0 \choose a_3\ 1},\
{1\ a_4\choose 0\ 1 },\ {1\ 0 \choose a_5\ 1}.\]
 Their product is the topological monodromy and thus equal to
${1\ 0\choose 0\ 1}$. The base change $e_1,e_2\mapsto \lambda e_1,
\lambda e_2$ does not effect these matrices. Hence the coordinate
ring of $\mathcal{R}$ is generated by $a_1,\dots ,a_5$ and their relations are given by the above matrix identity.

 After eliminating $a_2$ by  $a_2=1+a_4a_5$ and $a_1$ by 
$a_1=-1-a_3a_4$,  one obtains for the remaining variables 
$a_3,a_4,a_5$ just one equation and $\mathcal{R}$ is a non singular affine cubic  surface with three lines at infinity, given by 
$a_3a_4a_5+a_3+a_5+1=0$.

\section{ The Painlev\'e equations}

\subsection{Finding the Painlev\'e equations}

For each of the ten families of Section 3, with the exception of $(0,0,0,0)$, which is the well known classical case leading to PVI, we want to derive a corresponding Painlev\'e equation $q''=R(q,q',t)$.  

We choose one of the other nine cases. A Zariski open part $\mathcal{M}^0$ of the corresponding moduli space is represented by a suitable matrix differential operator.
Recall that there is a morphism  $pr: \mathcal{M}^0\rightarrow T \times \Lambda$, where $T$ denotes the space of the `time variable' $t$ and the parameter space $\Lambda$ consists of the local exponents for the regular singular points and the 
constant term of the generalized local exponents at the irregular singular points.

Choose $\lambda \in \Lambda$, let $a\in \mathcal{P}$ be the image
of $\lambda$ in the parameter space of $\mathcal{R}$. Write $\mathcal{M}^0_\lambda 
=pr^{-1}(T\times \{\lambda \})$ and  $\mathcal{R}_a$ for the fibre of $\mathcal{R}\rightarrow \mathcal{P}$ at $a$. Then the Riemann--Hilbert map restricts to
$RH_\lambda :\mathcal{M}^0_\lambda \rightarrow \mathcal{R}_a$ and the fibres of $RH_\lambda$ are parametrized by $t$. In particular, $\mathcal{M}^0_\lambda $ has dimension 3. This space
is represented by an explicit family of differential operators $\frac{d}{dz}+A$, where the entries
of $A$ are rational functions in $z$ with coefficients depending on three explicit variables, say
$f,g,t$.  Later on we will make a rather special choice for $f,g$. 
  
An isomonodromic family  $\frac{d}{dz}+A=\frac{d}{dz}+A(z,t)$ on  $\mathbb{P}^1$, parametrized by $t$, is a fibre of some $RH_\lambda$. The earlier variables $f,g$ are now functions of $t$.
Let $S$ denote the singular locus.  On $\mathbb{P}^1\setminus S$ there exists a multivalued fundamental matrix $Y(z,t)$, i.e., $(\frac{d}{dz}+A(z,t))Y(z,t)=0$, normalized by $\det Y(z,t)=1$. By isomonodromy, $\frac{d}{dt}Y(z,t)$ and $Y(z,t)$ have the same behaviour for Stokes and monodromy and thus $B(z,t):=-\frac{d}{dt}Y(z,t)\cdot Y(z,t)^{-1}$ is univalued and extends in a meromorphic way at the set $S$. Moreover 
$B:=B(z,t)$ has trace 0 since $\det Y(z,t)=1$. Therefore  the entries of $B$ are rational functions in $z$ and are analytic in $t$. It follows that the operators $\frac{d}{dz}+A(z,t)$ and $\frac{d}{dt}+B(z,t)$ commute. This is equivalent to the identity
 \[ \frac{d}{dt}A=\frac{d}{dz}B+[B,A] \mbox{ and } tr(B)=0 \  .\]   
This equality is seen as a differential equation for matrices $B$, rational in $z$ and with trace 0. Assume (as we will in the examples) that $\frac{d}{dz}+A$ is irreducible, then $B$ is unique. Indeed, the difference $C$ of two solutions is  rational in $z$ and satisfies $\frac{d}{dz}C=[C,A]$. Thus $C(\frac{d}{dz}+A)=(\frac{d}{dz}+A)C$ and $C$ is an endomorphism of $\frac{d}{dz}+A$. By irreduciblity, $C$ is a multiple of the identity
and $C=0$ because $tr(C)=0$.

For the actual computation of $B$ for the cases of Subsection 2.2, the following remarks are useful.
If $z=c$ is a regular, or regular singular point (without resonance), then $B$ has no pole at $c$. If the Katz invariant $r=r(c)>0$ is an integer, and the top coefficient of the eigenvalues at $c$ do not depend on $t$, then $ord _c(B)\geq -r+1$. If however this top coefficient depends on $t$, then $ord_c(B)\geq -r$. If the Katz invariant $r(c)=m+\frac{1}{2}$ with integer $m\geq 0$, then 
$ord_c(B)\geq -m-1$.  The above matrix equation yields explicit differential equations for $f,g$ as functions of $t$, and an explicit $B$.\\

The symbols  $p,q$  denote a {\it preferred  choice} for the variables $f,g$. To define and find them we consider
a pair $(t,\lambda )\in T\times \Lambda$ and the 2-dimensional space $\mathcal{M}^0_{t,\lambda}:=
pr^{-1}(\{(t,\lambda )\}$. Let $\frac{d}{dz}+A$ be the corresponding matrix differential operator
and let a cyclic vector $e$ be given. The monic  scalar differential operator 
$L:=(\frac{d}{dz})^2+a_1\frac{d}{dz}+a_0$ defined by $Le=0$ has, in general, a number of new singularities,  called {\it apparent singular points}. In Subsection 4.2, we will find {\it good cyclic vectors} $e$, defined by the condition that there is only one apparent singular point. This singular
point, varying in the family $\mathcal{M}^0_{(t,\lambda )}$, is the choice for $q$. 
To make this explicit, suppose that $A={a\ \ \ b\choose c\ -a }$ and that the first basis vector is the cyclic vector $e$. Then 
\[ L=(\frac{d}{dz})^2-\frac{c'}{c}\cdot \frac{d}{dz}-a'-a^2-bc+a\cdot \frac{c'}{c}, \mbox{ where } a'=\frac{da}{dz} \mbox{ etc.}\]
Thus $c$ has as rational function in $z$ a simple zero at $q$ and this yields a pole at $q$ with residue 1 in the coefficient of $\frac{d}{dz}$ in $L$. Now $p$ is defined as the residue at $q$ of the `constant term'
$-a'-a^2-bc+a\cdot \frac{c'}{c}$ of $L$, multiplied by a factor 
$F\in \{1,q,q^2,q(q-1)\}$ depending on the family $\mathcal{M}$.  This factor is introduced  for geometrical reasons in connection with the
 Okamoto--Painlev\'e pairs \cite{STT02,T} (see the formulas of 4.3).\\

 A Zariski open, dense part of the space $\mathcal{M}^0_{(t,\lambda )}$ is now parametrized by $p,q$. On this space we introduce the {\it symplectic structure} by the closed 2-form $\frac{dp\wedge dq}{F}$ (with $F\in \{1,q,q^2,q(q-1)\}$) and thus $p,q$ are canonical coordinates. The Zariski open subset of the space $\mathcal{M}^0_{\lambda}$ is  parametrized by $p,q,t$. This space has a foliation given by the isomonodromy families, i.e., the fibres of $RH_\lambda$.
 There is an {\it Hamiltonian function} $H=H(p,q,t)$, rational in $p,q$ and $t$, such that this foliation coincides with the foliation deduced from  the  closed 2-form $\Omega =\frac{dp\wedge dq}{F} 
 - dH\wedge dt$ on $\mathcal{M}^0_\lambda$. 
More precisely, the vector field $v = \frac{\partial}{\partial t} + v_p \frac{\partial }{\partial p} + v_q \frac{\partial}{\partial q}$ describing ismonodromic families satisfies $v \cdot \Omega = 0$ (see \cite{STT02}, Section 6 and \cite{T}, Subsection (2.3)).

  The important fact is that for an isomonodromic family, $q$ as function of $t$ satisfies the 
  Painlev\'e equation  $q''=R(q,q',t)$ that we are looking for. The functions $p,q$ of $t$ satisfy the Hamiltonian equations, modified with the factor $F\in \{1,q,q^2,q(q-1)\}$, thus $p'=F\cdot \frac{\partial H}{\partial q},\   q'=-F\cdot \frac{\partial H}{\partial p}$.

\subsection{Apparent singularities}
\label{ss:apparent}
 
 Let  $M$ denote a differential module over $\mathbb{C}(z)$ of rank 2 with $\det M\cong {\bf 1}$,
with singular points $0,1,\infty$ and  represented by  a connection 
$(\mathcal{V},\nabla )$ with $\mathcal{V}$ free and 
$\nabla : \mathcal{V}\rightarrow \Omega (n_0[1]+n_1[1]+n_\infty [\infty ])\otimes \mathcal{V}$ for
integers $n_0,n_1,n_\infty \geq 1$. Put  $V:=H^0(\mathbb{P}^1, \mathcal{V})$. Then $M=\mathbb{C}(z)\otimes V$ and $\partial :=\nabla _{\frac{d}{dz}}=\frac{d}{dz}+B_0+B_1+B_\infty$ with    
$B_0, \ B_1,\ B_\infty$ polynomials in $z^{-1},\ (z-1)^{-1},\ z$ of degrees $\leq n_0, \ n_1, -1+n_\infty$ and
with coefficients in ${\rm End}(V)$. The free module $N:=\mathbb{C}[z,\frac{1}{z(z-1)}]\otimes V$ over  $\mathbb{C}[z,\frac{1}{z(z-1)}]$ is invariant under $\partial$ and can be considered as a differential
module over  $\mathbb{C}[z,\frac{1}{z(z-1)}][\frac{d}{dz}]$. \\

Let $e\in M=\mathbb{C}(z)\otimes V$ be a cyclic vector, producing  the  scalar equation\\
$(\partial ^2+a_1\partial +a_0)e=0$. The poles of $a_1,a_0$, different from $0,1,\infty$ are called the {\it apparent singularities}. Let $s\neq 0,1,\infty $ have local parameter $u=z-s$. The elements  $e\in \mathbb{C}((u))\otimes V$ are written as formal Laurent series $\sum _{n\geq *}v_nu^n$ with all $v_n\in V$. 
Now $ord_s(e)$, the order of $e\neq 0$ at $s$, is defined to be the minimal integer $d$ with $v_d\neq 0$. \\

We will use the second exterior power $\Lambda ^2 N=\mathbb{C}[z,\frac{1}{z(z-1)}]\otimes \Lambda ^2V$. Let $e\in N$ be a cyclic vector and  $N^0\subset N$ the submodule generated by $e$ and $\partial e$. Then $\Lambda ^2N^0 =b\cdot \Lambda ^2N$ for some monic polynomial $b$ with $b(0)\neq 0,\ b(1)\neq 0$. 
The following lemma is an explicit calculation corresponding to \cite{IIS2}, Subsection(4.2).

\begin{lemma} The zero's of $b$ are the apparent  singular points. 
\end{lemma}
\begin{proof} We fix a point $s\neq 0,1,\infty$ and show that $ord _s(b)>0$ if and only
if $s$ is an apparent singularity.  First we consider the case that $ord _s(e)=0$.
Since $s$ is non singular, $\mathbb{C}[[u]]\otimes V$ has a free basis $w_1,w_2$ over $\mathbb{C}[[u]]$ with $\partial w_1=\partial w_2=0$. Write $e=c_1w_1+c_2w_2$ with $\min (ord (c_1),ord (c_2))=0$.
 We may suppose that $ord (c_1)=0$ and $ord( c_2)=m\geq 1$. The equation for the
 cyclic vector is $\partial ^2+a_1\partial +a_0$ with
 $a_1=\frac{(-c_1''c_2+c_1c_2'')}{(c_1c_2'-c_1'c_2)},\ a_0=\frac{(-c_1''c_2'+c_1'c_2'')}{(c_1c_2'-c_1'c_2)}$. If $m=1$, then $ord_s(c_1c_2'-c_1'c_2)=0$ and $s$ is not
 an apparent singularity. If $m\geq 2$, then \begin{small}
 \[ord_s(c_1c_2'-c_1'c_2)=m-1,\  ord_s(-c_1''c_2+c_1c_2'')=m-2, \ 
 ord_s(-c_1''c_2'+c_1'c_2'')\geq m-2.\] \end{small}
 Thus $ord _s(a_1)=-1,\ ord_s(a_0)\geq -1$ and $s$ is an apparent singularity.
 
 Suppose now that $e= u^{n} f$, $n\geq 1,\ ord_s(f)=0$.  
The equation for $e$ is obtained from the scalar equation 
$\partial ^2+a_1\partial +a_0$ for $f$ by  the substitution 
$\partial \mapsto \partial -nu^{-1}$ and reads 
$\partial ^2+(-2nu^{-1}+a_1)\partial +(n^2+n)u^{-2}-na_1u^{-1}+a_0$. This introduces a pole if there was no pole before and  a pole of order 2 if there was already a pole.

For $e$ with $ord_s(e)$ one has 
$e\wedge \partial e= (c_1c_2'-c_1'c_2)w_1\wedge w_2$ and \\
$ord _s(c_1c_2'-c_1'c_2)=m-1$ and for $e=u^nf$ one has 
$e\wedge \partial e=u^{2n}f\wedge \partial f$. From this the statement follows.\end{proof}
 
By multiplying a given cyclic vector $e$ with $\prod _{s\neq 0,1}(z-s)^{-ord_s(e)}$, the number of zero's of $b$ (counted with multiplicity) goes down. The cyclic vectors with minimal degree for $b$ have the form $e\in \mathbb{C}[z,\frac{1}{z(z-1)}]\otimes V$ (or even, after multiplying with $z^*(z-1)^*$ one has $e\in \mathbb{C}[z]\otimes V$) and $ord_s(e)=0$ for all $s\neq 0,1$. We note that the condition $ord _s(e)=0$ is equivalent to $b$ has at most  simple zero's. We call a cyclic vector $e$ 
{\it good} if the corresponding $b$ has degree one (and thus there is only one apparent singularity).\\

{\it Application of Lemma 4.1 for finding good cyclic vectors $v\in V$}.\\ 

\noindent 
{\it Family $(0,0,1)$},  $\partial =\frac{d}{dz}+z^{-1}A_0+(z-1)^{-1}A_1+A_\infty$ and $A_0,A_1,A_\infty \in {\rm End}(V)$.\\
We only consider good cyclic vectors $v\in V$. The operator $\partial $ is multiplied by
$z(z-1)$. The condition that $v$ produces only one apparent singularity is equivalent to
$v\wedge ( z(z-1)A_\infty (v)+(z-1)A_0(v)+zA_1(v))$ (as element of $\mathbb{C}[z]\otimes \Lambda ^2V$)  has only one zero $\neq 0,1$. This is equivalent to $v$ is an eigenvector of $A_\infty$ or $A_0$ or $A_1$. Thus in total there are 6 good cyclic vectors, in  general.\\
\noindent {\it Family} $(0,0,1/2)$,  same formula for $\partial$ but with $A_\infty$
nilpotent.  The good cyclic vectors are the eigenvectors of the matrices $A_0,A_1,A_\infty$. There are, in general, 5 good cyclic vectors.\\
\noindent {\it Family} $(1,-,1)$. Now we work with a module over $\mathbb{C}[z,z^{-1}]$. 
Here $z^2\partial =z^2\frac{d}{dz}+A_2+A_1z+A_0z^2$. The condition on $v$ is
$v\wedge (A_2(v)+A_1(v)z+A_0(v)z^2)$ has only one zero $\neq 0$. Thus $v$ is an eigenvector of $A_2$ or of $A_0$. In general, there are total 4 good cyclic vectors and they come in pairs.\\
\noindent {\it Family} $(1/2,-,1)$. As before, but now $A_2$ is nilpotent and, in general, there are in 3 good cyclic vectors. \\
\noindent {\it Family} $(1/2,-,1/2)$. As before, but both $A_2$ and $A_0$ are nilpotent. Thus, in general, 2 good cyclic vectors.\\
\noindent {\it Family} $(0,-,2)$. $z\partial =z\frac{d}{dz}+A_0+zA_1+z^2A_2$. Then 
a  good cyclic vector is eigenvector of $A_0$ or of $A_2$. Thus, in general 4 good cyclic
vectors.\\
\noindent {\it Family} $(0,-,3/2)$. As before, but now, in general, 3 good cyclic vectors because $A_2$ is nilpotent.\\
\noindent {\it Family} $(-,-,3)$. Now we work over the ring $\mathbb{C}[z]$ and 
$\partial =\frac{d}{dz}+A_0+zA_1+z^2A_2$. The possible $v$ are eigenvector for $A_2$. Thus 2 good cyclic vectors.\\ 
\noindent {\it Family} $(-,-,5/2)$. As before, but only one good cyclic vector, since 
$A_2$ is nilpotent.\\

\begin{remarks}  {\rm 
(1) In general there are more complicated good cyclic vectors, than those in $V$. However
for $(-,-,5/2)$ there is only one good cyclic vector.\\
(2) If a cyclic vector $v\in V$ is eigenvector of two of the matrices, then
the corresponding scalar equation has no apparent singularity.
}\end{remarks} 

\newcommand{\C}{{\mathbb{C}}}
\newcommand{\cM}{{\mathcal M}}

\subsection{Family $(0, 0, 1)$ and Painlev\'e V,  PV($\tilde{D}_5$)}
\label{ss:p5-d5}

In terms of the parameters $(p,q)$ of Subsection 4.1, the family reads

\begin{table}[h]
\begin{center}
\begin{tabular}{|c|c|c|c|} \hline
The singularities $z$ & $0$ & $1$ & $\infty $ \\ \hline 
Katz invariant & $0$ & $0$ &  $1$ \\ \hline 
  & & & \\
generalized local exponents & $\pm \frac{\theta_0}{2} $ & $\pm \frac{\theta_1}{2}$ & 
$\pm ( \frac{t}{2} z + \frac{\theta_{\infty}}{2})$ \\
  & & & \\   \hline 
\end{tabular}

\label{tab:p5-d5}
\end{center}
\end{table}

\begin{equation}\label{eq:p5-linear}
\nabla_{\frac{d}{dz}}= \frac{d}{dz}+\frac{A_0}{z}+\frac{A_1}{z-1}+A_{\infty} =\frac{d}{dz} + \frac{1}{z(z-1)}A \ \ \mbox{ with }
\end{equation}

 \[
A_0=\left(
\begin{array}{cc}
 -p-\frac{1}{2} q \left(q t-t+\theta
   _{\infty }\right) &
   \frac{(q-1)\left( \left(p+\frac{1}{2} q \left(q
   t-t+\theta _{\infty
   }\right)\right){}^2-\frac{\theta
   _0^2}{4}\right)}{ q } \\
 -\frac{q}{q-1} & p+\frac{1}{2} q \left(q t-t+\theta
   _{\infty }\right)
\end{array}
\right),
\] 
\[
A_1=\left(
\begin{array}{cc}
 p+\frac{(q-1)( q
   t+\theta _{\infty })}{2} &
   \left(\frac{\theta _1}{2}\right)^2-\left( p+\frac{(q-1)(
   q t+ \theta _{\infty
   })}{2}\right){}^2 \\
 1 &  -\left(p+\frac{(q-1)( q
   t+\theta _{\infty })}{2}\right)
\end{array}\right)
\]
\[
 \quad A_{\infty} = \left(
\begin{array}{cc}
-\frac{t}{2} & 0  \\
 0 & \frac{t}{2}  
\end{array}
\right). 
\]
Write $ A=z(z-1)(A_{\infty}+ A_{0}/z + A_{1}/(z-1))= \left(
\begin{array}{ll}
a(z) & b(z) \\
c(z) & -a(z) 
\end{array}
\right)$ with 
\footnotesize
\begin{eqnarray*}
 a(z) & = & p+\frac{1}{2} \left(-2 q
   t+t-\theta _{\infty }\right) (z-q) -\frac{1}{2} t (z-q)^2, \\
b(z) & = &  -\frac{z\left(\left(p+\frac{1}{2} (q-1) q\right)^2-\frac{\theta
   _1^2}{4}+(q-1) \left(\frac{\theta
   _0^2}{4}-\frac{\theta _1^2}{4}-\frac{q \theta
   _{\infty }^2}{4}\right)\right)}{q} \\
&& -\frac{(q-1) \left(p+\frac{1}{2} \left((q-1) q
   t-\theta _0+q \theta _{\infty }\right)\right)
   \left(p+\frac{1}{2} \left((q-1) q t+\theta _0+q
   \theta _{\infty }\right)\right)}{q}, 
\\
c(z) & = &   -\frac{ (z-q)}{q-1}.
\end{eqnarray*}

\normalsize
The first basis vector is chosen as cyclic vector and following Subsection 4.1,
$q$ is the unique zero of $c(z)$ and $p=a(q)$.

The parameter $(p, q)$ gives  canonical coordinates on an affine Zariski open set $U_0\cong \C \times \left(\C \setminus \{ 0, 1 \}\right) $ of the moduli space 
$\cM_{t, \lambda}$ of the connections with fixed $t$ and fixed generalized local exponents $\lambda$.  The symplectic form 
on $U_0$, which is natural from the view point of Okamoto--Painlev\'e pair,  
 is given by $ \frac{dp \wedge dq}{q(q-1)}.$

The matrix different operator $\frac{d}{dt}+B$, commuting with 
 $\nabla _{\frac{d}{dz}}$, has the form 
\[ B= z B_0  + \frac{1}{t}B_1 \mbox{ where }
B_0= \left(
\begin{array}{cc}
 -\frac{1}{2} & 0 \\
 0 & \frac{1}{2}
\end{array}
\right) \mbox{ and } B_1=
\]  
\[
\left(
\begin{array}{cc}
 -\frac{p}{q-1}-\frac{1}{2} (q-1)
   t-\frac{\theta _{\infty }}{2} & -\frac{\left(p+\frac{1}{2} (q-1) q
   t\right)^2-\frac{\theta _1^2}{4}+(q-1)
   \left(\frac{\theta _0^2}{4}-\frac{\theta
   _1^2}{4}-\frac{q \theta _{\infty
   }^2}{4}\right)}{q } \\
 -\frac{1}{q-1} & \frac{p}{q-1}+\frac{1}{2} (q-1)
   t+\frac{\theta _{\infty }}{2}
\end{array}
\right).\]

>From $[\frac{d}{dt}+B,\nabla _{\frac{d}{dz}}]=0$ one deduces the following.\\
\noindent
{\bf Painlev\'e V, PV($\tilde{D}_5$)}
\small
\begin{equation}\label{eq:p5} 
\left\{ \begin{array}{lcl}
\displaystyle{\frac{dq}{dt}} & = & \displaystyle{\frac{2 p}{t}} \\
\displaystyle{\frac{dp}{dt}} & = & 
\displaystyle{\frac{(2 q-1) p^2}{(q-1) q t}
+\frac{\theta_0^2(q-1)^2-\theta_1^2 q^2}{4q (q-1) t}
+\frac{1}{4} (q-1) q \left(2 q
   t-t+2 \theta _{\infty }-2\right)} 
\end{array}\right. 
\end{equation}
\normalsize

\noindent
The equation (\ref{eq:p5}) is equivalent to the second order differential equation of $q$
\begin{equation}\label{eq:p5-s}
q'' =  \frac{(2 q-1)
   \left(q'\right)^2}{2 (q-1)q}
    -\frac{q'}{t}
    +\frac{(q-1) q \left(2 q t-t+2\theta _{\infty }-2\right)}{2t}
   +\frac{\theta _0^2}{2 qt^2}
   +\frac{\theta _1^2}{2 (q-1)t^2}\ .
\end{equation}
By a rational transformation of $(p,q)$, this can be transformed  into the classical Painlev\'e V in \cite{JM}.

Now we compute the Hamiltonian function $H_V=H_V(p, q, t, \theta)$ for (\ref{eq:p5}), which is a rational function of $(p, q, t)$. It is defined by the property
that the foliation given by the  2-form $\Omega = \frac{dp \wedge dq}{q(q-1)} - dH_V \wedge dt$ on $U_{0} \times (\C \setminus \{0\})$ coincides with the
foliation given by isomonodromy. The latter is given by the vector field $v$, 
equivalent to (\ref{eq:p5}), satisfying $v\cdot \Omega =0$ and of the form
\[
v = \frac{\partial}{\partial t} + v_p \frac{\partial}{\partial p} + v_q \frac{\partial}{\partial q}, 
\mbox{ with } v_p = \frac{dp}{dt},\  v_q = \frac{dq}{dt}.\]  
\[ \mbox{Now }
0=v \cdot \Omega = dH_V + v_p \frac{dq}{q(q-1)} -  v_q \frac{dp}{q(q-1)}, 
\mbox{ is equivalent to} \] 
\begin{equation}
\left\{
\begin{array}{ccc}
\displaystyle{\frac{dp}{dt}} & =& q(q-1)\frac{\partial H_V}{\partial q} \\
\displaystyle{\frac{dq}{dt}} &=&  -q(q-1)\frac{\partial H_V}{\partial p} \\
\end{array} \right.
\end{equation}
Comparing this with (\ref{eq:p5}), one obtains the following expression for $H_V$  
\begin{equation}
H_V(p, q, t)=-\frac{p^2}{(q-1) q t}-\frac{\theta _0^2}{4 q
   t}+\frac{\theta _1^2}{4 (q-1) t}+\frac{1}{4} q
   \left(q t-t+2 \theta _{\infty }-2\right)\ .
\end{equation}

\normalsize
\subsection{\bf Family $(0, 0, 1/2)$ and degenerate Painlev\'e V,  
${\rm PV}_{\rm deg}$($\tilde{D}_6$).}
\label{ss:degp5-d6}

\normalsize
${\rm PV}_{\rm deg}$ stands for `degenerate PV' (cf. \cite{OO}) which turns out to be equivalent to Painlev\'e equation of type PIII($\tilde{D}_{6}$).
The first basis vector is chosen as cyclic vector, the $(p,q)$ are as in Subsection
4.1 and the family reads

\begin{table}[h]
\begin{center} \small
\begin{tabular}{|c|c|c|c|} \hline
The singular points  $z$ & $0$ & $1$ & $\infty $ \\ \hline 
Katz invariant & $0$ & $0$ &  $\frac{1}{2}$ \\ \hline 
  & & & \\
generalized local exponents & $\pm \frac{\theta_0}{2} $ & $\pm \frac{\theta_1}{2}$ & 
$\pm  t z^{\frac{1}{2}}$ \\
  & & & \\   \hline 
\end{tabular}

\label{tab:deg-p5}
\end{center}
\end{table}

\begin{equation}\label{eq:pdeg5-linear}
\nabla_{\frac{d}{dz}}= \frac{d}{dz}+\frac{A_0}{z}+\frac{A_1}{z-1}+A_{\infty} 
= \frac{d}{dz} +\frac{1}{z(z-1)}A \mbox{ with }
\end{equation}
\[
A_0=\left(
\begin{array}{cc}
 -p & \frac{\theta_0^2-4 p^2}{4 q} \\
 q & p
\end{array}
\right)
 \quad 
A_1=\left(
\begin{array}{cc}
 p & \frac{4 p^2-\theta _1^2}{4 (q-1)} \\
 1-q & -p
\end{array}
\right)
\quad 
A_\infty =\left(
\begin{array}{ll}
 0 & t^2 \\
 0 & 0
\end{array}
\right)
\]
\[
A=
\left(
\begin{array}{cc}
 p & L \\
 z-q & -p
\end{array}
\right), \
L:=\frac{(q+z-1) p^2}{(q-1) q}+\frac{(z-1) \theta _0^2}{4 q}-\frac{z \theta _1^2}{4
   (q-1)}+t^2 (z-1) z
\]
The operator $\frac{d}{dt}+B$ with $[\frac{d}{dt}+B,\nabla _{\frac{d}{dz}}]=0$ satisfies $ B= z B_0  + B_1$, where
\[B_0= \left(
\begin{array}{cc}
 0 & 2t \\
 0 & 0
\end{array}
\right), 
B_1=\left(
\begin{array}{cc}
 0 & \displaystyle{\frac{2 p^2}{(q-1) q
   t}+\frac{\theta _0^2}{2 q
   t}-\frac{\theta _1^2}{2 (q-1) t}+2
   (q-1) t} \\
 \displaystyle{\frac{2}{t}} & 0
\end{array}
\right).
\]

Solving $[\frac{d}{dt}+B,\nabla _{\frac{d}{dz}}]=0$ with {\em Mathematica}
yields the following. \\
\noindent
{\bf Degenerate Painlev\'e V, ${\rm PV}_{\rm deg}$($\tilde{D}_6$)}
\begin{equation} \label{eq:degp5-sys}
\left\{
\begin{array}{lcl}
\displaystyle{\frac{dq}{dt}} & = & \displaystyle{\frac{4p}{t}} \\
\displaystyle{\frac{dp}{dt}} & = & \displaystyle{  \frac{2(2q-1) p^2}{(q-1) q t}+\frac{(q-1)\theta_0^2}{2 q t}-\frac{q\theta_1^2}{2 (q-1) t}+2 q(q-1) t}
\end{array}
\right.
\end{equation}

\begin{equation}
q''=\frac{(2 q-1) \left(q'\right)^2}{2 (q-1) q}-\frac{q'}{t}+ \frac{2 (q-1) \theta _0^2}{q t^2}-\frac{2 q \theta _1^2}{(q-1)
   t^2}+8 (q-1)q
\end{equation}

The  2-form on $\C \times (\C \setminus \{ 0, 1 \}) \times (\C \setminus \{ 0 \})$, natural for the Okamoto--Painlev\'e pair of type  $\tilde{D}_6$,  is given by 
\[ \Omega =\frac{dp \wedge dq}{q(q-1)} - dH_{dV}\wedge dt, 
\mbox{ where }H_{dV} = H_{dV}(p, q,t,  \theta) \mbox{ is equal to}\]   
\begin{eqnarray}\label{eq:pd5}
H_{dV}(p, q,t,  \theta) & = & -\frac{2 p^2}{(q-1) q t}-
\frac{\theta _0^2}{2 q t}+\frac{\theta
   _1^2}{2 (q-1) t}+2 q t,  \\
   & = &  \frac{2(p^2- \left(\frac{\theta_0}{2}\right)^2)}{tq} 
   - \frac{2(p^2- \left(\frac{\theta_1}{2}\right)^2)}{t(q-1)}+2 q t.
\end{eqnarray}
We note that equation (\ref{eq:degp5-sys}) is equivalent to the  Hamiltonian system
\begin{equation}
\left\{
\begin{array}{ccc}
\displaystyle{\frac{dq}{dt}} &=& 
\displaystyle{-q(q-1)\frac{\partial H_{dV}}{\partial p}}, \\
\displaystyle{\frac{dp}{dt}} & =& 
\displaystyle{ q(q-1)\frac{\partial H_{dV}}{\partial q}}. 
\end{array} \right .  
\end{equation}
 \normalsize

\subsection{\bf  Family $(1,-,1)$ and Painlev\'e III,  PIII($\tilde{D}_6$).}
\label{ss:p3-d6}

As before, the first basis vector is chosen to be the cyclic vector,  $(p,q)$
are as introduced in Subsection 4.1 and the operator $\frac{d}{dt}+B$
commuting with $\nabla _{\frac{d}{dz}}$ has the form 
$B=zB_0  + B_1 + \frac{1}{z}B_2$. We present now the data.
\begin{equation}\label{eq:p3-d6}
\nabla_{\frac{d}{dz}}= \frac{d}{dz}+\frac{1}{z^2}A_0+\frac{1}{z}A_1+A_{2} 
= \frac{d}{dz} +\frac{1}{z^2}A. \end{equation}
\begin{table}[h]
\begin{center}
\begin{tabular}{|c|c|c|} \hline
The singular points  $z$ & $0$ &  $\infty $ \\ \hline 
Katz invariant & $1$ &   $1$ \\ \hline 
   & & \\
generalized local exponents & $\pm (\frac{t}{2} z^{-1} +\frac{\theta_{0}}{2})$ &  
$\pm (\frac{t}{2} z +\frac{\theta_{\infty}}{2})$ \\
  & &\\   \hline 
\end{tabular}
\label{tab:p3-d6}
\end{center}
\end{table}
 \[
A_0=\left(\begin{array}{cc} 
\frac{1}{2} \left(-t  q^2-\theta_{\infty } q+2  p\right) & 
\frac{t^2 q^4+2   t \theta _{\infty }   q^3+\theta _{\infty }^2  q^2-4 p t q^2-4 p \theta  _{\infty } q+4 p^2-t^2}{4 q } \\ -q  & 
\frac{1}{2} \left(t  q^2+\theta _{\infty } q-2 p\right)
\end{array}\right),  \] 
\[A_1 
= \left(\begin{array}{cc} \frac{\theta _{\infty }}{2} &  \frac{t^2 q^4-
\theta  _{\infty }^2 q^2-4 p t   q^2-2 t \theta _0 q+4  p^2-t^2}{4 q^2 } \\
 1 & -\frac{\theta _{\infty  }}{2}\end{array}\right),  \ \ 
 A_2 =  \left( \begin{array}{cc} \frac{t}{2} & 0 \\ 0 & -\frac{t}{2}\end{array} \right),  \]
\[A=\left(
\begin{array}{cc}
 p+\frac{1}{2} (z-q) \left(q
   t+z t+\theta _{\infty
   }\right) & L 
   \\
 (z-q) & -p-\frac{1}{2} (z-q)
   \left(q t+z t+\theta
   _{\infty }\right)
\end{array}
\right), 
\]
\footnotesize
\begin{eqnarray*}& 
L & = \frac{q \left(t q^2+\theta _{\infty } q-2 p+t\right)
   \left(t q^2+\theta _{\infty} q-2 p-t\right)}{4q^2} \\ && \hspace{3cm} + \frac{z
   \left(t^2 q^4-\left(\theta_{\infty }^2+4 p
   t\right) q^2-2 t \theta_0 q+4 p^2-t^2\right)}{4q^2}. 
\end{eqnarray*}
\normalsize
\[B_0 =  \left(
\begin{array}{cc}
 \frac{1}{2} & 0 \\
 0 & -\frac{1}{2}
\end{array}
\right),  \\
B_1  = 
\left(
\begin{array}{cc}
 q +\frac{\theta_{\infty}}{2t}  & \frac{t^2 q^4-\theta
   _{\infty }^2 q^2-4 p t
   q^2-2 t \theta _0 q+4
   p^2-t^2}{4 q^2  t} \\
 \frac{1}{t} & -q -\frac{\theta_{\infty}}{2t}
\end{array}
\right),    \]
\[ B_2  =  
\left(
\begin{array}{cc}
 \frac{t q^2+\theta _{\infty }
   q-2 p}{2 t} & \frac{-4
   p^2+\left(1-q^4\right)
   t^2+2 q^2 t \left(2 p-q
   \theta _{\infty }\right)+q
   \theta _{\infty } \left(4
   p-q \theta _{\infty
   }\right)}{4 q  t} \\
 \frac{q }{t} & \frac{-t
   q^2-\theta _{\infty } q+2
   p}{2 t}
\end{array}
\right). \] 
Solving the equation $[\frac{d}{dt}+B,\nabla _{\frac{d}{dz}}]=0$ yields the following.\\

\noindent
{\bf Painlev\'e III, PIII($\tilde{D}_6$). }
\begin{equation} \label{eq:p3-d6-sys}
\left\{
\begin{array}{lcl}
\displaystyle{\frac{dq}{dt}} & = & \displaystyle{\frac{4 p+q}{t}} \\
  &  &  \\
\displaystyle{\frac{dp}{dt}} & = & \displaystyle{ 
\frac{4 p^2}{q
   t}+\frac{p}{t}+t q^3 + q^2-\frac{t}{q}-\theta _0+q^2 \theta
   _{\infty }}
\end{array}
\right.
\end{equation}
The system (\ref{eq:p3-d6-sys}) is equivalent to the following 
second order equation. 
\begin{equation}\label{eq:p3-d6-single}
q''=  \frac{\left(q'\right)^2}{q}-\frac{q'}{t}
-\frac{4 \theta_0}{t}+
\frac{4 (\theta _{\infty}+1) q^2}{t}
+4 q^3 -\frac{4}{ q}.
\end{equation} 
 The equations (\ref{eq:p3-d6-sys}) or (\ref{eq:p3-d6-single}) are defined on 
$ \C \times \C \setminus \{ 0 \} \times \C \setminus \{ 0 \} $ and the 2-form $\Omega$
on this affine open set  is 
\[
\Omega =\frac{dp \wedge dq}{q^2} - dH_{III}\wedge dt
\]
where $H_{III} = H_{III}(p, q,t,  \theta) $ is 
 a Hamiltonian function for PIII given by
\begin{equation}\label{eq:p3-d6}
H_{III}(p, q,t,  \theta)  =  -\frac{2 p^2}{q^2 t}-\frac{p}{q
   t}+q+\frac{q^2 t}{2}+\frac{t}{2
   q^2}+\frac{\theta _0}{q}+q \theta
   _{\infty } 
\end{equation}
As before, the equation (\ref{eq:p3-d6-sys}) is equivalent to the  Hamiltonian system:
\begin{equation}
\left\{
\begin{array}{ccc}
\displaystyle{\frac{dq}{dt}} &=& 
\displaystyle{-q^2\frac{\partial H_{III}}{\partial p}}, \\
\displaystyle{\frac{dp}{dt}} & =& 
\displaystyle{ q^2
\frac{\partial H_{III}}{\partial q}}. 
\end{array} \right.
\end{equation}

\normalsize

\subsection{ Family $(1/2,-, 1)$:  Painlev\'e III$D_7$,  PIII($\tilde{D}_7$).}
\label{ss:p3-d7}

This family can be written as 
\begin{equation}\label{eq:p3-d7-linear}
\nabla_{\frac{d}{dz}}= \frac{d}{dz}+\frac{1}{z^2}A_0+\frac{1}{z}A_1+A_2 
= \frac{d}{dz} +\frac{1}{z^2}A.
\end{equation}
The items $p,q,B$ are as before and the form of $B$ is $zB_0+B_1$. We give now the
explicit data and the results on the Painlev\'e equation and the Hamiltonian. 
\begin{table}[h]
\begin{center}
\begin{tabular}{|c|c|c|} \hline
The singular points  $z$ & $0$ &  $\infty $ \\ \hline 
Katz invariant & $1/2$ &   $1$ \\ \hline 
   & & \\
generalized local exponents & $\pm 
z^{-1/2}$ &  
$\pm (\frac{t}{2} z +\frac{\theta_{\infty}}{2})$ \\
  & &\\   \hline 
\end{tabular}
\label{tab:p3-d7}
\end{center}
\end{table}
 \[ A_0 =\left(
\begin{array}{cc}
 \frac{1}{2} \left(-t q^2-\theta _{\infty } q+2
   p\right) & \frac{\left(t q^2+\theta _{\infty }
   q-2 p\right){}^2}{4 q } \\
 -q  & \frac{1}{2} \left(t q^2+\theta _{\infty }
   q-2 p\right)
\end{array}
\right)
\]
\[
A_1 =  \left(\begin{array}{cc} \frac{\theta _{\infty }}{2} & \frac{t^2
   q^4-\theta _{\infty }^2 q^2-4 p t q^2-4 q+4   p^2}{4 q^2} \\
 1 & -\frac{\theta _{\infty }}{2}\end{array} \right), \ \
  A_2 =  \left( \begin{array}{cc} \frac{t}{2} & 0 \\
 0 & -\frac{t}{2}\end{array} \right) \mbox{ and } A=\]
\[\begin{footnotesize} \left(
\begin{array}{cc}
 p+\frac{1}{2} (z-q) \left(q t+z t+\theta _{\infty
   }\right) & 
   \frac{q(-t q^2-\theta _{\infty } q+2 p)^2 +z (t^2  q^4-\theta _{\infty }^2 q^2-4 p t q^2-4 q+4 p^2)}{4 q^2 } \\
 (z-q) & -p-\frac{1}{2} (z-q) \left(q t+z t+\theta_{\infty }\right)
\end{array}
\right) \end{footnotesize} \]

\footnotesize

\normalsize
\[B_0  =   \left(
\begin{array}{cc}
 \frac{1}{2} & 0 \\
 0 & -\frac{1}{2}
\end{array}
\right),  \ \
B_1 =  
\left(
\begin{array}{cc}
\frac{q}{2}+\frac{\theta _{\infty }}{2 t} & \frac{t^2 q^4-\left(\theta _{\infty }^2+4 p
   t\right) q^2-4 q+4 p^2}{4 q^2 t} \\
 \frac{1}{t} &- \frac{q}{2}-\frac{\theta _{\infty }}{2 t}
\end{array}
\right). \]
\noindent
{\bf Painlev\'e III$D_7$, PIII($\tilde{D}_7$)}
\begin{equation} \label{eq:p3-d7-sys}
\left\{
\begin{array}{lcl}
\displaystyle{\frac{dq}{dt}} & = & \displaystyle{\frac{2 p}{t}} \\
  &  &  \\
\displaystyle{\frac{dp}{dt}} & = & \displaystyle{ \frac{2
   p^2}{t q}+ \frac{t q^3}{2}+\frac{1}{2} \left(\theta
   _{\infty }+1\right) q^2-\frac{1}{t}
}
\end{array}
\right.
\end{equation}
The system (\ref{eq:p3-d7-sys}) is equivalent to the following 
second order equation. 
\begin{equation}\label{eq:p3-d7-single}
q''= \frac{\left(q'
   \right)^2}{q}-\frac{q'}{t}+\frac{\left(\theta _{\infty }+1\right)
   q^2}{t}+ q^3-\frac{2}{t^2}
\end{equation} 

\begin{eqnarray}\label{eq:p3-d7}
H_{IIID_7}(p, q,t,  \theta) & = & -\frac{p^2}{q^2 t}+\frac{q^2 t}{4}+\frac{1}{2} q
   \left(\theta _{\infty }+1\right)+\frac{1}{q t} .
\end{eqnarray}
\begin{equation}
\left\{
\begin{array}{ccc}
\displaystyle{\frac{dq}{dt}} &=& 
\displaystyle{-q^2\frac{\partial H_{IIID_7}}{\partial p}}, \\
\displaystyle{\frac{dp}{dt}} & =& 
\displaystyle{ q^2
\frac{\partial H_{IIID_7}}{\partial q}}. 
\end{array} \right.
\end{equation}

\subsection{\bf Family $(1/2,-, 1/2)$: Painlev\'e III$D_8$,  PIII($\tilde{D}_8$).} \label{ss:p3-d8}

We present the data and the results of the computation.
\begin{equation}\label{eq:p3-d8-linear}
\nabla_{\frac{d}{dz}}= \frac{d}{dz}+\frac{1}{z^2}A_0+\frac{1}{z}A_1+A_{2} 
= \frac{d}{dz} +\frac{1}{z^2}A.
\end{equation}

\begin{table}[h]
\begin{center}
\begin{tabular}{|c|c|c|} \hline
The singular points  $z$ & $0$ &  $\infty $ \\ \hline 
Katz invariant & $1/2$ &   $1$ \\ \hline 
   & & \\
generalized local exponents & $\pm 
\sqrt{t}\cdot  z^{-1/2}$ &  
$\pm z^{1/2}$ \\
  & &\\   \hline 
\end{tabular}
\label{tab:p3-d8}
\end{center}
\end{table}
\[A_0  =  \left( \begin{array}{cc}0 & 0 \\ -q & 0\end{array}\right),
\ \
A_1 =  \left(\begin{array}{cc} \frac{p}{q} & -\frac{t}{q} \\
 1 & -\frac{p}{q}\end{array}\right), \ \
 A_2 = \left(\begin{array}{cc} 0 & 1 \\ 0 & 0\end{array}\right), 
 \] 
\[ A=\left( \begin{array}{cc}
 \frac{p z}{q} & \frac{z (q z-t)}{q} \\ z-q & -\frac{p z}{q}\end{array}\right),\
B = B_0  + \frac{1}{z} B_1 \mbox{ where } \]
\[ B_0  =   \left( \begin{array}{cc} 0 & \frac{1}{q} \\ 0 & 0 \end{array} \right),  \ \
B_1  = \left( \begin{array}{cc} 0 & 0 \\ \frac{q}{t} & 0\end{array} \right). \] 
\noindent
{\bf Painlev\'e III$D_8$, PIII($D_8$).}
\begin{equation} \label{eq:p3-d8-sys}
\left\{
\begin{array}{lcl}
\displaystyle{\frac{dq}{dt}} & = & \displaystyle{ \frac{2 p+q}{t}} \\
  &  &  \\
\displaystyle{\frac{dp}{dt}} & = & \displaystyle{\frac{2 p^2}{q
   t}+\frac{p}{t}+\frac{q^2}{t}-1=\frac{q^3+p q-t q+2 p^2}{q t}}
\end{array}
\right.
\end{equation}
The system (\ref{eq:p3-d8-sys}) is equivalent to the following 
second order equation. 
\begin{equation}\label{eq:p3d8-sing}
q''= \frac{\left(
   q'\right)^2}{q}-\frac{q'}{t}+\frac{2
   q^2}{t^2}-\frac{2}{t}
\end{equation} 
\begin{eqnarray}\label{eq:p3-d8}\Omega=\frac{dp \wedge dq}{q^2} - dH_{IIID_8}\wedge dt, \
H_{IIID_8} & = &-\frac{p^2}{q^2 t}-\frac{p}{q
   t}+\frac{1}{q}+\frac{q}{t}
\end{eqnarray}
The equation (\ref{eq:p3-d8-sys}) is equivalent to the following Hamiltonian system:
\begin{equation}
\left\{
\begin{array}{ccc}
\displaystyle{\frac{dq}{dt}} &=& 
\displaystyle{-q^2\frac{\partial H_{IIID_8}}{\partial p}}, \\
\displaystyle{\frac{dp}{dt}} & =& 
\displaystyle{ q^2
\frac{\partial H_{IIID_8}}{\partial q}}. 
\end{array} \right.
\end{equation}

\subsection{ Family $(0,-,2)$ and   Painlev\'e IV,  PIV($\tilde{E}_{6}$).}
\label{ss:p4-e6}

The family of connection with this data can be written as  
\begin{equation}\label{eq:p4-e6-linear}
\nabla_{\frac{d}{dz}}= \frac{d}{dz}+\frac{1}{z}A_0+A_1+zA_{2} 
= \frac{d}{dz} +\frac{1}{z}A,
\end{equation}
\[ A_0  =  \left(
\begin{array}{cl}
 \displaystyle{-q^2-\frac{tq}{2}+p} & \displaystyle{\frac{q^4+t q^3+\frac{t^2 q^2}{4}-2 p q^2- 
 t pq+p^2-\frac{\theta _0^2}{4}}{q }} \\
 -q  & \displaystyle{q^2+\frac{tq}{2}-p}
\end{array}
\right), \]
\[ A_1 =  \left(
\begin{array}{ll}
 \frac{t}{2} & 2 q^2+t q-2 p+\theta _{\infty }
   \\
 1 & -\frac{t}{2}
\end{array}
\right), 
 \ \
 A_2  =  \left( \begin{array}{cc}
 1 & 0 \\
 0 & -1
\end{array}
\right). \]
\begin{table}[h]
\begin{center}
\begin{tabular}{|c|c|c|} \hline
The singular points  $z$ & $0$ &  $\infty $ \\ \hline 
Katz invariant & $0$ &   $2$ \\ \hline 
   & & \\
generalized local exponents & $\pm \frac{\theta_0}{2}$ &  
$\pm ( z^{2} + \frac{t}{2} z+ \frac{\theta_{\infty}}{2} )$ \\
  & &\\   \hline 
\end{tabular}
\label{tab:p4-e6}
\end{center}
\end{table}
\[B=zB_1+B_2,
B_1  =   \left(
\begin{array}{cc}
 1/2 & 0 \\
 0 & -1/2
\end{array}
\right),  
B_2  =  
\left(
\begin{array}{cc}
 \frac{q}{2}+\frac{t}{4} & 
 q^2+\frac{tq}{2}-p+\frac{\theta _{\infty}}{2} \\
 \frac{1}{2} &   -\frac{q}{2}-\frac{t}{4}
\end{array}
\right)\]
\noindent
{\bf Painlev\'e IV, PIV($\tilde{E}_6$)}

\begin{equation} \label{eq:p4-e6-sys}
\left\{
\begin{array}{lcl}
\displaystyle{\frac{dq}{dt}} & = & p \\
  &  &  \\
\displaystyle{\frac{dp}{dt}} & = & 
\displaystyle{\frac{3 q^3}{2}+t q^2+\frac{1}{8} \left(t^2+4
   \theta _{\infty }+4\right) q+\frac{4
   p^2-\theta _0^2}{8 q}} \\
   & = & \displaystyle{\frac{12 q^4+8 t q^3+t^2 q^2+4 \theta _{\infty
   } q^2+4 q^2+4 p^2-\theta _0^2}{8 q}}
\end{array}
\right.
\end{equation}
The system (\ref{eq:p4-e6-sys}) is equivalent to the following 
second order equation. 
\begin{equation}\label{eq:p4-e6-sing}
q''= \frac{\left(q'\right)^2}{2 q}+\frac{3 q^3}{2}+t q^2+\frac{1}{8} \left(t^2+4
   \theta _{\infty }+4\right) q-\frac{ \theta_0^2}{8q}.
\end{equation} 
\begin{eqnarray}\label{eq:p4-e6}
H_{IVE_6}(p, q,t,  \theta) & = &-\frac{p^2}{2q}+\frac{q^3}{2}+
\frac{ t q^2}{2} +\frac{(t^2+4\theta _{\infty }+4)q}{8}+
\frac{\theta_0^2}{8q}.
\end{eqnarray}
Equation (\ref{eq:p4-e6-sys}) is equivalent to the 
following Hamiltonian system:
\begin{equation}
\left\{
\begin{array}{ccc}
\displaystyle{\frac{dq}{dt}} &=& 
\displaystyle{-q\frac{\partial H_{IVE_6}}{\partial p}}, \\
\displaystyle{\frac{dp}{dt}} & =& 
\displaystyle{q\frac{\partial H_{IVE_6}}{\partial q}}. 
\end{array} \right.
\end{equation}

\subsection{Family $(0,- ,3/2)$ and Painlev\'e II,  PIIFN($\tilde{E}_{7}$).}
\label{ss:p2-e7-fn}

PIIFN($\tilde{E}_{7}$) stands for the  Flaschka--Newell equation \cite{FN} which is equivalent to the Painlev\'e equation PII.  A family of connection with this data is 
\begin{equation}\label{eq:p2fn-e7-linear}
\nabla_{\frac{d}{dz}}= \frac{d}{dz}+\frac{1}{z}A_0+A_1+zA_{2} 
= \frac{d}{dz} +\frac{1}{z}A,\  A=\left(
\begin{array}{cc}
 p & \frac{p^2+q z^2-\theta _0^2+q^2 z-2 q t z}{q}
   \\
 z-q & -p
\end{array}
\right). 
\end{equation}
\[
A_0  =  \left(\begin{array}{cc} p & \frac{p^2-\frac{\theta _0^2}{4}}{q} \\-q & -p
\end{array} \right), \ \
A_1  =  \left(\begin{array}{cc} 0 & q+t \\ 1 & 0\end{array}\right), \ \
 A_2 =  \left(\begin{array}{cc} 0 & 1 \\ 0 & 0\end{array} \right).\] 
\begin{table}[h]
\begin{center}
\begin{tabular}{|c|c|c|} \hline
The singular points  $z$ & $0$ &  $\infty$ \\ \hline 
Katz invariant & $0$ &   $3/2$ \\ \hline 
   & & \\
generalized local exponents & $\pm  \frac{\theta_0}{2} $ &  
$ \pm ( z^{3/2} + \frac{t}{2} z^{1/2} )$ \\
  & &\\   \hline 
\end{tabular}
\label{tab:p2-e7-fn}
\end{center}
\end{table}
\[ B := B_0 + z B_1,\  B_1 =  \left(\begin{array}{cc} 0 & 1 \\ 0 & 0 \end{array} \right) \ \
B_0 = \left(\begin{array}{cc} 0 & 2q+t \\ 1 & 0\end{array} \right).\]
\normalsize
\noindent
{\bf Painlev\'e II, PIIFN ($\tilde{E}_{7}$)}
\begin{equation} \label{eq:p2-e7-fn-sys}
\left\{
\begin{array}{lcl}
\displaystyle{\frac{dq}{dt}} & = & \displaystyle{2p} \\
\displaystyle{\frac{dp}{dt}} & = & 
\displaystyle{ \frac{\left(2 q^3+  t q^2+p^2-\frac{\theta_0^2}{4}\right)}{q}}
= 2q^2 + t q + \frac{p^2-\frac{\theta_0^2}{4}}{q}
\end{array}
\right.
\end{equation}

The system (\ref{eq:p2-e7-fn-sys}) is equivalent to the following 
second order equation. 
\begin{equation}\label{eq:p2-e7-fn-sing}
q''= \frac{\left(q'\right)^2}{2q}+ 4 q^2+2 t q-\frac{
   \theta _0^2}{2q}
\end{equation} 
\[ 
\Omega =\frac{dp \wedge dq}{q} - dH_{IIFNE_7}\wedge dt \quad \mbox{ with }
\] 
\begin{equation}\label{eq:p2-e7-fn-ham}
H_{IIFNTE_7} = -\frac{p^2-\frac{\theta_0^2}{4}}{q}+ q^2 +t q.
\end{equation}
Equation (\ref{eq:p2-e7-fn-sys}) is equivalent to the following Hamiltonian system:
\begin{equation}
\left\{
\begin{array}{ccc}
\displaystyle{\frac{dq}{dt}} &=& 
\displaystyle{-q\frac{\partial H_{IIFNE_7}}{\partial p}}, \\
\displaystyle{\frac{dp}{dt}} & =& 
\displaystyle{ q \frac{\partial H_{IIFNE_7}}{\partial q}}. 
\end{array} \right.
\end{equation}

\subsection{\bf Family $(-, -, 3)$ and  Painlev\'e II,  PII($ \tilde{E}_{7} $).}
\label{ss:p2-e7}
We present the data and the results of the computation.
\begin{equation}\label{eq:p2-e7-linear} 
\nabla_{\frac{d}{dz}}= \frac{d}{dz}+A_0+zA_1+z^2A_{2} 
= \frac{d}{dz} +A \mbox{ where }
\end{equation} 
\[
A_0  =  \left(
\begin{array}{cc}
 p-q^2 & 2 q^3-2 p q+t q+\theta _{\infty } \\
 -q & q^2-p
\end{array}
\right), 
\ 
A_1  =  \left(
\begin{array}{cc}
 0 & 2 q^2-2 p+t \\
 1 & 0
\end{array}
\right),  A_2= \]
\[  \left(
\begin{array}{cc}
 1 & 0 \\
 0 & -1
\end{array}
\right), 
A=\left(
\begin{array}{ll}
p+z^2 -q^2& (q+1) t-2(p-q^2) (z+q)
   z+\theta _{\infty } \\
 z-q &-p -z^2+q^2
\end{array}
\right). \]
\begin{table}[h]
\begin{center}
\begin{tabular}{|c|c|} \hline
The singular points  $z$ &  $\infty  $ \\ \hline 
Katz invariant &    $3$ \\ \hline 
   & \\
generalized local exponents &  
$\pm ( z^{3} + \frac{t}{2} z+ \frac{\theta_{\infty}}{2} )$ \\  &\\   \hline \end{tabular}
\label{tab:p2-e7}
\end{center}
\end{table}
\[ B := B_0 + zB_1,\  B_0  = \left(
\begin{array}{ll}
 \frac{q}{2} &  q^2-
   p+\frac{t}{2} \\
 \frac{1}{2} & -\frac{q}{2}
\end{array}
\right), \
B_1  =  \left( \begin{array}{cc}
 1/2 & 0 \\
 0 & -1/2
\end{array}
\right).  \]
\noindent
{\bf Painlev\'e II, PII($\tilde{E}_7$)}
\begin{equation} \label{eq:p2-e7-sys}
\left\{
\begin{array}{lcl}
\displaystyle{\frac{dq}{dt}} & = & p \\
\displaystyle{\frac{dp}{dt}} & = & 2q^3+ t q + \frac{\theta_{\infty} + 1}{2} 
\end{array}
                        \right. 
\end{equation}. 
\begin{equation}\label{eq:p2-e7-sing}
q''=  2 q^3 +  qt + \frac{ \theta_{\infty} + 1}{2}
\end{equation} 
\[
\Omega =dp \wedge dq- dH_{IIE_7}\wedge dt, \mbox{ where }
\]
\begin{eqnarray}\label{eq:p2-e7} 
H_{IIE_7}(p, q,t,  \theta) & = &\frac{1}{2}(-p^2+q^4+ t q^2+ (\theta_{\infty }+1) q)
\end{eqnarray}
Equation (\ref{eq:p2-e7-sys}) is equivalent to the 
following Hamiltonian system:
\begin{equation}
\left\{
\begin{array}{ccc}
\displaystyle{\frac{dq}{dt}} &=& 
\displaystyle{-\frac{\partial H_{IIE_7}}{\partial p}}, \\
\displaystyle{\frac{dp}{dt}} & =& 
\displaystyle{  \frac{\partial H_{IIE_7}}{\partial q}}. 
\end{array} \right.
\end{equation}

\subsection{\bf Family $(-, -, 5/2)$ and  Painlev\'e I,  PI($\tilde{E}_{8}$).}
\label{ss:p1-e8}

The family of connection with the data can be written as  
\begin{table}[h]
\begin{center}
\begin{tabular}{|c|c|} \hline
The singular points  $z$ &  $\infty $ \\ \hline 
Katz invariant &    $\frac{5}{2}$ \\ \hline 
   & \\
generalized local exponents &  
$ \pm ( z^{5/2} + \frac{t}{2} z^{1/2} )$ \\
  &\\   \hline 
\end{tabular}
\label{tab:p1-e8}
\end{center}
\end{table}

\begin{equation}\label{eq:p1-e8-linear}
\nabla_{\frac{d}{dz}}= \frac{d}{dz}+A_0+zA_1+z^2A_{2} 
= \frac{d}{dz} +A, \mbox{ where }
\end{equation} 
\begin{equation}
A_0 =  \left(
\begin{array}{cc}
 p & q^2+ t \\
 -q & -p
\end{array}
\right), \quad
A_1  = \left(
\begin{array}{cc}
 0 & q \\
 1 & 0
\end{array}
\right)
 \quad 
 A_2 =  \left(
\begin{array}{cc}
 0 & 1 \\
 0 & 0
\end{array}
\right), 
\end{equation}
\[
A=\left(
\begin{array}{cc}
 p & q^2+z q+z^2+t \\
 z-q & -p
\end{array}
\right). 
\]

\par
\noindent
\normalsize
\[ B := B_0 + zB_1, \quad   B_0  = \left(\begin{array}{cc}
 0 & 2q \\
 1 & 0
\end{array}
\right),  B_1 = \left(
\begin{array}{cc}
 0 & 1 \\
 0 & 0
\end{array}
\right). \]
\noindent
{\bf Painlev\'e I, PI($\tilde{E}_8$)}
\begin{equation} \label{eq:p1-e8-sys}
\left\{
\begin{array}{lcl}
\displaystyle{\frac{dq}{dt}} & = & \displaystyle{2p} \\
  &  &  \\
\displaystyle{\frac{dp}{dt}} & = & 3q^2+t
\end{array}
\right.
\end{equation}
The system (\ref{eq:p1-e8-sys}) is equivalent to the following 
second order equation. 
\begin{equation}\label{eq:p1-e8-sing}
q''=  6 q^2 + 2t 
\end{equation} 
\begin{eqnarray}\label{eq:p2-e7} \Omega =dp \wedge dq- dH_{IE_8}\wedge dt, \
H_{IE_8}(p, q,t,  \theta) & = &-p^2+q^3+ t q
\end{eqnarray}
Equation (\ref{eq:p1-e8-sys}) is equivalent to the 
following Hamiltonian system:
\begin{equation}
\left\{
\begin{array}{ccc}
\displaystyle{\frac{dq}{dt}} &=& 
\displaystyle{-\frac{\partial H_{IE_8}}{\partial p}}, \\
\displaystyle{\frac{dp}{dt}} & =& 
\displaystyle{  \frac{\partial H_{IE_8}}{\partial q}}. 
\end{array} \right.
\end{equation}

\vspace{0.5cm}

\begin{center}
{\bf Acknowledgments}
\end{center}

The second author would like to thank  Marius van der Put and  Jaap Top for their hospitality during his 
visits of the department of Mathematics of the university of Groningen.

\end{document}